 \documentclass[11pt]{amsart}
 
 \usepackage{color}

 \usepackage{stix} 
 \usepackage[cal=boondoxo]{mathalfa} 

 \usepackage[all]{xy}
 \usepackage{tikz}
\usetikzlibrary{arrows, calc,  backgrounds, shapes.geometric, positioning, decorations.markings}

\usepackage{amsmath,amsthm,amscd,graphicx}

 \usepackage{hyperref}
\usepackage{enumitem}

\newtheorem{thm}{Theorem}[subsection]
\newtheorem{thmA}{Theorem}[section]
\newtheorem*{thm*}{Theorem}
\newtheorem{lemma}[thm]{Lemma}
\newtheorem*{lemma*}{Lemma}
\newtheorem{prop}[thm]{Proposition}
\newtheorem{propA}[thmA]{Proposition}
\newtheorem*{prop*}{Proposition}
\newtheorem{corr}[thm]{Corollary}
\newtheorem{corrA}[thmA]{Corollary}
\newtheorem*{corr*}{Corollary}

\theoremstyle{definition}
\newtheorem{dfn}[thm]{Definition}
\newtheorem{dfnA}[thmA]{Definition}
\newtheorem{exmple}[thm]{Example}
\newtheorem{exmples}[thm]{Examples}

\theoremstyle{remark}

\newtheorem*{rmq}{\textit{Remark}}
\newtheorem{rmk}[thm]{\textit{Remark}}

\newcommand\mbold[1]{{\mathbb{#1}}}
\newcommand{\bC}{{\mbold{C}}}
\newcommand{\bQ}{{\mbold{Q}}}
\newcommand{\bR}{{\mbold{R}}}
\newcommand{\bZ}{{\mbold{Z}}}

\newcommand{\bP}{{\mbold{P}}}

\newcommand\germ[1]{{\mathfrak{#1}}}
\def\gn{\germ n}
\def\gq{\germ q}
\newcommand{\gh}{\germ h}%
\newcommand{\ga}{\germ a}%
\newcommand{\geg}{\germ g}%

\newcommand{\cA}{{\mathcal A}}

\newcommand\cC{{\mathcal C}}

\newcommand\cF{{\mathcal  F}}

\newcommand\cH{{\mathcal H}}

\newcommand\cM{{\mathcal M}}
\newcommand\cQ{{\mathcal Q}}
\newcommand\cI{{\mathcal I}}

\newcommand\cE{{\mathcal E}}

\newcommand\cU{{\mathcal U}}
\newcommand\cO{{\mathcal O}}

\newcommand\cY{{\mathcal Y}}
\newcommand\cZ{{\mathcal Z}} 

\def\mapright#1{\mathop{\vbox{\ialign{
                                ##\crcr
    ${\scriptstyle\hfil\;\;#1\;\;\hfil}$\crcr
 \noalign{\kern2pt\nointerlineskip}
    \rightarrowfill\crcr}}\;}}
\def\into{\hookrightarrow}

\renewcommand\setminus{-}

\def\comp{\raise1pt\hbox{{$\scriptscriptstyle\circ$}}}
\def\del{\partial}
\def\e{\mathop{ \rm e}\nolimits}
\def\half{\frac{1}{2}}
\def\higgs{\text{\rm Higgs}}
\def\hor{\text{\rm hor}}
\def\id{\text{\rm id}} 
\def\ii{{\mathbf i}}
\def\im{\text{\rm Im}}
 
\def\mhs{\text{\sf MHS}} 

\def\prim{\text{\rm prim}}
\def\rank{\text{\rm rank}}

\def\tr{\mathop{\rm Tr}\nolimits}

\newcommand\Tr{{}^{\mathsf{T}}\kern-0.9pt} 
\def\var{\text{\rm var}}
\def\vmhs{\text{\sf VMHS}}  
\def\lset{\{}  
\def\rset{ \}}  

\def\set#1{\lset#1\rset} 
\def\sett#1#2{\lset #1 \mid  #2 \rset}  
\newcommand\coker{\mathop{\rm Coker}\nolimits}
\renewcommand\ker{\mathop{\rm Ker}\nolimits}
\def\aut#1{\operatorname{Aut}(#1)}

\def\endo{\mathop{\rm End}\nolimits}
\def\qendo#1{\geg(#1)}
\def\ext{\mathop{\rm Ext}\nolimits} 
 
\def\gr{\text{\rm Gr}}
\renewcommand\hom{\mathop{\rm Hom}\nolimits}

\newcommand\re{\operatorname{Re}}

\newcommand\Ad[1]{\operatorname{Ad}{#1}\,}
\newcommand\ad[1]{\operatorname{ad}{(#1})}

\newcommand\gl[1]{\operatorname{GL}({#1})}
\newcommand\gll[1]{\operatorname{gl}({#1})}
 \newcommand\lie[1]{\operatorname{Lie}(#1)}
 
 \newcommand\slgr[1]{\operatorname{SL}_{#1}}
 \newcommand\sll[1]{\operatorname{sl}_{#1}}
 
 \newcommand\smpl[2]{\operatorname{Sp}_{#2}({#1})}

 \def\warn#1{\textcolor{black}{#1}}

\begin{document}
\title[Deformations and Rigidity for mixed period maps]{Holomorphic Bisectional Curvature and Applications to Deformations and Rigidity for Variations of Mixed  Hodge Structures 
}

\author{Gregory Pearlstein, \\
Chris  Peters }

\begin{abstract}

In this article, we prove a rigidity criterion for period maps of
admissible variations of graded-polarizable mixed Hodge structure, and
establish rigidity in a number of cases, including families of
quasi-projective curves, projective curves with ordinary double
points, the complement of the canonical curve in families of
Kynev--Todorov surfaces, period maps attached to the fundamental
groups of smooth varieties and normal functions. 
\end{abstract}

\maketitle
  
\section{Introduction}
\label{sec:intro}

\subsection{Historical background}
\label{ssec:history}

The rigidity concept the title refers to,  concerns a Hodge theoretic variant of a rigidity property that S.\ Arakelov \cite{ara} discovered.
He  showed that one cannot deform   families $\set{C_s}_{s\in \bar S}$ of 
  curves of genus $g\ge 2$ parametrized by a  smooth projective curve  $\bar S$ with  varying moduli,  
   keeping $\bar S$ fixed as well as  the set, say  $\Sigma$, over which singular fibers occur. In terms of $\cM_g$, the  moduli space
  of curves of genus $g$, this result states that if the moduli map 
  $\mu: S =\bar S\setminus \Sigma  \to  \cM_g$ is not constant, it is rigid, keeping source and target fixed.\footnote{For an approach using Teichm\"uller theory see \cite{imashi}.} In the remainder of this introduction \textbf{\emph{we shall only consider deformations of maps keeping source and target fixed}}.
  
   The cohomology groups  $H^1(C_s,\bZ)$ admit  a canonical polarizable weight one Hodge structure. These are classified by a period domain,
   in this case the generalized Siegel upper half-space  $\gh_g$. Since the group of  integral automorphisms preserving the polarization is  the 
   symplectic group $ \smpl g \bZ $, the period map in this case is a holomorphic map
    $F: S   \to \cA_g:=\smpl {  g }{\bZ}\backslash \gh_g$ which factors through the   morphism $\cM_g \to \cA_g$. The latter morphism  is an embedding
    (this is Torelli's theorem).    
    
    It might be the
   case that,  although $\mu$ is rigid   keeping $(\bar S,\Sigma)$ fixed, this need not be the case for $F$.  
    Geometrically interpreted, polarized weight one Hodge structures are  polarized
   Abelian varieties and G.\ Faltings, in \cite{arafal} investigated the analog of Arakelov rigidity in this situation.  Let us recall
   his result in Hodge theoretic terms.     The period domain $\gh_g$ classifies (polarized) weight $1$ Hodge structure on a  free $\bZ$-module $H$.
   Such a Hodge  structure induces Hodge structures of weight $0$ on  $\endo (H)$ as well on its  subspace $\endo(H,Q)$ of the $Q$-endo\-morphisms, that is
  those   $u\in \endo H$ for which $Q(ux,y)+Q(x,uy)=0$ for all $x,y\in H$. 
  By means of the  period map $F:S\to \cA_g$,  Hodge structures $F(s)$ of weight  one  
  are put on $H$.  The group $\Gamma$ acts on $H$ as well as $\endo(H,Q)$. In particular, its commutant
  \[
  \endo^\Gamma(H_\bC,Q):= \sett{u\in \endo(H_\bC,Q)}{ \gamma\comp u \comp \gamma^{-1}=u}
  \]
  inherits   natural Hodge structures as well. By W.\ Schmid's result \cite[Corollary 7.23]{schmid},  these Hodge structures are independent of $s$. In technical terms,  
  the period map defines a local system on $S$ carrying  a variation of Hodge structure inducing one on the endomorphism bundles and  
  the Hodge decomposition extends as a 
    flat  decomposition,  hence  is independent of $s$. See Section~\ref{ssec:purepd}.   G.\ Faltings' result is as follows:
   
   \begin{thm*}[\protect{\cite[Theorem 2]{arafal}}] The space of infinitesimal deformations of  a period map  $F: S   \to \cA_g$  over a curve $S$
    can be canonically identified with  the direct summand    of 
    $\endo^\Gamma(H_\bC,Q)$ of  Hodge type $(-1,1)$. Consequently, if $F$ is not constant, $F$ is rigid if and only if $\endo^\Gamma  (H,Q)$
    is pure of type $(0,0)$.
   \end{thm*}
   
   Faltings gave an example with $g=8$  for which $\endo^{\Gamma}(H,Q)^{-1,1}\not=0$ and so this gives 
   a non-rigid (non-isotrivial) family of $8$-dimensional Abelian varieties.  M.-H Saito \cite{Sa} made a systematic study  and classified these in any dimension.
    
    The Hodge-theoretic rigidity question for higher weight and over any quasi-projective smooth base was first consider by the second author in \cite{hodge4}
    and it turns out that G.\ Faltings' result is in essence valid for all weights. There are a couple of differences. 
    Of course,  since $S $  is allowed to be higher-dimensional, 
    one has to impose the condition that the period map is generically an immersion 
    instead of being non-constant.  Secondly, on a more fundamental level, one  should  incorporate 
    ``Griffiths' transversality'' (cf. \cite{periods})
    an infinitesimal  property of  variations of geometric origin which is automatic for weight $1$ but gives a constraint for most
    types of higher weight variations.    
    Geometrically this condition means that tangents to 
      the image of the period map belong to the so-called horizontal tangent bundle. This is encapsulated in the statement that
      period maps are horizontal.
      It is natural to demand that deformations  preserve this property. The result from loc. cit. indeed takes this into account:
      
      \begin{thm*}[\protect{\cite[Theorem 3.4]{hodge4}}] Let $S$ be smooth and quasi-projective and $F:S\to \Gamma \backslash D$ a period map.
 The   space of infinitesimal deformations of $F$  remaining  horizontal  can be canonically identified with
  $\endo^{\Gamma}(H_\bC,Q)^{-1,1}$. 
 \end{thm*}
     
The proof of this result is reviewed in Section~\ref{sec:pure}.
      
   \subsection{Main results on deformations of mixed period maps}  
   \label{ssec:DeformPure}
   
  For the purpose of this introduction,  a free $\bZ$-module $H$  is said to carry a mixed Hodge structure, if  $H_\bQ=H\otimes \bQ$ carries an increasing finite filtration
   $W$, the weight filtration  and $H_\bC=H\otimes \bC$ carries a decreasing filtration $F$, the Hodge filtration which induces
   a pure Hodge structure of weight $k$ on $\gr^W_kH$.  If, moreover, each of those are polarized by $Q_k$, we 
   write $Q$ for the collection of the $Q_k$ and say that $(H,W,F,Q)$ is a graded polarized mixed Hodge structure.
   
   Motivated by geometry, for classifying purposes we  keep the weight filtration and the polarization fixed.
   So on a fixed triple $(H,W,Q)$ we   allow only the Hodge filtration to vary. 
   The associated  period domains and period maps
   have been studied in \cite{usui,higgs,dmj,sl2anddeg}.
   \par
   There are several important differences   with  the pure  situation. First of all, $H_\bC$ does not have a "mixed" Hodge decomposition,
   but instead, a canonical decomposition, introduced by Deligne \cite{tdh}, the Deligne-decomposition $H_\bC= \oplus I^{p,q}$ where
   $I^{p,q}$  has the same dimension as the $(p,q)$-component of the Hodge structure on $\gr^W_{p+q}$, but it is no longer
   the case that $I^{q,p}$ is the complex conjugate of $I^{p,q}$.

   Secondly, although, as  in  the pure case the period domain $D$  is homogeneous under a   Lie group $G$,
  say $D=G/G^F$, the isotropy group $G^F$need not be  compact. Moreover, the group $G$ has in general no real
  structure: it generally strictly contains   $G_\bR$, the automorphism group of $(H_\bR, W,Q)$.
  
  As in the pure case the  polarization induces a Hodge metric on the
  tangent bundle to $D$, which is equivariant with respect to $G_{\bR}$,
  but not the full group $G$. Period maps are holomorphic, there is a notion of Griffiths'  transversality 
   and a concept of horizontal tangent bundle.  Period maps $F$ have tangents in the latter bundle.  
 As before,  through the period map one gets  mixed Hodge structures on $H$  depending on $s$, i.e.,  the holomorphic vector bundle $\cH$ on $S $ with fibers $\simeq  H_\bC$
 receives  a variation of mixed Hodge structure (VMHS). The induced 
 varying mixed Hodge structures on  the Lie algebra 
 \[
 \geg_\bR = \endo^W(H_\bR,Q)
 \]
 of endomorphisms which preserve $W$ and act by infinitesimal isometries
 on $Gr^W$ defines a VMHS on the holomorphic vector bundle  
 \[
 \geg(\cH) =\endo^W(\cH,Q)
 \]
  over $S$ and, again by  \cite[Corollary 7.23]{schmid},     the   Deligne decomposition on the space of global $\Gamma$-equivariant sections of $\geg(\cH)$ is 
  a flat decomposition, that is, ``constant in  $s\in S$''.
The horizontality constraint implies that we restrict  our attention to
\[
 \cU^{-1} \geg (\cH)= \bigoplus_{q\leq 1} \geg^{-1,q} (\cH),
 \]
 the \textbf{\emph{horizontal endomorphism bundle}}.     The main result can now be stated as follows:
  
  \begin{thm*}[=Theorem~\ref{thm:Main}] 
  Let $S$ be quasi-projective and $F:S\to \Gamma\backslash D$ a horizontal holomorphic map to a mixed domain $D$ parametrizing
mixed Hodge structures on    $(H,W,Q) $  such that the corresponding   VMHS  is admissible.
\warn{Suppose that     $v\in \cU^{-1}  \geg (\cH ) $ is  Hodge-harmonic, that
moreover, $v$  is equi\-variant with respect to the monodromy group $\Gamma$
  and that  the Hodge norm $\|v\|$ is bounded near infinity.}
 
Then     infinitesimal deformations of $F$ that stay horizontal correspond one-to-one to 
 $\Gamma$-equivariant horizontal endomorphisms of $ \geg (\cH)  $.
 The space of such deformations is smooth at $F$. 
 \end{thm*}
 
 The statement requires some explanation.   Let $v(s)$ be a section of the bundle \warn{$\cU^{-1} \geg (\cH)$ on $S$ of 
 the horizontal endomorphisms 
of $\geg(\cH)$}.
   In the pure case, as shown in the proof of \cite[Theorem 3.2]{hodge4},  negativity of the bisectional curvature in horizontal directions
implies that its  Hodge norm gives rise to a plurisubharmonic function $s\mapsto \|v(s)\|$. 
One can do with a slightly weaker condition
which is more suitable  in the mixed situation. This weaker condition is the
plurisubharmonicity of an   endomorphism $v$ of \warn{$\geg(\cH)$}  and will be explained  in  Section~\ref{ssec:Curv}. In the pure case
it  indeed implies plurisubharmonicity of  the  Hodge norm $\|v\|$, and we show that this is also true
  for  several   types  of mixed Hodge structures of geometric interest.
 As is well known (see for example  \cite{psh}), bounded  plurisubharmonic functions on a quasi-projective manifold are constant. 
To make use of this, it suffices to show  that  $\|v(s)\|$ is bounded   near infinity whenever  $v$ is preserved by the local monodromy at infinity. 
 This is indeed the case for pure Hodge structures as follows from  W. Schmid's norm estimates in  \cite{schmid}.
  Unfortunately, \warn{as Section~\ref{subsect:higher-normal-functions}   shows,} the desired estimates do not hold for mixed variations in general, not even for admissible variations. 
  However, for  several  cases of geometric  interest,  boundedness still holds as shown in  \warn{the remainder} of Section~\ref{sec:norm-estimates}. 
  
\begin{rmk} Although we only consider period maps to "classical" mixed period domains, the same methods apply to 
variations with extra structure corresponding to period maps to mixed \textbf{\emph{Mumford--Tate domains}}. 
To explain this, first of all, the differential geometric input based on curvature calculations only uses Lie-theoretic calculations
involving  the mixed Hodge metric and the Deligne types and these calculations remain the same. Indeed, a Mumford--Tate domain is
a homogeneous space of the form $M/M^F$ where $M$ is a subgroup of a group $G$ acting transitively on some mixed domain $D$ and $M^F=M\cap G^F$ so that the Hodge metric is the one from $D$ restricted to $M/M^F$,  and the Deligne types are  the same as the
ones for  the mixed Hodge structure of the Lie algebra of $G$.  See also \cite[Remarks 1.1, 2.4]{pdmixed}. 

Secondly, the calculations for
boundedness of the mixed Hodge metric are based on the $\slgr 2$-orbit  theorem. Its proof uses Lie theory within a given group
and one can show that these calculations stay within $M\subset G$.
See \cite[Section 4]{kerr2017polarized} where the pure case is treated. For the mixed situation the arguments are the same. 
 \end{rmk}

\subsection{Boundedness results}
Although for our purposes we only need a one-variable boundedness result, there is one situation where we prove  a multivariable version
which may be of independent interest:
\begin{thm*}[=\ref{hodge-tate-estimate}] A flat section of an admissible Hodge--Tate
variation $\cH$ with unipotent monodromy   has bounded
Hodge norm   with respect to the mixed Hodge metric. Likewise, for  a flat sections of
$\endo \cH $.
\end{thm*}
Furthermore, we prove the following $1$-variable results, similarly of independent interest:
\begin{thm*}[=Theorems~\ref{unpotent-estimate}, \ref{thm:weight-minus-2-case}, Corollaries~\ref{cor:IandII},\ref{corr:height-pairing}]  Let $\cH$ be a $1$-variable  admissible variation with unipotent monodromy  of one of the following types:
  \\
  1. of unipotent type;\\
  2.  of type $(\text{\rm I})$ or $(\text{\rm II})$\footnote{We recall that from~\cite{sl2anddeg} that a variation is type $(\text{\rm I})$ if there exists an integer $k$ such that the Hodge numbers
$h^{p,q}$  are zero unless $p+q=k$, $k-1$  (i.e. $\gr^W$ has
exactly two non-zero weight graded-quotients which are adjacent) and it 
  is type $(\text{\rm II})$ if there is an integer $k$ such that
$h^{p,q}=0$ unless $(p,q)=(k,k)$, $(k-1,k-1)$ or $p+q=2k-1$ and $h^{k,k}$,
$h^{k-1,k-1}$ are non-zero.};\\
  3.  a variation whose  sole weight graded quotients are $\gr^W_0\cong\bR(0)$ and  $\gr^W_{-2}$;\\
  4.  a variation whose  sole  weight graded quotients  are $\gr^W_0\cong\bZ(0)$,
$\gr^W_{-2}$ and $ \gr^W_{-4}\cong\mathbb Z(0)$.  
\medskip

Then a flat section of $\cH$ or of $\geg(\cH)$   has bounded mixed Hodge norm.
\end{thm*}

We also show that  for variations whose  sole  weight graded quotients  are $\gr^W_0=\bZ$ and $\gr^W_{-k}$ for $k>2$, the norm estimates required to
obtain rigidity need not hold. See Lemma~\ref{lem:UnBounded}.

\subsection{Geometric applications}
 The first application concerns     families of  \textbf{\emph{quasi-projective smooth curves}} of genus $g$.  In  Example~\ref{exmpls:OpenRig} 
 we show that if the monodromy acts irreducibly on cohomology, the family 
is rigid in the  $(-1,0)$-directions provided the curves can be completed by adding   $<2g$ points. \warn{The mixed Hodge structures on    projective curves
with $k$ double points are in a certain sense dual to the ones on a quasi-projective curves which can be compactified to a smooth projective curves
upon adding $k$ points. Indeed, there is a dual result}  for families of \textbf{\emph{curves  with   $>2g$ double points}} (see Example~\ref{ex:CurvWithDPS}).
Perhaps worth mentioning here  is the use of the rather recent concept of
pure \textbf{\emph{variations having maximal Higgs field}}, a concept introduced by E. Viehweg and exploited in \cite{VZ}.
For instance in Proposition~\ref{prop:maxHiggsIsRigid} we state and prove that  having a weight one maximal Higgs field implies rigidity.  Hence, for  the preceding examples maximal Higgs leads to period maps rigid in all
horizontal directions.

Next we mention families  of   \textbf{\emph{Kynev and Todorov surfaces}}. V. Kynev \cite{kynev} has given a construction of surfaces of general type
with invariants $h^{1,0}=0$, $h^{2,0}=1$, $K^2=1$ that violate the infinitesimal Torelli theorem. Other counterexamples to  infinitesimal  Torelli were  given
by  A. Todorov \cite{todorov}. His surfaces have the same  invariants $h^{1,0}=0, \, h^{2,0}=1$,  but  $2\le K^2\le 8$.
The period domain of both types of surfaces resemble that of a K3 surface. Like a K3 surface, there is an up to scaling  unique holomorphic $2$-form
but here it  vanishes along the  canonical curve which is smooth for a generic such surface.
Removing this  curve gives an  open surface intrinsically 
 associated to a  Kynev or Todorov surface. 
Its  cohomology then provides an example of  a mixed Hodge structure. The Todorov surfaces with $K^2=2,\dots,8$ generalize  Kynev surfaces that were previously also 
investigated in detail  by F. Catanese \cite{catanese} and A. Todorov~\cite{todorov2} and so we shall call these  \textbf{\emph{CKT-surfaces}}. We show (cf. Proposition~\ref{prop:CKT-surfs}) that  a modular  family of open CKT-surfaces  or of Todorov surfaces
 is rigid,  as is any  sufficiently generic subfamily.

We shall  also consider deformations of certain unipotent variations. Firstly  \textbf{\emph{Hodge--Tate variations }} (Section~\ref{ssec:PSH}, Example (2))   and, secondly,    
\textbf{\emph{variations  associated 
to the fundamental group}} of an algebraic manifold (Section~\ref{ssec:PSH}.  Example (3)).   
For the latter, an explicit rigidity  criterion  is stated later  as  Proposition~\ref{prop:OnMHSonFundGrps}.
 It involves  the geometry  of the exterior algebra of the $1$- and $2$-forms of $S$. 

Deformations of  other types of algebraic families  are investigated  in Example~\ref{exmpl:NonRigid}  and, more elaborately, in Section~\ref{sec:Exampl}. These   include \textbf{\emph{normal functions}}, certain  \textbf{\emph{higher normal functions}} and  \textbf{\emph{biextensions}} coming from higher Chow groups.  

%
%
%

    \subsection{Structure of the paper}
   \label{ssec:Struct}
   
   In Section~\ref{sec:pure} we recall in detail the pure case and the proof of  the main result from \cite{hodge4}.
   The proof   presented  here differs slightly from the one given in  loc. cit.   since we want to highlight where problems arise    for
   the  mixed  case.    
     Further basic developments have been taken place since the publication of \cite{hodge4} which
    we  recall    in Section~\ref{ssec:ExRigPure}.  Several of these newer examples  
    serve as  building blocks  in the   mixed  situation to which we turn in later sections.
     
 In Section~\ref{sec:permaps} we recall some  basic material concerning mixed period maps.

One of the main ingredients in the proof of  our results is  the curvature calculation from \cite{pdmixed}.
We explained in loc. cit. that, unlike in the pure case, the holomorphic sectional curvature is not in general $\le 0$ in
horizontal directions and so this  is a fortiori  true for the holomorphic bisectional curvature. The latter  plays 
a central role in the proof of \cite{hodge4} and our  original strategy was to list classes of types of mixed Hodge structure
for which this is also true. In Section~\ref{sec:DG} we come back to the calculations of \cite{pdmixed} and show that
instead of focusing on bisectional curvature, it is better to  use   a new  property, that of  plurisubharmonicity of certain  global  endomorphisms of the Hodge bundle.
 \par
 The second main ingredient, the norm estimates for the Hodge metric
are given in  Section~\ref{sec:norm-estimates}. The techniques employed in this section are of an entirely different, mainly Lie-theoretic  nature. 
\par
The proper topic of this paper,  deformation theory in the mixed situation,  is treated in Section~\ref{sec:DefoMPS} where 
we prove the main theorem,
Theorem~\ref{thm:Main} and give criteria for rigidity.
The main technical result, Proposition~\ref{prop:harmendo},   leads to the geometric examples which are treated in detail in Section~\ref{sec:Exampl}.

In Appendix~\ref{sec:admis} the notion of admissibility is reviewed,  and  in Appendix~\ref{sec:ProperDisc} we show that, like in the pure case, the monodromy action
on the period domain  coming from a  mixed variation with an integral  structure is properly discontinuous so that
the quotient has the structure of an analytic  space. 

\section{The pure case} \label{sec:pure}

\subsection{Basics on period domains and period maps} \label{ssec:purepd}

Recall that a \textbf{\emph{period domain}}  parametrizes polarized Hodge
structures of weight $k$ on a  finite dimensional real vector space $H_{\bR}$  
with given Hodge numbers $\{h^{p,q}\}$, polarized  by a non-degenerate
bilinear form $Q$ of parity $(-1)^k$.   Such a domain $D$ is homogeneous under the real Lie group $G_{\mathbb R} \subset \gl{H_\bR}$ of automorphisms of
the polarization $Q$. The  isotropy groups $G_{\mathbb R}^F$, $F\in D$ are
compact. The domain $D$ is an open set  in the  compact dual  $\check D$ upon 
which the complexification $G_\bC$ of $G_{\mathbb R}$  acts transitively:
\[
G_{\mathbb R}/G_{\mathbb R}^F= D \subset \check D= G_\bC/G^F_\bC.
\]

The Hodge structure on $H_\bR$ given by $F$ induces a Hodge
structure on the Lie algebra of
$G_{\mathbb R}$ as a sub-Hodge structure of $\endo H_{\bR}$. It has
weight zero with Hodge decomposition
$
        \geg_\bC  =\bigoplus_p \geg^{p,-p}
$
where $\geg^{p,-p}$ consists of those endomorphisms that send $H^{s,t}$ to $H^{s+p,t-p}$.

A point   $F \in \check D$  can be considered as  a   filtration on $H_\bC$. Then
$F^0   \geg_\bC $ is the Lie algebra of the stabilizer of $F$ in $G_\bC$. Hence 
the   tangent space $T_F\check D$ of $\check D$ at $F$  is   isomorphic to  $  \geg_\bC / F^0   \geg_\bC $.  Accordingly, since   $F^p   \geg_\bC  = \bigoplus_{a\geq p}\, \geg^{a,-a} $, it follows that
  \begin{align}
   \label{eqn:TangSpasHS}
   T_FD = \geg_\bC/F^0\geg_\bC \simeq \bigoplus_{a> 0}\, \geg^{- a, a}.
  \end{align} 
  
Every Hodge structure    $F\in D$ defines the \textbf{\emph{Hodge metric}} on $H_\bC$   which is given by
\begin{equation}   \label{eqn:hmetric}
h_F(x,y) := Q(C_Fx,\bar y), x,y\in H_\bC,
\end{equation}
where $C_F|H^{p,q}= \ii^{p-q}$ is the Weil-operator.   The Hodge metric is a hermitian metric relative to which the Hodge decomposition of $H_{\mathbb C}$ is
orthogonal.  The induced metric on $\mathfrak g_{\mathbb C}$ satisfies
$\mathfrak g^{a,-a}\perp\mathfrak g^{b,-b}$ unless $a=b$.  In particular,
via the isomorphism $T_FD\simeq\bigoplus_{a> 0}\, \geg^{- a, a}$, we obtain
a Hodge metric on $T_FD$. Moreover, since
\[
h_{gF}(x,y)= h_F (g^{-1} x,g^{-1} y),\quad g\in G_{\mathbb R}.
\]
it follows that the Hodge metric defines a $G_{\mathbb R}$-invariant metric on
the tangent bundle of $D$.
\par
A (real) \textbf{\emph{variation of polarized Hodge structure}} over a  complex manifold  $S$ consists of 
  a local system $\underline H_\bR$ of finite dimensional real  vector spaces  
 equipped with a weight $k$ Hodge structure polarized by a $(-1)^k$-symmetric form $Q$ such that

\begin{itemize}
 
 \item  the  Hodge filtrations   glue  to   a holomorphic  filtration 
$\cF$ of the holomorphic bundle $ \cH= \underline H_\bR\otimes\cO_S$;

\item  (Griffiths' transversality) the natural flat connection $\nabla$ induces  a vector bundle map $\cF^p\to \cF^{p+1}\otimes \Omega^1_S$.
\end{itemize}
\begin{rmq}
The motivation of this concept is geometric:   if  $f:X\to S$ is   a smooth, proper morphism between
complex algebraic varieties, then, by the work of P. Griffiths \cite{periods},
the associated local system $\underline H_{\bR} = R^k {f^*}\bR_X$
underlies a variation of pure Hodge structure of weight $k$. It comes  equipped with a natural polarization induced by the cup-product and  the Lefschetz decomposition in cohomology. In fact, we may instead consider cohomology with  rational  coefficients and consider polarizations defined by ample
classes. In this way we obtain a rational variation of polarized Hodge structure.  There is even a canonical flat  integral structure
equipped with a polarizing form.\end{rmq}

By its very definition, locally in a simply connected
open  neighborhood $U$ of  $s\in S$, the assignment $s\mapsto \cF_s$ gives a holomorphic period map $  U \to D$.
To make sense of this globally, one needs to incorporate the effect of the fundamental group at $s$: giving a local system $\underline H$ 
is equivalent to giving a representation on $H $, the fiber of $\underline H$  at $s$. 
This representation preserves $Q$ and so the image of
the fundamental group is a  subgroup $\Gamma$ of $G$, the \textbf{\emph{monodromy group}} of the variation. For variations coming from geometry this subgroup belongs to $G_\bZ$, the subgroup preserving the integral structure coming from integral cohomology.
The monodromy group being   closed and 
discrete,  acts   properly discontinuously on $D$.   It follows that  the quotient $\Gamma\backslash D$ is an analytic space.
The   \textbf{\emph{period map}} in its global incarnation is  the holomorphic map
\[
F:S \to \Gamma\backslash D.
\] 
The Griffiths'  transversality property is equivalent to the statement  that the derivative of the tangent map at $s$ lands in  
\[
T^\hor_{F(s)} D = F ^{-1} \geg_\bC  / F ^0   \geg_\bC \simeq \geg^{-1,1},
\]
 the \textbf{\emph{horizontal tangent space}} at $F$. The corresponding vector bundle is the \textbf{\emph{horizontal tangent bundle}}  
 \begin{equation}
 \label{eqn:TangHorPure}
 T^\hor D = \cF^{-1} \qendo{\cH}/\cF^{0} \qendo{\cH} \simeq \qendo {\cH}^{-1,1}  ,
 \end{equation} 
 where the isomorphism is in the category of  $C^\infty$ hermitian vector bundles.
 Conversely, a holomorphic map from a complex manifold to a quotient of a period domain $D$ by 
 a discrete closed  subgroup of $G$ is a period map provided it is
 locally liftable to $D$ as a horizontal holomorphic map.

\subsection{Curvature properties} \label{ssec:CurvPure}

By \cite[Theorem 9.1]{curv} the holomorphic sectional curvature of $D$
along horizontal tangents is negative and bounded away from zero. As shown in \cite{hodge4},   the full   curvature 
tensor along a  $(1,0)$-tangent vector   of $u  \in \geg^{-1,1}$  is given by
$R(u,\bar u) = -\ad {[u, \bar u] }$
so that the bisectional curvature in the $(u,v)$ unit-norm direction becomes
\[
K_F(u,v) =  h_F(R(u,\bar u)v,v )= -  h_F( [ [u,\bar u] v] ,v)  = \|[u,v]\|_F^2 - \|[\bar u, v]\|_F^2.
\]
As recalled below, in geometric situations $u$ and $v$ commute, which implies $K_F(u,v)\le 0$.
We shall outline how this implies that for a global section $\eta$ of   $\qendo{\cH}$ which is of Hodge type $(-1,1)$, 
  the function  $F\mapsto \|\eta(F)\|^2$ is   plurisubharmonic   on $S$.

\par
This phenomenon occurs   more generally for sections $\eta$ of any  holomorphic  vector bundle $\cE$ equipped with a hermitian metric $h$.
Recall that   there is a unique metric $(1,0)$-connection $\mathsf D$, the \textbf{\emph{Chern connection}} for $(\cE,h)$.
The bisectional curvature appears in    a Bochner type formula  \cite[Prop. 11.1.5]{3authorsBIS}, a special case of which reads
\begin{equation}
 \label{eqn: plurisubh}
 \begin{array}{rl}
 \del_{u} \del_{\bar{u} } \| v \|^2  &= \| \mathsf D  _{u} v \|^2   - h  ( R_{\mathsf D} (u,\bar u ) v, v )\\
 & \quad\quad u \in T^{1,0} _s S,\quad v=\eta(s).
\end{array}
\end{equation} 
Recall that a real  $C^2$-function $f$ on an open subset $U$ of $\bC^n$ is plurisubharmonic if $\ii \del\bar \del f$ is a positive definite $(1,1)$-form. This is equivalent to $\del_{u} \del_{\bar{u} } f \ge 0$ for all type $(1,0)$-tangent vectors on $U$.
If $h  ( R_{\mathsf D} (u,\bar u ) v, v )\le 0$,     formula~\eqref{eqn: plurisubh} shows that $s\mapsto \|\eta(s)\|$ is a plurisubharmonic function on $S$.    

We apply this to  our situation with $\cE$ the   bundle   $\qendo{\cH}$ on $S$. A holomorphic section $\eta$ of this bundle is 
  invariant under the global monodromy and so  in particular invariant under
local monodromy at infinity. 
We now invoke:
\begin{prop}[\protect{\cite[Cor. 6.7']{schmid}}] Let there be given a polarized variation over the punctured disk. Then an  invariant holomorphic section
of the Hodge bundle  remains bounded. \label{prop:HNboundPure}
\end{prop}
Quasi-projective manifolds do not admit bounded plurisubharmonic functions except constants 
(cf. \cite{psh}).  Consequently,  in the present situation  the Hodge norm $\|\eta\|$  is constant along curves in $S$ and hence on all of $S$.
The bundles on $S$ are pull backs under the period map $F$ of bundles on $D$ and the calculation takes place on   $D$. 
In particular,  tangent vectors from $\xi\in  T_sS$ of type $(1,0)$  are pushed to  $u=F_*\xi \in F_*(T_sS) \subset T^\hor_{F(s)}D=\geg^{-1,1}_{F(s)}$. 
Summarizing the discussion so far we have shown:

\begin{lemma} \label{lem:PluriSubConseq} Let there be given a polarized variation of Hodge structure    $(\cH,Q,\cF)$ 
over a quasi-projective complex manifold $S$. 
Let $\eta$ be  a holomorphic section of   the endomorphism bundle $\qendo{\cH}$ which is of Hodge type $(-1,1)$.

Suppose that for all    $u \in T^\hor _F D$ tangent to the image of the period map at $F=F(s)$, $s\in S$, one has $[u, v]=0$, $v=\eta(s)$. Then $\|\eta\|$ 
is a plurisubharmonic bounded (and hence constant) function, $\mathsf D \eta=0$ and $[\bar u,v]=0$.
\end{lemma}

The next step is to relate the Chern connection and the Gauss--Manin connection $\nabla$ as explained  in \cite[Prop. 13.1.1]{3authorsBIS}.
It uses the \textbf{\emph{Higgs bundle structure}} on the  Hodge bundle  $\cH=\bigoplus_{p+q=k} \cH^{p,q}$.
To explain this, note that the Hodge decomposition is only a $C^\infty$-decomposition. However,  $\cH^{p,q}$ receives a complex structure
through the isomorphism $\cH^{p,q}\simeq \cF^p/\cF^{p+1}$. There is a corresponding operator $\bar\del :\cH \to \cH\otimes \cE^{0,1}_S$ with the property that local sections $v$ of
$\cH^{p,q}$ are holomorphic if and only if $\bar\del v=0$. 
The Gauss--Manin connection $\nabla$     can be decomposed as follows:
\begin{equation}
 \label{eqn:HiggsPure}
 \nabla = \sigma  + \underbrace{\bar\del + \del}_{\mathsf D } + \sigma^*.
\end{equation} 
 Here   $\del: \cH \to \cH \otimes \cE^{1,0}_S$  is a differential operator which preserves Hodge type.   
 The operator 
 \[
  \sigma : \cH \to \cH\otimes \cE^{1,0}_S,
  \]
  an   endomorphism of $\cH$ of Hodge type $(-1,1)$ with values in the $(1,0)$-forms,   is called the \textbf{\emph{Higgs field}}.
 Its adjoint with respect to the Hodge metric is the linear operator $\sigma^* \in \qendo {\cH }^{1,-1}\otimes  \cE^{0,1}_S$.
 
  By functoriality, a similar decomposition holds for  the  bundle  $\qendo{\cH}$. Since the tangent bundle comes from the adjoint representation of $G$ on the endomorphism bundle, it follows from
 \cite[Prop. 11.4.3]{3authorsBIS}  that  
  for any  horizontal tangent vector $u$ of type $(1,0) $  at $F\in D$ we have:
\begin{gather*}
  \nabla_u     =  \del_u      +     \ad {u}  ,  \\
\nabla_{\bar u}     =  \bar\del_u +    \ad { {\bar u}}    .
  \end{gather*}
Assuming, as before, that $\ad{u} v=[u,v]=0$, by Lemma~\ref{lem:PluriSubConseq}, $\del_u v=\bar\del_u v=0$ and $[\bar u,v]=0$.
Invoking Lemma~\ref{lem:PluriSubConseq}, we may summarize the above discussion as follows:

 \begin{prop} \label{prop:WhenFlatPure} Let  a     $\eta$  be a holomorphic section of $\qendo{\cH}$ of type $(-1,1)$.
 At a point $F$ in the image of the period map, set $v=\eta(F)$ and assume that  $[u,v]=0$ for all vectors $u\in T_{F }D$,  tangent   to the period map.
Then $\eta$   is parallel with respect to the Gauss--Manin connection. Moreover, one has $[\bar u,v]=0$.   \end{prop}

\subsection{Deformations  of period maps}
\label{sec:tang}
 
The kind of deformations we are interested in are deformations of holomorphic maps $\varphi: X\to Y$
between complex  spaces  $X$ and $Y$ that keep $X$ and $Y$ fixed. By definition, these are given by
complex-analytic maps $\Phi: X\times T\to Y\times T$ with $(T,0)$ a germ of an analytic space centered at $0$ such that
\begin{itemize}
\item $\Phi(x,t)= (\varphi_t(x), t)$;
\item $\Phi(x,0)= \varphi(x)$.
\end{itemize}
Such deformations are in one-to-one correspondence to deformations   of the graph of $\varphi$ and as in \cite[\S 3.4.1]{Sern}
the tangent space at $\varphi$ of such deformations is given by the space  $H^0(X, \varphi^*T(Y))$, the space   of infinitesimal deformations of $\varphi$ keeping $X$ and $Y$ fixed.    
Here $T(Y)$ is the tangent sheaf of $Y$, i.e.,  the dual of the cotangent sheaf of $Y$.

We apply this to  period maps  $F:S\to \Gamma\backslash D$.
In geometric situations we are  interested in deformations of   families  of varieties and
the corresponding  deformations $ S\times T \to \Gamma\backslash D$ of period maps $F$ that stay locally liftable and horizontal.   
We pass to the smallest
unramified  cover of  $S$ over which there is no monodromy and lift the period map accordingly, say
to $\tilde F:\tilde S\to D$ and then the  space of infinitesimal deformations in which  we are interested   is
the subspace of  $H^0(\tilde F^* T^\hor(D))$ consisting of sections commuting with the  monodromy action.
 In view of the isomorphism \eqref{eqn:TangHorPure},  such a deformation lifts  to a  holomorphic section of  $\qendo {\cH}$
which at any given point   $F\in D$ in the image of the period map projects to $ \geg^{-1,1}$.  
In this situation we can apply Proposition~\ref{prop:WhenFlatPure} since
the condition $[u,v]=0$ follows from  horizontality
 (see \cite[Prop. 5.5.1]{3authorsBIS}) and we conclude:

\begin{thm} Let $S$ be smooth and quasi-projective and $F:S\to \Gamma \backslash D$ a period map.
 The   space of infinitesimal deformations of $F$  remaining  horizontal is isomorphic to
 \warn{the space of   flat   sections of type $(-1,1)$ of the bundle $\qendo{\cH}$}.  
 Moreover, at a point $F$ in the image of the period map, setting $v=\eta(F)$, one has    $[\bar u,v]=0$, $v=\eta(F)$, for all tangent vectors at $F$ tangent to the period map. 
 \label{thm:MainDefPure}
 \end{thm}

Complementing this result we remark that according to an argument generalizing the one given by Faltings \cite{arafal} for weight $1$, the corresponding deformation space is smooth at $F$:
\footnote{See also the proof of Theorem~\ref{thm:Main}. (2).}

\begin{prop} \label{prop:smoothdefs} 
The space of deformations of a period map $F$ which keep  source and target fixed,  and stay locally liftable and horizontal,  is smooth at $F$.
\end{prop}

It follows that   $F$ is locally rigid precisely when $\endo^\Gamma( H_\bC,Q)^{-1,1}= 0$.    This gives criteria for rigidity. From the last property,  $[\bar u,v]=0$, we see that   the concept of a regularly tangent period map as introduced in \cite{hodge4} comes up naturally:

\begin{dfn}   \label{dfn:RegTangPure} 
 The period map  $F$ is called  \textbf{\emph{regularly tangent}}  at  $s\in S$ if    
     the only vector $v\in \geg^{-1,1}_{F(s)}$ with $ [\ \bar u,v]=0$ for  all $u\in F_* T_sS $ is the zero vector.
 If this is the case for all $s$ we speak of a period map which  is regularly tangent along $S$.
 \end{dfn}

\begin{corr} A period map $F:S\to \Gamma \backslash D$  is rigid (as a period map) if one of the following two conditions hold:
\begin{itemize}
\item The only flat   endomorphism of the underlying local system which is of Hodge type $(-1,1)$ is the zero endomorphism.
\item $F$ is everywhere regularly tangent.
\end{itemize}
\label{cor:MainCrit}
\end{corr}

\subsection{Examples of rigid period maps}
\label{ssec:ExRigPure}

We first recall  the following concept:
\begin{dfn}  (1) A polarized
 real variation   of weight $k$ has  Higgs field   of \textbf{\emph{Hodge--Lefschetz  type}} $a$ if
\begin{itemize}
	\item the Hodge depth is  $a$, that is   the only non-zero Hodge numbers are in the range
		$(a,k-a),\dots,(k-a,a)$;
		\item  the Higgs field in some, or equivalently, in a generic direction has components   $u^{j} :   \cH_s^{k-j,j}  \to  \cH_s^{k-j-1,j+1} $,  
		$j=a,\dots, k-a$ which are all   isomorphisms. 
	\end{itemize}
This implies that the Hodge depth  is exactly $a$ and all non-zero Hodge numbers are equal.   
\\
(2) A  polarized pure variation has
(strictly) \textbf{\emph{maximal Higgs field}} 
if    it  is a direct sum of variations with Higgs field     of Hodge--Lefschetz  type, the \textbf{\emph{strands}}
of the field.\footnote{For complex variations of Hodge structure, the definition
as given in  \cite{VZ}  is more complicated, but for real variations  it reduces to the one given here.}
\label{dfn:MaxHiggs}
\end{dfn}

\begin{prop} \label{prop:maxHiggsIsRigid}
 A pure variation which  has   maximal Higgs field    is regularly tangent and hence rigid.
 \footnote{ 
This confirms the result \cite[Lemma~4.3]{VZ}  for strictly maximal Higgs fields over curves.
  In loc.\ cit. several examples are given of families $\set{X_s}_{s\in S}$ of $d$-dimensional 
  Calabi--Yau's over a curve  for which the middle dimensional cohomology gives variation with strictly
  maximal Higgs field.}
\end{prop}
\begin{proof} Let $\xi\in T_sS$ be generic so that the components of $u=F_*\xi$ are isomorphisms on each
Hodge--Lefschetz strand of the variation.
Assume $[\bar u ,v]=0$ which at an   extremal Hodge component  means either $ \bar u  \comp v=0$ or 
 $v\comp  \bar u =0$.
But since  the Hodge components $u $ and its adjoints are isomorphisms on each strand, this implies that  the extremal  components of $v$ vanish and hence, by induction, all components.  
\end{proof}

\warn{We can now enumerate some examples.}

\medskip
 \textbf{(1)  Maximal Higgs fields, weight 1.}  Let $(H_\bR,Q)$ be a weight one polarized Hodge structure and set $V=H^{1,0}$.
Consider the hermitian inner product $h(x,y) = \ii Q(x,\bar y)$ on $V$. The anti-complex linear map $\bar x \mapsto h(-,x)$ induces an  identification  
$\warn{\overline V}=V^*$.
The Higgs field becomes  a $Q$-symmetric  endomorphism $u\in \hom (V,V^*)$ and  hence  can be identified with
$P_u\in S^2 V^*$, a quadratic homogeneous polynomial function on $V$. Under this identification, $u$ is an isomorphism precisely when $P_u$ has maximal rank.
Hence  a polarized weight one variation has maximal Higgs field if and only if 
the corresponding quadratic polynomial     has  generically maximal rank.
\par

\textbf{(2)  Maximal Higgs fields, weight two.} We recall some general properties of weight two polarized Hodge structures $(H_\bR,Q)$, say
$V=H^{2,0}$, $W=H^{1,1}$. The hermitian  product $h(x,y)= Q(x,\bar y)$ restricts non-degenerately on $V$ 
and  under  the anti-linear map $\bar x \mapsto h(-,x)$
there are identifications $\warn{\overline V}=V^*$ and $\warn{\overline W}=W^*$. 
If $A:V\to W$ is linear, the anti-linear dual is denoted $\widehat A: \warn{\overline W} \to \warn{\overline V}$.
Suppose that  $u\in \endo {H,Q}$ is horizontal, that is,  of type $(-1,1)$. Then we have  
$V\mapright{u_1} \warn{W = \overline W}\mapright{u_2} \overline V$ with $ u_2=\widehat{u_1}$.

One easily sees that $Z:=\im(u_1)=[\ker(u_2)]^\perp$ and that $u^*=(u_1^*,u_2^*)$ is such that
$u_1^*=0$ on $Z^\perp$ and $u_2^*: \overline  V \to Z\subset W=\overline W$. Applying this to a weight two variation we see 
that subvariation associated to  $(H^{2,0},Z,H^{0,2})$ is of Hodge--Lefschetz type  if and only if $u_1$ is an injection.
Note that the Higgs field is zero on $Z^\perp$ and so it can only be of Hodge--Lefschetz type if it vanishes.
Concluding, we can only have a maximal Higgs field if $Z^\perp=0$ and then $h^{2,0}=h^{1,1}=h^{0,2}$.
 
\par 

\par
 \textbf{(3) Irreducible modules.} 
  If $(H,Q)$ is the typical stalk of a variation of pure polarized Hodge structure on $S$ and $H_\bC$ is irreducible as a 
  $\pi=\pi_1(S)$-module, $\endo^\pi_\bC(H ,Q)$  is $1$-dimensional and
since it has a pure Hodge structure, it has type $(0,0)$.  Consequently 
we have $\endo^{\pi,\hor}(H ,Q)=0$
and so, by Corollary~\ref{cor:MainCrit}, such a  variation is rigid.

As a geometric example we may consider a \textbf{Lefschetz pencil of complete intersections} in projective space.
\warn{By S.\ Lefschetz' theory of the variable cohomology  (cf. e.g. \cite[Section 4.2.]{3authorsBIS}) the  latter is  always  absolutely  irreducible} under the action of the monodromy group.
The period map for the family is   an immersion except for a cubic surface or an even dimensional intersection of two
quadrics (see e.g. \cite[Thm. 2.1.]{flenner}). Hence the Lefschetz pencil itself is rigid as well.

\par
\textbf{(4) Abelian varieties (or polarizable weight one variations).}  Ma. Saito \cite{Sa} gives a complete classification of the non-rigid families $\set{A_s}_{s\in S}$ of
$g$-dimen\-sional abelian varieties $A_s$. From this it follows that
rigid families occur in abundance   as we now show.
We can decompose the variation into irreducible factors.  
Assume that none of these factors are isotrivial. Then the family is rigid  if  one of the following situations occur:
	\begin{itemize}
	\item $g\le 7$;
	\item the variation is irreducible and $g$  is prime;
	\item $S$ is non-compact,  the variation is irreducible  and some local monodromy operator at the boundary 		has infinite order.
	\end{itemize}
Observe that any  weight one variation  coming from curves is irreducible since the polarization comes from the irreducible theta-divisor.
So non-isotrivial families of genus $g$ curves have   rigid
period map  if for example $g$ is an odd  prime number, or  if  the family has infinite order local monodromy
 at infinity. 
\par
\textbf{(5) K3-type variations.}   
A  variation of Hodge structure on  a local system  is of \textbf{\emph{K3-type}}, if it has weight $2$ and $h^{2,0}=1$. In general such a system splits as $S\oplus \underline T$ where $S$ is locally constant.
If  $\underline T\not=0$ it is an irreducible variation, again  of K3-type.
Geometric examples come from the  primitive $2$-cohomology of a  projective algebraic K3 surface $X$
which  splits as 
\[
H^2_\prim(X) =  S(X)\oplus T(X),
\]
where $S(X)$ is spanned by the classes of the algebraic cycles and $T(X)=S(X)^\perp$.
In a family  of K3 surfaces, $\set{X_s}_{s\in S}$,  there is a    maximal locally constant part $\mathsf S$  of the variation
given by  algebraic cycles  classes, and so the variation  splits 
as $\mathsf S\oplus \underline T$.  In  \cite{SaZ}   $\underline T$ is called the essential variation. The subset $D(\mathsf S)$ of the period domain 
corresponding to K3 surfaces for which the Picard lattice contains $\mathsf S$ has dimension $20-\rho$ where $\rho=\rank (\mathsf S)$.
The generic K3 with period point in $D(\mathsf S)$ has Picard lattice isomorphic to $\mathsf S$.  The period map of an essential variation
has $D(\mathsf S)$ as its target.  As a special case of the results of \cite{SaZ} we mention:

\begin{prop}  An essential   K3-type variation of rank $k$  on a quasi-projective variety $S$ 
with immersive period map is rigid  in  each  of the following cases:
	\begin{enumerate}
	\item $k$ is not divisible by $4$;
	\item $S$ is not compact  and  some  local monodromy operator at infinity  has 
	maximal order of unipotency $3$. 
	\end{enumerate} \label{prop:RigidK3}
\end{prop}

\par
\textbf{(6) Calabi--Yau manifolds.}   
 Proposition~\ref{prop:RigidK3} (2)  generalizes to   Calabi--Yau's: 

\begin{thm*}{\protect{\cite[Cor. 3.5]{hodge10}}}
Let $f:X\to S$ be a non-isotrivial family of $k$-dimensional Calabi-Yau's over
 a non-compact curve $S$ and suppose that there is a point at infinity 
 where the local monodromy operator for $H^k$ has maximal order of unipotency $k+1$. Then $f$ is rigid.
\end{thm*} 

\section{Mixed period domains and Hodge metrics}
\label{sec:permaps}

We recall some material from \cite{tdh,higgs,dmj,sl2anddeg,usui} on mixed Hodge structures and related period domains.

\subsection{Basics on mixed Hodge structure}
\label{ssec:BasMHS}

Fix a   finite dimensional  $\bQ$-vector space  $H_\bQ$
endowed with a finite increasing weight filtration $W $  
whose graded pieces $\gr^{W }_k$ are equipped with non-degenerate
$(-1)^k$-symmetric  real-valued  bilinear forms $Q_k$.  These data are denoted $(H,W,\set{Q_k})_\bQ$.
Associated to these data the following groups are relevant:
the real Lie group
\[
 G_\bR   = \sett{g\in \gl{H_\bR}}{ g(W_k)\subset W_k, \gr^W\!\!(g)\in \aut{\gr^W(H_\bR,Q) }}
 \]
 and its complexification $G_\bC$ as well as an  intermediate group
\begin{equation}
\label{eqn:groupG}
 G= \sett{g\in G_\bC}{g \text{ induces   a real transformation on } \gr^W(H)}.
\end{equation} 
  A decreasing filtration $F$ on $H_\bC$ together with  the data $(H,W,\set{Q_k})_\bR$
 define  a  \textbf{\emph{graded polarized mixed Hodge structure}} if $F$ induces a pure weight-$k$ Hodge 
structure on $\gr_k^{W}$ polarized  by $Q_k$. 
A basic tool is  the \textbf{\emph{Deligne splitting}}~\cite{tdh} for the mixed Hodge structure, a   
unique functorial bigrading,  
\begin{equation}
      H =H_\bC = \bigoplus_{p,q}\, I^{p,q} \label{eqn:DelSplit}
\end{equation}
such that $F^p = \bigoplus_{a\geq p}\, I^{a,b}$,
$W_k \otimes \bC= \bigoplus_{a+b\leq k}\, I^{a,b}$  and
\[
\overline{I^{p,q} }= I^{q,p} \mod \bigoplus_{a<p,b<q}\, I^{a,b}.
\]
The  graded polarized mixed Hodge structures $(H,W , \set{Q_k})_\bR$   with fixed   Hodge numbers $h^{p,q}= \dim I^{p,q}$
are parametrized by a mixed period domain  which we always denote   $D$.

\begin{rmk}\label{rmk:split}
A mixed Hodge structure is \textbf{\emph{split over $\bR$}} if $\overline {I^{p,q}} = I^{q,p}$. Examples
occur if the weight filtration  has only two adjacent weights. Consider for instance mixed Hodge structures
with $h^{0,0}=h^{-1,-1}=1$, an example of a Hodge--Tate structure. The corresponding mixed domain is $\bC$
(while the extension data are isomorphic to $\ext (\bZ(0),\bZ(1)) = \bC^*$).
\end{rmk}

In analogy with the pure case, $D$ is a complex manifold. Moreover,  $D$   is  a homogeneous domain under
the  group $G$ defined by \eqref{eqn:groupG} and so
\[
D= G/ G^F,\quad G^F=\text{ stabilizer of  } F \text{ in } G.
\]
There are important differences with the pure case since
the group $G^F$ is in general not compact in contrast to $G_\bR^F$. The real Lie group $G_\bR= G\cap \gl {H_\bR}$ 
acts only  transitively on the locus
 of split mixed Hodge structures which need not be a complex manifold. However, if $D$ parametrizes split mixed 
 Hodge structures,  $D=G_\bR/G_\bR^F=G/G^F$, although in general we have $G\not= G_\bR$ and, while 
 $G_\bR ^F$ is compact, $G^F$ need not be compact. See \cite{usui} for the case of adjacent weights.

As in the pure case, there is a ``compact dual'' of $D$,
\begin{equation}
\label{eqn:CompactDual}
\check D= G_\bC/ G_\bC^F.
\end{equation} 
By functoriality, any point $F\in D$ induces a mixed Hodge structure on $\endo(H)$ with Deligne splitting
\begin{equation}
\label{eqn:EndoDeligne}
\begin{array}{lcl}
\endo (H) &= &\bigoplus_{p,q} \endo^{p,q}(H)  , \\
\endo^{p,q}(H)  &=& \sett{u\in \endo(H)}{u( I^{r,s}) \subset I^{r+p,s+q} \text{ for all } r,s}.
\end{array}
\end{equation}
and also on the space $ \geg_\bC    =   \text{Lie}(G_{\bC})=\endo (H,W,Q) _\bC$ of endomorphisms preserving $Q$:
\[
\geg  ^{r,s} =  \geg_\bC  \cap \endo^{r,s} (H) \quad r+s\le 0.
\] 
The  restriction on the bigrading comes from the weight preserving property of elements of $G_\bC$.

There is also an analog of \eqref{eqn:TangSpasHS}. To see this, first observe that the exponential map
$u\mapsto\e^u$ maps a neighborhood $U$ of $0$ biholomorphically to an open neighborhood of $G_\bC$
and so, composing with the orbit map yields a biholomorphic map 
\begin{align}
\varphi: U\cap \gq_F & \mapright{\simeq} \im (\varphi) \subset  D  \label{eqn:coords}\\
           u &\mapsto \e^u. F\,.\nonumber
           \end{align}
 Since the Lie algebra of $G_\bC^F$ equals  $ F^0  \geg_\bC= \bigoplus_{r\geq 0}\, \geg^{r,s} $, the subspace
 \begin{equation}
      \gq_F = \bigoplus_{r<0}\, \geg^{r,s}       \label{eqn:tang-alg}
\end{equation}
is a vector space complement to $F^0 \geg_\bC  $ in $ \geg_\bC $.
Accordingly, $d\varphi(0)$  induces a natural isomorphism  of  complex vector spaces 
\begin{equation}
\label{eqn:HolTngSp}
    T_F (D)  \simeq \gq_F   .
\end{equation}

\subsection{Period maps for variations of mixed Hodge structure}
\label{ssec:PerMapMHS}
 
Similarly to a  pure variation, one can speak of  a \textbf{\emph{variation of  graded polarized mixed Hodge structure}} on $S$. The only difference
with the pure case  is the presence of a weight filtration with the property that
on its $k$-graded quotients the Hodge filtration induces a pure polarized variation of weight $k$. 
Such variations are in one to one correspondence with period maps to the mixed period domain $D$
for the graded polarized mixed Hodge structure on a  typical fiber. 
The map
sending $s$ to  the point $F(s)\in D$ corresponding to the mixed Hodge structure on the fiber over $s$ of the local system  is well defined locally, for instance if $S$ is a polydisc or, more generally, a simply connected manifold. We say that we have a \textbf{\emph{local period map}} $ S\to D$, $s\mapsto F(s)$.
As in the pure case, there is a monodromy group $\Gamma$ and we get a well defined (global) \textbf{\emph{period map}}
\[
F: S\to \Gamma\backslash D.
\]
Again, as in the pure case, variations coming from geometry have an underlying integral structure. In particular, this implies that $\Gamma$ acts properly
discontinuously on $D$  and so  $\Gamma\backslash D$ is an analytic space.
For lack of a good reference, we provide a proof of this fact in
Appendix~\ref{sec:ProperDisc}.

The period map is horizontal, meaning that the derivative at $s\in S$ sends $T_s S$ to the subspace of the tangent space $T_{F(s)}D$
given by 
\[
 \gr_{\cF}^{-1}   \qendo {\cH} _{s}=    \bigoplus_ {q\le 1} \geg^{-1,q}_{F(s)}.
\]
This  is a consequence of Griffiths' transversality. Since one only uses the Hodge filtration  
 to describe  of the tangent bundle as well as  the horizontal tangent bundle the description
 in the mixed case parallels the one in the pure case. For later reference we make this more explicit.
 Using the induced Hodge filtration on the endomorphism bundle, we have a surjective map  of   holomorphic vector bundles
on $\check D$ 
\begin{equation}
\label{eqn:endoHodgeFlag}
\xymatrix{
\cF^{-1}\qendo {\cH}   \ar[r]^{\pi^{\hor\quad\quad} } &    T^\hor( \check D) =\gr_{\cF}^{-1}   \qendo {\cH}  .
  }
\end{equation} 
 Mixed period maps of geometric origin have all of the above properties. See e.g.\ \cite{SZ,usui}.

To close this section, we observe that the same argument used in  the  pure case shows: 
 \begin{lemma} For a  local period map $  F:  S\to D$,   the image of the  tangent space at $s$ is an abelian subalgebra
of  $ \geg_\bC$ contained in \warn{$U^{-1}\geg_{F(s)}= \bigoplus_ {q\le 1} \geg^{-1,q}_{F(s)}$}. \label{lem:AbSA}
\end{lemma}

\subsection{Mixed Hodge metrics}
\label{ssec:HM}

The  mixed  Hodge metric   $h_{(F,W)}$ on $H$ is defined as follows.
 We   first  declare  the splitting \eqref{eqn:DelSplit} to be  orthogonal and then define the metric
on $I^{p,q}$ making use of the graded polarization   on  $ \gr^W \!  H $ as follows.
The summand $I^{p,q}$ maps isomorphically onto the subspace
$H^{p,q}$ of $\gr^W_{p+q}$. So  on classes $[z]$   of elements    $z\in I^{p,q}\subset W_{p+q}$
modulo $W_{p+q-1}$ the metric  $h_{F,W}$ can be defined by setting:

\begin{equation} \label{eqn:grhmetric}
h_{(F,W)}(x,y) = (\gr   h)_F ([x],[y]) ,\quad x,y \in I^{p,q}.
\end{equation}
 Let $*$ denote the adjoint with respect to the metric $h_F$.  Then,
\begin{equation}  
           *:\endo^{p,q}(H)\to \endo^{-p,-q}(H) .    \label{eqn:adjoint-pq}
\end{equation}
The Hodge metric induces a metric on $\endo{H}$ given by 
\begin{equation}
         h_F(\alpha,\beta) = \tr(\alpha \beta^*)   \label{eqn:mhm}
\end{equation}
where $\beta^*$ is the adjoint of $\beta$ with respect to $h_{F,W}$.
The Deligne splitting \eqref{eqn:EndoDeligne} of $\endo{H}$ is then  orthogonal with respect to the associated metric.
The induced Hodge metric on  the holomorphic tangent space $T_{D,F}$ of $D$ at $F$ comes from the natural identification
\eqref{eqn:HolTngSp}. 

In the sequel, we make use of the following orthogonal splittings.
  \begin{equation*}
    \begin{array}{lcl}
     \geg_\bC &= &\underbrace{\gn_+ \oplus \geg^{0,0}_F}_{\lie{G_\bC^F}}\oplus  \underbrace{\gn_- \oplus  \Lambda^{-1,-1}_F}_{\gq_F} ,\\
    \text{where}&&\\
    \gn_+ &= &\bigoplus_{a\ge 0,b<0} \geg^{a,b}_F,\\
   \gn_-&=&\bigoplus_{a< 0,b\ge 0}  \geg^{a,b}_F,\\
    \Lambda^{-1,-1}_F&=& \bigoplus_{a\le -1,b\le -1} \geg^{a,b}_F.
   \end{array} 
   \end{equation*}
  See Figure~\ref{fig:splitt} below.
 
    \begin{figure}[htbp]
\begin{center}
\begin{tikzpicture}[scale=0.7] [>=stealth]
\draw [->] (0,0) -- (6,0) node [at end, right] {$a$}; \draw [->] (0,0) -- (0,6) node [at end, left] {$b$};
\draw[semithick] (-6,6) -- (6,-6);
 \draw[fill=black]  (0,0) circle [radius=0.15cm]; \node at ( 0.5, 0.5)  {$ \germ g^{0,0}_F$};
 \draw[fill=black]  (1,-1) circle [radius=0.1cm];  \node at ( 1.8,  -0.7)  {$(1,-1)$};
\draw[fill=black]  (-1,-1) circle [radius=0.1cm];  \node at ( -1.5,  -1.5)  {$(-1,-1)$};
\draw[fill=black]  (-1,0) circle [radius=0.1cm];  \node at ( -1.5,  -0.5)  {$(-1,0)$};
\draw[fill=black]  (-1,1) circle [radius=0.1cm];  \node at ( -1,  1.5)  {$(-1,1)$};
\draw[fill=black]  (0,-1) circle [radius=0.1cm];  \node at ( 0.5,  -1.5)  {$(0,-1)$};
  
  \node at (2.5,-4)  {$\germ n_+$}; 
 \node at (-4,2.5)  {$\germ n_-$}; 
 \node[matrix]  at  (-6, -0.5) {$ \gq_F  $ \Bigg{\{ }  & $\quad$  \\
           $\quad$ & $\quad$\\};
    \node at (-3.5,-4)  {$   \Lambda^{-1,-1}_F=\bar \Lambda^{-1,-1}_F$}; 
  \draw[color=blue,  thick] (-6,-1) -- (-1,-1)--(-1,-6);
   \draw[color=red, thick] (-6,6) -- (-1,1)--(-1,0)--(-6,0);
   \draw[color=green,   thick] (6,-6) -- ( 1,-1)--(0,-1)--(0,-6);
\clip (5.9,-6)-- (-6,5.9)-- (-5.9,-5.9);
\draw[step=1cm,   dotted] (-6,-6) grid (6,6);
 \end{tikzpicture}
\caption{Decomposition of $ \geg_\bC$}
\label{fig:splitt}
\end{center}
\end{figure}
The orthogonal decomposition   $ \geg_\bC= \gn_+ \oplus \geg^{0,0}\oplus \gn_- \oplus  \Lambda^{-1,-1}_F$ defines
 respective  orthogonal projectors    
\begin{equation}
\label{eqn:OrthProjs}
 \quad\quad\begin{array}{rccc}
  \pi_\pm&:  \geg_\bC  &\mapright{\quad} &\gn_\pm ,\\
  \pi_0&:  \geg_\bC &\mapright{\quad} &\geg^{0,0},\\
  \pi_ {\Lambda^{-1,-1}}&:  \geg_\bC &\mapright{\quad}  & \Lambda^{-1,-1}_F,\\
  \pi_\gq&: \geg_\bC &\mapright{\quad} & \gq_F.
 \end{array}   
 \end{equation}

\subsection{Higgs bundles in the mixed setting}
\label{ssec:Higgs}
As in the pure case, the Hodge filtration defines a Higgs bundle structure on a variation of mixed Hodge structure over $S$.
The role of $H^{p,q}$ is played by
\[
U^p_F := \bigoplus_q  I_F^{p,q},
\]
which cuts out $H^{p,q}$ on $\gr^W_{p+q}H$. These glue into the  $\cC^\infty$ bundles 
   \[
   \cU^p=\bigoplus_q  \cI^{p,q},
   \]
  isomorphic to  $\cF^p/\cF^{p+1}$.
  The Higgs structure is slightly more involved than in the pure case: by
  \cite{higgs}, the Gauss--Manin connection of $\cH$
decomposes as
\begin{equation}
 \label{eqn:Higgs}
 \nabla = \tau_0  + \bar\del + \del + \theta.
\end{equation} 
 Here $\bar\del$ and $\del$ are differential operators of type $(0,1)$
and $(1,0)$ which preserve $\cU^p$.  The first, $\bar\del$, gives the holomorphic structure induced  by the $\cC^\infty$-isomorphism 
$\cU^p \simeq \cF^p/\cF^{p+1}$.
 The \textbf{\emph{Higgs field}} in this setting is the operator $\theta$, an   endomorphism of $\cH$ sending $\cU^p$ to $\cU^{p-1}$ with values in the $(1,0)$-forms
 and  $ \tau_0$ is an   endomorphism  sending $\cU^p$ to $\cU^{p+1}$ with values in the $(0,1)$-forms.
 The Higgs field has a geometric interpretation which directly follows from its construction:
 
 \begin{lemma} Let     $F:S \to  D$ be \label{lem:OnInfPermap}
a local  period map.  Under the correspondence \eqref{eqn:HolTngSp}, 
     the Higgs field in  a tangent direction  $\xi\in T_sS$ can be 
     identified with   $F_*\xi$ viewed as a degree $-1$  endomorphism of $\cU_\higgs$:
\[
 \theta^{1,0}_\xi  =F_*\xi  :     \cU_{\higgs,s} \to \cU_{\higgs,s} , \quad  F_*\xi\in \gq_{F(s)}^\hor.
\]
In particular,  the period map is injective, if and only if  for all non-zero directions $\xi$ the map $\theta^{1,0}_\xi$ is not the zero-map.
\end{lemma}

    By functoriality   all this  applies to  the endomorphism bundle $\qendo{\cH}$ with induced variation of mixed Hodge structure. 
   In the latter set-up we  have:

\begin{lemma}  \label{lem:FlatDefs}
Let $\eta $ be a    local holomorphic section of $\cU^{-1} \qendo{\cH}   $ at $s\in S$ and $\xi \in T_s(S)$  a tangent vector of type $(1,0)$
at $s$. Set   $ v=  \eta (s) , u= F_*\xi\in  \geg ^\hor_{F(s)}  $.
Then  
\begin{gather}
\label{eqn:TngIdAdj} \nabla_\xi v   =  \del_\xi   v +     \ad u v,  \\
\nabla_{\bar \xi} v   =    \pi^{(0)}  \ad {\pi_+{\bar u}} v     .
 \label{eqn:TngIdAdjBis} 
\end{gather}
Here the  bundle map   $\pi^{(0)}  $ stands for the orthogonal projection onto $\cU^{ 0 }$. 
  \end{lemma}
 \begin{proof} First consider the general case of a mixed variation $\cH$ and $u\in \cU^{ p }$. 
 The operator  $\bar\del$  in \eqref{eqn:Higgs} breaks up in a component of bi-degree $(0,0)$ and a component $\tau_-$ of bi-degrees $(0,-1)+(0,-2)+\dots$.
Comparing with \cite[Lemma 5.11]{higgs}, letting  $\pi^{(p)}  $ stands for the orthogonal projection onto $\cU^{ p }$, we see
 \[
\pi^{(p)}  \pi_+(\bar u) = \tau_-,\quad \pi^{(p+1)}  \pi_+(\bar u)  =\tau_0   .
\]
Since  the action  of $\geg_{F(0)}$ on $\endo (\cH_{F(0)})$
is  through the adjoint action, setting $p=-1$ we see that $\tau_0$ gives rise to $\pi^{(0)} \ad {\pi_+(\bar u)} v $.
Since $\eta$ is holomorphic,  $ \bar \del \eta=0$. As to $\theta$, comparing with equation (5.20) in loc. cit. we see that
$\theta$ gives $\ad u v$. This proves the result.
 \end{proof}
 
\section{Differential geometry}
\label{sec:DG}

\subsection{The Chern connection on the endomorphism  bundle} 
\label{ssec:ChernConn} 
  
  Let $\mathsf D$ be the Chern connection on the endomorphism bundle. In \cite[\S 5]{pdmixed} we calculated it for the bundles $\cU^{(p)}$ and
  found
  \[
  \mathsf D =\bar \del  + \del - \tau_-^*.
  \]
  We already calculated   
  $\tau_-= \pi^{(p)}  \pi_+(\bar u)$ and so $\tau_- ^*= \pi^{(p)}  (\pi_+(\bar u))^*$.
 By functoriality this holds also for the endomorphism bundle using   the adjoint action, where we apply it for $\cU^{(-1)}$.   Since  in this  situation  $\pi^{(-1)}$ is the same
  as projection onto $\gq$,
  we get:
\begin{align}
  \mathsf D _\xi  \eta     & =  \del_\xi   v     - \pi_\gq [(\pi_+ \bar u) ^*,v]   ,\quad u= F_*\xi,\, v= \eta(s) .   \label{eqn:OnMetConn}
\end{align}

 \subsection{Curvature and plurisubharmonicity of   Hodge norms}
 \label{ssec:Curv}
 
 In contrast to the pure case, the  biholomorphic bisectional curvature of the horizontal tangent bundle is not always semi-negative
 as expressed by the following theorem.
   
 \begin{prop}The bisectional curvature of the Hodge metric  in unit directions $u,v\in U^{-1}\geg_F $ equals
  \begin{align*}
   K(u,v)   
  =      \|  [u^{-1,1} ,v] \|^2 +\| \pi_\gq [(\pi_+\bar u)^*  ,v] \|^2      - \| [\pi_+\bar u, v]\|^2    - \re  h( \pi_\gq [\pi_+ [u,\bar u] , v]).
\end{align*} \label{prop:BisectCurvMixed}
 \end{prop}
 
 \begin{proof} We use the curvature tensor for the Hodge metric $h$ as given in 
 \cite[Theorem 3.4]{pdmixed}:
\begin{align*}
  R _{h}  (u,\bar u) & =R_1+R_2+R_3 \\
  R_1& =  -[    \pi_\gq \ad{ (\pi_+ \bar u)^*}   ,  \pi_\gq
       \ad{(\pi_+\bar u }   ] \\
   R_2& =- \ad{\pi_0[u,\bar u]} \\
    R_3& =\pi_\gq \left(\ad{\pi_+[\bar u^*, u]} \right) +\pi_\gq \left(\ad{\pi_+[\bar u, u]} \right).
 \end{align*} 
 To calculate  $K(u,v)$ from this, we follow  the proof 
  of \cite[Theorem 4.1]{pdmixed}  and  calculate  the terms $h(R_j v,v)$ for $j=1,2,3$ of the biholomorphic sectional curvature.
 With $\|  - \|= \|  - \|_F$  the Hodge norm on $\endo{H}$     we have 
 
    \begin{align*}
 h(R_1 v,v) &= \underbrace{ -\| \pi_\gq [ \pi_+(\bar u),v]  \|^2 }_{A_1}  +
 \underbrace{\| \pi_\gq [(\pi_+\bar u)^*  ,v] \|^2}_{A_2} \\
  h(R_2 v,v)   & = - h([ \pi_0 [u,\bar u],v],v)= A_3\\
h(R_3 v,v)&=  -\re  h( \pi_\gq [\pi_+ [u,\bar u] , v] ,  v) .
 \end{align*} 
 To calculate  $A_3$ remark that 
  $\pi_\gq(\ad{\pi_0[ u, \bar u]}) =\ad{ [u^{-1,1},  (u^{-1,1})^*]}$ and so
\[
 A_3=h(R_2 v,v) = \|[u^{-1,1},v]\|^2-\|[(u^{-1,1})^*,v]\|^2.  
 \]
Next, observe that   $ [(u^{-1,1})^ *,v] \in U^0_F$    and $\pi_\gq [\pi_+\bar u, v]\in  U^{-1}_F $ have different bidegrees 
and hence are   mutually orthogonal   
with sum equal to  $  [\pi_+\bar u, v]$. Consequently, 
\[
- \|[ (u^{-1,1})^  *,v]\|^2 \underbrace{-\|\pi_\gq [\pi_+\bar u, v]\|^2}_{A_1}   =- \|  [\pi_+\bar u, v]\|^2.   
 \]
 The result follows.
 \end{proof}
 
We consider now Eqn.~\eqref{eqn: plurisubh} in the present situation. As a consequence of Eqn.~\eqref{eqn:OnMetConn} and 
 Proposition~\ref{prop:BisectCurvMixed}, we have 
 \begin{equation*} 
\left \{ 
\begin{array}{rcl}
  \del_{u} \del_{\bar{u} } \|v\|^2  &=&\|  \mathsf D _u \eta \| ^2  - K (u,v) \\
 &    =&  \| \del_u   v (s) +\pi_\gq [\pi_+ \bar u^*,v]\|^2 -  \|  \pi_\gq [\pi_+  \bar u^*,v] |^2 +   \|[\pi_+\bar u, v] \|^2 \\
& &\hspace{1em}  -\|[u^{-1,1},v]\|^2    + \re  h( \pi_\gq [\pi_+ [u,\bar u] , v],v)  .
\end{array}
 \right.
\end{equation*}
If  $ \pi_\gq [\pi_+  \bar u^*,v]$ and   $v$ are orthogonal,
this simplifies to give 
\begin{align}
\label{eqn:PSHform}
\begin{array}{rcl}
 \del_{u} \del_{\bar{u} } \|v\|^2 &=&\| \del_\xi   v (s) \|^2 +     \|[\pi_+\bar u, v] \|^2   \\
 && \quad - \left(
  [u^{-1,1},v]\|^2   + \re  h( \pi_\gq [\pi_+ [u,\bar u] , v]),v \right).
\end{array}
\end{align} 

A direct consequence of  \eqref{eqn:PSHform} and Lemma~\ref{lem:FlatDefs} gives  the following  variant of Proposition~\ref{prop:WhenFlatPure}  in  the mixed case:
\begin{prop}  \label{prop:OnPSH} Let there be given a graded polarized mixed  variation of Hodge structure    $(\cH, Q,\cF)$ 
over a quasi-projective complex manifold $S$. 
Let $\eta$  a holomorphic section of      $\cU^{-1}(\qendo{\cH})$.
For  $s\in S$,  let $v=\eta(s)$, viewed  as a horizontal tangent vector at $F=F(s)\in D$.
Suppose that for all    $u \in T^\hor _F D$ tangent to the image of the period map at all images  $F\in D$ of the period map,   one has
\begin{align}
  [ u^{-1,1} ,v ]        &= 0    \label{eqn:HarmFunct1}   \\
    h(\pi_\gq [\pi_+  \bar u^*,v],v) &=  0      \label{eqn:HarmFunct2}\\
    \re h( \pi_\gq [\pi_+ [u,\bar u] , v], v )&= 0 \label{eqn:HarmFunct3}   .
 \end{align}
 Then the function $\| v \|^2$
is plurisubharmonic and, if bounded (and hence constant), we have 
\begin{align}
\del_\xi    v & =[\pi_+\bar u, v]=0  . \label{eqn:HarmFunct4}
\end{align}
If, moreover,   $[u,v]=0$, $\eta$ is a flat section.

Conversely, if $\eta$ is flat, then  $\| v(s) \|$ is constant and \eqref{eqn:HarmFunct4} holds.
\end{prop}

 \begin{rmk} (1)  
In the cases of interest to us, flat sections are  bounded in the mixed Hodge norm. See Section~\ref{sec:norm-estimates}, although
this  is not the case in general as shown in Subsection~\ref{subsect:higher-normal-functions}.
 \\
 (2) In the pure case the   conditions $[u,v]=0$ and $[u^{1,1},v]=0$ are equivalent and the two remaining conditions hold for type reasons.
 \end{rmk}

 For easy   reference, a section $\eta$  with the property that for all tangent vectors $u$ along $S$ the conditions \eqref{eqn:HarmFunct1}--\eqref{eqn:HarmFunct2} hold,  is called \textbf{\emph{a pluri-sub\-harmonic endomorphism}}.

To give geometric examples where this phenomenon occurs, we first prove:

\begin{prop} In the situation of Proposition~\ref{prop:OnPSH},  assuming that $[u,v]=0$, the endomorphism $\eta$ is plurisubharmonic in the following cases:
\begin{enumerate}
\item the pure case;
\item $\bR$-split variations (e.g. two adjacent weights) in directions  $v =v^{-1,0}$;
\item in the setting of   unipotent variations  (i.e.  $u^{-1,1}=0$)  provided  either $ \Lambda^{-1,-1}=0$ and $ v=v^{-1,0}$,  or
$u \in  \Lambda^{-1,-1}$ and $v^{-1,1}=0$.  
\item variations with $u =u^{-1,1}+u^{-1,-1}$ in directions  $v=v^{-1,-1}$.
\item   two non-adjacent weights, say $0, k$, $|k|\ge 2$ with $h^{0,0}=1$, $h^{p,-p}=0$ for  $p\not=0$,
in directions   $ v= v^{-1, -k+1}$.

\item  A variation of type
\[
\xymatrix{  & I^{0,0} \ar[d]_{u^{-1,-1}} \ar[dr]^{u^{0,-2}}  \\
       I^{-2,0}  \ar[dr]_{u^{0,-2}} & \ar[l]  I^{-1,-1} \ar[l]_{u^{-1,1}}  \ar[l]_{u^{-1,1}}    \ar[d]^{u^{-1,-1}}   &\ar[l]_{u^{-1,1}}  I^{0,-2}    \\
 & I^{-2,-2}
}
 \]   
 in directions  $v=v^{-1,-1}$.
\end{enumerate}
In cases (1), (4)  and (5),  one has $K(u,v)\le 0$. 

In all cases, if  $\| \eta\|$ is bounded, then $\eta$ is parallel for the Gauss--Manin connection.

\label{prop:harmendo}
\end{prop}

\begin{proof}    The pure case  is Lemma~\ref{lem:PluriSubConseq}.   In the remaining cases we consider the conditions for $v$  to be pseudo-plurisubharmonic separately. Condition~\eqref{eqn:HarmFunct1} follows either trivially
since  $u^{-1,1}=0$ , or it follows from $[u,v]=0$  since    $u$ has  two Hodge types  while   $v^{-1,1}=0$. 
 
For   conditions \eqref{eqn:HarmFunct2} and  \eqref{eqn:HarmFunct3} we   write  
 \[
 u= \alpha+\beta +\lambda , \quad \alpha=u^{-1,1}, \quad \beta= u^{-1,0},  \quad \lambda =\pi_ \Lambda^{-1,-1} u .
 \]
 Observe that
 \begin{align*}
    \bar u  &= \alpha^* + \epsilon  + \delta , \\
    &\quad  \epsilon   =  \pi^{(0)} (\bar u)   \in \bigoplus _{q\ge 2}  \geg^{0,-q},\\ 
    &\quad  \delta= \pi_{+} (\bar u^{-1,0})\in \geg^{0,-1}  
 \end{align*} 
 and so 
 \begin{align*}
   \pi_\gq [ \pi_+[  u, \bar u] ,v] &= [[\beta,\alpha^*],v]+ [[\lambda,\alpha^*],v]  ,
   \\& \quad [\beta,\alpha^*] \in \geg^{0,-1},\\
   &\quad [\lambda,\alpha^*]\in  \bigoplus_{k\ge 2}   \geg^{0,-k} .
\end{align*}
First  consider condition \eqref{eqn:HarmFunct3}. In case (3)  one has
$  \pi_+ [u,\bar u]=0$. In case (2),  $\pi_+ [u,\bar u]= [\beta,\alpha^*] $ has bi-degree $(0,-1)$ and  so sends $v=v^{-1,0}$ to $0$.
In case (4)  and (6), $  \pi_+ [u,\bar u] = [\lambda,\alpha^*] $ has bi-degree $(0,-2)$ and so sends $v=v^{-1,-1}$ to zero.
In  case (5) $  \pi_+ [u,\bar u]$ has bi-degree $(0,-k)$
and so  sends  $v=v^{-1,1-|k|}$to zero.  

Next, consider \eqref{eqn:HarmFunct2} and remark that 
   \begin{align*}
  (\pi_+ \bar u)^* &= \alpha+ \epsilon^* + \delta^*, \\
 		 & \epsilon^*  \in \bigoplus _{q\ge 2}  \endo^{0,q},\quad 
		   \delta^* \in \endo^{0,1}.
 \end{align*}

 \begin{enumerate}[leftmargin=2em]
\item In the $\bR$-split case, $\epsilon=0$. In  $\pi_\gq [(\pi_+ \bar u) ^*,v]$ the terms of bi-degree   $(-1,1)$ come  from  $[\delta^*, v^{-1,0}]+ [\epsilon^*,\pi_ {\Lambda^{-1,-1}} v]$.
This proves \eqref{eqn:HarmFunct2} since then   $\pi_\gq [(\pi_+ \bar u) ^*,v] =[u^{-1,1}, v^{-1,0}] +[\delta^*, v^{-1,0}]$ has bi-degree $(-2,1)+(-1,1)$ and hence is orthogonal to $\del_\xi   v$ since it has bi-degree $(-1,0)$.

\item In the unipotent situation we also have $\epsilon=0$ and  now    $\pi_+ \bar u^* =   \delta^*$ which vanishes if $
u\in  \Lambda^{-1,-1}$ and else has pure type $(0,1)$. But then 
  $  \pi_\gq [(\pi_+ \bar u) ^*,   v^{-1,0} ]$ has bi-degree  $ (-1,1)$ and so is orthogonal to $ v=v^{-1,0} + v_ {\Lambda^{-1,-1}}$. 
  \item In   this case $\epsilon=0$ and $\delta=0$,
 we find that $(\pi_+ \bar u) ^*=\alpha$  so   that $\pi_\gq [(\pi_+ \bar u) ^*, v]=[u^{-1,1},v]=0$  which is condition \eqref{eqn:HarmFunct1} and  we just proved it.
  
 \item  Here  we show that $\pi_\gq [(\pi_+ \bar u) ^*, v]=0$ using:
   \begin{figure}[htbp]

\begin{center}
\begin{tikzpicture}[scale=0.35] [>=stealth]
\draw[step=0.3cm,   dotted] (-5 ,2)-- (6 ,2);
 \draw[fill=black]  (6,2) circle [radius=0.1cm];  \node at ( 6.3, 2.5)  {$ I^{ k,0  } $};

 \draw[fill=black]  (-3,2) circle [radius=0.1cm];  \node at ( -2.2, 2.5)  {$ I^{ 1, k-1  } $};
 \draw[fill=white]  (-5,2) circle [radius=0.1cm];  \node at ( -4.5, 2.5)  {$ I^{0, k} $};
  \draw[fill=white]  (4,-2) circle [radius=0.1cm];  \node at ( 5, -1.5)  {$ I^{k-1, 1-k} =0$};

\draw[fill=black]  (-5,-2) circle [radius=0.1cm];  \node at ( -4 , -1.5)  {$ I^{0,0} $};
  \draw [color=red, ->] (-5,-1.9) -- (-5,1.8);
 \draw [color=blue,->] (-3,1.9) -- (-4.9,-1.9);
\draw [color=blue,->] (5.9,1.9) -- (3.9,-1.9);

   \node at (-5,0)  {$a$}; 
  \node at (-3.3,0)  {$b$}; 
\node at (5 ,0)  {$b=0$}; 
  \end{tikzpicture}
 \end{center}
\end{figure}
\end{enumerate}

\begin{lemma} Let $a \in \endo^{0,k}$,  $b \in \geg^{-1,1-k  }$ and let $c= \pi_\gq (a\comp b) \in \geg^{-1,1}$. Suppose
  that  $h^{p,  q}=0$  unless $p+q=k\ge 1$ or $p=q=0$.
  Then $c=0$.
  \end{lemma}
  \begin{proof} Let $x\in I^{ 1,k-1 } $. Then $c(x) \in I^{ 0,k}$. 
  To show that $c(x)=0$ it suffices to show that it is orthogonal to $I^{0,k}$. Observe that every element   $y\in I^{0,k}$ is of the form
  $y=\bar z$ for some $z\in I^{k,0 }$ which is the case because of the assumption on the Hodge numbers.
  But  $\pm \ii ^k h(c(x), y)= Q(c(x), z)= - Q(x,c(z))=0$ since $b(z)=0$.
  \end{proof}
 We apply this lemma with     $a=\pi_+ \bar u^*= \epsilon^* \in \endo^{0,k}$, $b= v \in \geg^{-1,1-k}$.\\

\begin{enumerate}[leftmargin=2em,resume]
 \item The last case is clear from type considerations.

\end{enumerate}

For the assertion about the curvature, observe that the only term in the expression for $K(u,v)$ given in Proposition~\ref{prop:BisectCurvMixed} that causes trouble is 
$\pi_\gq [(\pi_+\bar u)^*  ,v]$ which, as we showed above,  vanishes in cases  (4)  and (5). 
 \end{proof}
 
 \subsection[Geometric examples of plurisubharmonic endomorphisms]{Horizontal plurisubharmonic endomorphisms: geometric examples}
 
 \label{ssec:PSH}

 We  indicate how  some of the geometric examples mentioned in the introduction fit in with the cases exhibited in 
 Proposition~\ref{prop:harmendo}.
  
  \begin{enumerate}[leftmargin=1.5em]
 
\item   \textbf{\emph{Normal functions.}}
We explain how to  interpret   a classical normal function as a variation of $\bZ$-mixed Hodge structure.
Suppose that $X=X_o$ is a smooth projective variety. A homologically trivial
algebraic $p$-cycle  $Z$ in $X$ canonically determines an extension 
\[
\nu_Z\in \ext^1 _ \mhs (\bZ(0),H_{2p+1}(X,\bZ(-p))).
\]
in the category of $\bZ$-mixed Hodge structures by pulling back 
 the exact sequence
\[
0 \to H_{2p+1}(X,\bZ(-p)) \to H_{2p+1}(X,Z,\bZ(-p)) \to H_{2p}(Z,\bZ(-p))
\to \cdots
\] 
 along the inclusion $\bZ(0) \into H_{2p}(Z,\bZ(-p))$  sending $1$  to the class of $Z$. It is well known (cf. \cite{carlson})
 that 
\[
\ext^1 _ \mhs (\bZ(0),H_{2p+1}(X,\bZ(-p))) \simeq J^p(X),
\] 
the intermediate Jacobian of $X$. The point in $JH_{2p+1}(X,\bZ(p))$ corresponding to the cycle $Z$
under this isomorphism is $\int_\Gamma$, where $\Gamma$ is a
real $2p+1$ chain that satisfies $\del \Gamma = Z$.

If $X=X_o$ varies in a smooth family $X_{s}$ with smooth base $S$, say $\pi:X\to S$, the groups $H_{2p+1}(X_s,\bZ(-p))$ form a local system $\underline H_{2p+1}(-p)$ defining a variation of Hodge structure. The intermediate
Jacobians vary holomorphically, and glue together to give  the relative
intermediate Jacobian $J^p(X/S)$.   

Suppose that $Z$ is an algebraic cycle in
$X$ which is proper over $S$ of relative dimension $p$ and such that $Z_s$ the fiber over $s\in S$ is homologous to zero.
Then $Z_s$ defines a point $\nu_{Z_s}$ 
in the intermediate Jacobian $J^p(X_s)$. These give a holomorphic section $\nu_Z$ of $J^p(X/S)$,
and this is the classical normal function. It can be viewed as an extension
\[
\ext^1 _ \vmhs (\bZ(0),  \underline H_{2p+1}(-p))   
\]
in the category of variations of mixed Hodge structures. 
Such a variation has two adjacent weights $0,-1$ and by case (2) of Proposition~\ref{prop:harmendo}, $\|v^{-1,0}\|$ is plurisubharmonic. 
For this example  the term  $\pi_\gq [(\pi_+\bar u)^*  ,v]$  need not vanish and so we cannot conclude 
  from Proposition~\ref{prop:BisectCurvMixed}  that  $K(u,v)\le 0$. However, a more sophisticated argument as in \cite[proof of Prop. 6.2]{pdmixed} reveals that $K(u^{-1,0} ,v^{-1,0})\le 0$.  

 \item   \textbf{\emph{Hodge--Tate variations.}} Only extension data can be deformed. These are deformations with $v=v^{-1,-1}$.
 Case (3) of
 Proposition~\ref{prop:harmendo} shows  that  $\|v^{-1,-1}\|$ is harmonic. Of course the biholomorphic curvature is $0$ since $D$ is flat.
As a simple example of a $1$-parameter variation, suppose
$h^{-1,-1}= 2, h^{0,0}=1$ and let $\set{e_1,e_2,e_3}$ be a basis of the
lattice $H_{\mathbb Z}$.
Let $F_o$ denote the reference filtration  
 \[
  I^{0,0}_{(F_o,W)} = \bC e_1,\qquad I^{-1,-1}_{(F_o,W)} = \bC e_2\oplus\bC e_3.
 \]
Then  the period domain $D=G/G^{F_o}$ is isomorphic to the unipotent group $U_\bC$
consisting of the matrices 
\warn{$$ g_{a,b}= \begin{pmatrix}
 1& 0& 0 \\
 a &1 &0 \\
 b&  0 &1 
  \end{pmatrix},\qquad a,b\in \bC
  $$ }%
via the action of $U_\bC$ on $F_o$.
Consider the period map  $\bC^*\to \Gamma\backslash \bC^2$ given by
$u\mapsto g_{\log u , 0}.F_o$ and with  monodromy group $\Gamma$ the
 unipotent group consisting of elements $g_{a,0}\in G$, $a\in \bZ$.
 This variation clearly has a deformation   leading to a variation over
 $\bC^*\times \bC$ given by
 the map $(u,v)\mapsto g_{\log u, v}.F_o$.
 
 Contrast this with the following example of a biextension of Hodge Tate type
 with Hodge numbers $h^{0,0}= h^{-1,-1}=h^{-2,-2}=1$.
 Let  $\set{e_1,e_2,e_3}$ be a basis of $H_{\mathbb Z}$.  Let $F_o$ denote
 the reference filtration such that  
\warn{
\[
     I^{-2,-2}_{(F_o,W)} = \bC e_3,\qquad I^{-1,-1}_{(F_o,W)} = \bC e_2,\qquad
     I^{0,0}_{(F_o,W)} = \bC e_1
\] 
}
The period domain $D=G/G^F$ is isomorphic to the unipotent group $U_\bC$
consisting of matrices of the form
 \warn{ \begin{align*}
 g_{a,b,c}= \begin{pmatrix}
 1& 0 & 0\\
 a & 1 & 0\\
 c&b&1 
 \end{pmatrix},\quad a,b,c \in \bC 
 \end{align*}}%
by the action of $U_\bC$ on $F_o$.  Let $E_{ij}$ denote the
$3\times 3$ matrix whose only non-zero entry is $1$ in row $i$ and
column $j$.  Then, the Lie algebra of $U_\bC$ is equal to
$\geg^{-1,-1}\oplus\geg^{-2,-2}$ where $\geg^{-1,-1}$ is spanned by
$u_0 = E_{21}$ and $u_1 = E_{32}$ while $\geg^{-2,-2}$ is spanned by
$u_2 = E_{31} = [u_1,u_0]$.

Now a period map can be  given locally as $z \mapsto \exp (\Gamma (z)) . F_o$
 where 
 \[
 \Gamma(z) = f_0 u_0 + f_1 u_1 + f_2 u_2 \implies
  \exp (\Gamma(z) )=\begin{pmatrix}
  1&  0 &  0\\
  f_0 & 1 & 0 \\
  f_2 + \half f_0 f_1 & f_1 &1
 \end{pmatrix}.
 \]
If it is injective  we may assume that $f_0=z$. The commutativity condition for horizontal directions
  gives $df_0\wedge df_1=0$ and so
$f_1= \varphi(z)$ for some function $\varphi$ with  $ \varphi(0)=0$. 
 The   horizontality condition gives $ df_2+\half( f_1df_0+f_0df_1)=0$ and so $f_2=\psi(z)$ for some function $\psi$ with  $ \psi(0)=0$. This implies that a non-trivial injective period map has a curve as its image
 and  hence must be  rigid. 
 
 As a concrete example we take the Hodge--Tate variation associated to the dilogarithm. Here 
 $S=\bP^1\setminus\set{0,1,\infty}$ with   
 global coordinate $s$. The period map
 \[
 \bP^1\setminus\set{0,1,\infty} \to U_\bZ\backslash D
 \] is then given by  the functions
  $f_0(s)= -(\log 2+\log(1-s))$ and $f_1(s)= \log 2 +\log s$ which vanish at $s=\half$.
  The horizontality condition gives
  \[  f_2(s) =-\half  \int_{\half}^s\left( \frac{\log t}{1-t} + \frac{\log (1-t) }{ t}  \right) dt =-\half \text{Li}_2 s+\half \text{Li}_2(1-s) .
 \]

  \item   \textbf{\emph{Variations  of mixed Hodge structures attached to  fundamental groups.}} 

    Let us briefly explain which variations we are considering. Let $X$ be a
    smooth algebraic variety and let
 $J_x$  be  the kernel of the ring homomorphism $\bZ\pi_1(X,x)\to \bZ$ given by 
 $\sum n_\gamma \gamma \mapsto  \sum n_\gamma$,  $\gamma\in \pi_1(X,x)$. 
There are mixed Hodge structures on $J_x/J_x^n$   which depend on the base point $x\in X$. 
For $n=3$ these can be explicitly described, following
 \cite[Section 6]{hainfunda}: the  mixed Hodge structure on the dual,  $\hom_\bZ(J_x/J_x^3,\bC) $ is an extension
\[
0\to H^1(X)\to \hom_\bZ(J_x/J_x^3,\bC) \mapright{p}  \ker(H^1(X)\otimes H^1(X)\to H^2(X)) \to 0,
\]
provided  $H_1(X)$ is torsion free.
Here we want pure Hodge structures and this forces $H^1$ to be of pure weight $\ell=1$ or $2$ and weight $H^2=2\ell$.
Geometric examples include $X$ smooth projective or $X=\bP^1 \setminus \Sigma$, $\Sigma$ a finite set of points. The extension depends on $x$, but
the  two pure Hodge structures remain fixed so that $u^{-1,1}=0$ and we are in the unipotent situation with $v=v^{-1,-\ell+1}$, $u=u^{-1,-\ell+1}$.
If $\ell=1$ we have $ \Lambda^{-1,-1}=0, v=v^{-1,0}$ and if $\ell=2$, $u,v\in  \Lambda^{-1,-1}$ and so case (3) of
 Proposition~\ref{prop:harmendo} shows  that $\|v^{-1,-\ell+1}\|$ is 
plurisubharmonic.  One can  directly verify that also $K(u^{-1,-\ell+1},v^{-1,-\ell+1})\le 0$.

 \item
   \textbf{\emph{Nilpotent orbits associated to  K\"ahler classes.}}
 As explained in the introduction, these variations have Hodge types  $(-1,1)$ and $(-1,-1)$. 
  However, $v$ can a priori have any type $(-1,k)$, $k\le 1$.
By case (4) of Proposition \ref{prop:harmendo},   endomorphisms  for which   $v= v^{-1, -1} $  are plurisubharmonic and  
then $K(u,v)\le 0$.  Note that for a family of projective manifolds over a quasi-compact 
base $S$ we can assume that we have a variation of  integral Hodge structures 
  polarized by a family of independent flat  integral K\"ahler classes (corresponding to ample divisors).
 
 \item \textbf{\emph{Higher normal functions.}}   

Let   $\pi: X\to S$ be  a smooth projective family.
Recall (see the introduction) that a higher normal function is  an extension in   
\[\ext^1 _ \vmhs   (\bQ(0),  R^{p-1}\pi_*\bQ(q)  ), \quad w=p-2q-1<0.
\]
Case (5) of   Proposition~\ref{prop:harmendo} tells us that  $\|v^{-1,w+1}\|$ is plurisubharmonic and  
  $K(u, v^{-1,w+1})\le 0$.
  
\item \textbf{\emph{Biextensions of  bidegrees  $(0,0), (-2,-2), (-2,0), (-1,-1),(0,-2)$.}}  
Case (6) of  Proposition~\ref{prop:harmendo}. shows  that $\|v^{-1,-1}\|$ is plurisubharmonic.
Geometric examples arise as a special case of a more general construction
given by  J. Burgos Gill, S. Goswami and the first author in \cite{hheights}, two higher Chow cycles   in $\cZ^p(X,1)$  
on a $d$-dimensional variety $X$
with $p+q=d+2$ determine in a canonical way  a special type of mixed Hodge structure. For a family of  surfaces, we have $d=2$
and the resulting variation  is of  biextension type with bidegrees  $(0,0), (-2,-2), (-2,0), (-1,-1),(0,-2)$.  
For more details on this example see Section~\ref{subsect:higher-cycles}. 
 
  \end{enumerate} 
  
 \section{Norm Estimates}
\label{sec:norm-estimates}

\par Let $\cH\to\Delta^*$ be an admissible variation of
graded-polarized mixed Hodge structure over the punctured disk $\Delta^*$
with unipotent monodromy $T=e^N$.  Recall that
$\geg(\cH)\subset\cH\otimes\cH^*$ is 
the sub-variation of mixed Hodge structure generated by local sections
which preserve $W$ and induce infinitesimal isometries on $\gr^W$.
In this section, we show that in the cases enumerated below, the mixed Hodge
norm of a monodromy invariant section of $\geg(\cH)$ is bounded.
In section \eqref{subsect:higher-normal-functions}, we show that
$\|N\|$ can be unbounded on $\Delta^*$ for higher normal functions.

\par In section 12 of~\cite{degmhs}, K. Kato, C. Nakayama and S. Usui
prove mixed Hodge norm estimates using their $\slgr 2$-orbit theorem.  However,
the metric used in~\cite{degmhs} involves an artificial twisting of the Hodge
metric on each $\gr^W$, and hence is different than the metric used in this
paper.  In~\cite{asdeg}, The first author and Tatsuki Hayama construct an
intrinsic \lq\lq twisted metric\rq\rq{} on $D$ which which gives the same norm
estimates as~\cite{degmhs} for admissible variations for which the limit MHS
is not split over $\bR$.  The twisted metric considered in~\cite{asdeg}
is only invariant under $G_{\bR}$, i.e. $g\in\exp(\Lambda^{-1,-1}_{(F,W)})$
need not induce an isometry $L_{g*}:T_F(D)\to T_{g.F}(D)$ by left translation.
For this reason, the curvature computations of~\cite{pdmixed} do not apply
to this metric on $D$.

\par The material in this section assumes familiarity with the definition
and basic theory of admissible variations of mixed Hodge structure as
outlined in Appendix \eqref{sec:admis}.

\par Before continuing, we emphasize that if $(F,W)$ is a graded-polarized
mixed Hodge structure with underlying vector space $V$ and
\begin{equation}
  g\in G_{\bR}\cup\exp(\Lambda^{-1,-1}_{(F,W)}),\qquad
  \alpha\in \gll V,
  \label{eq:isom-group}
\end{equation}
then
\begin{equation}
      \|\alpha\|_{(g.F,W)} = \|g^{-1}.\alpha \|_{(F,W)}
      \label{eq:isom-eq}
\end{equation}
where $g.\alpha = \text{\rm Ad}(g)\alpha$.  Indeed, if $\{v_j\}$ is an
unitary frame with the respect to the mixed Hodge metric $h_{(F,W)}$ then
$\{gv_j\}$ is a unitary frame for $h_{(g.F,W)}$.  Therefore,
\[
\aligned
      \|\alpha\|_{(g.F,W)}^2
      &= \sum_j\, h_{(g.F,W)}(\alpha(gv_j),\alpha( gv_j)) \\
      &= \sum_j\, h_{(F,W)}(g^{-1}\alpha(gv_j),g^{-1}\alpha(gv_j)) \\
      &= \sum_j\, h_{(F,W)}((g^{-1}.\alpha)(v_j),(g^{-1}.\alpha)(v_j)) 
       = \|(g^{-1}.\alpha)\|_{(F,W)}
\endaligned
\]
In particular, if $\cH$ is a variation of type $\text{\rm I}$ or
$\text{\rm II}$ as defined in section \eqref{subsect:nf-biext} it will not be
the case that $\geg(\cH)$ is type $\text{\rm I}$ or
type $\text{\rm II}$.  Nonetheless, all of the calculations in section
\eqref{subsect:nf-biext} depend only on \eqref{eq:isom-eq} and a version
of the $\slgr 2$-orbit theorem for nilpotent orbits of type $\text{\rm I}$
or $\text{\rm II}$.

\par By way of notation $z=x+\ii y$ throughout this section.  In the several
variable case $z_j = x_j + \ii  y_j$.

\subsection{Variations of Pure Hodge Structure}\label{subsect:pure}

\par Let $\cH\to\Delta^*$ be a variation of pure Hodge structure over
the punctured disk with unipotent local monodromy.  By Corollary
$(6.7)$ of~\cite{schmid}, the Hodge norm of an invariant class
is bounded.  This result is a consequence of Schmid's $\slgr 2$-orbit
theorem~\cite{schmid}.  If $\cH$ is a variation of pure Hodge structure
then so is $\cH\otimes\cH^*$, and hence if
$\alpha\in\cH\otimes\cH^*$ is monodromy invariant then
$\|\alpha\|$ has bounded Hodge norm.

\par For future use, we recall that by the Monodromy
Theorem (see~Theorem $(6.1)$,~\cite{schmid}) if
$\mathcal H\to\Delta^*$ is a variation of pure Hodge structure with
unipotent monodromy $T=e^N$, then $N^{\ell} = 0$ where $\ell$ is the maximum
number of successive non-zero Hodge summands of $\mathcal H$ (e.g. for a family of
curves of positive genus, $\ell=2$ since $\mathcal H = H^{1,0}\oplus H^{0,1}$).

\par Applied to a variation of graded-polarized mixed Hodge structure
$\cH\to\Delta^*$ with unipotent monodromy $T=e^N$, it follows that
if $\gr^W_{2p}(\cH)$ is pure of type $(p,p)$ then $N$ acts trivially
on $\gr^W_{2p}$.

\subsection{Local Normal Form}\label{subect:local-normal-form}
\par In connection with deformations of mixed period maps and the derivation
of the norm estimates below, we recall the following
(eq. (2.5)~\cite{luminy-hodge}, eq. (6.8)~\cite{higgs}).  

\par Let $(s_1,...,s_{a+b})$ be holomorphic coordinates on the polydisk
$\Delta^{a+b}$ and $\Delta^{*a}\times\Delta^b$ denote the complement of
the divisor $s_1\cdots s_a=0$.  Let $U^a$ denote the $a$-fold product
of the upper half-plane with Cartesian coordinates $(z_1,\dots,z_a)$
and covering map
$
      U^a\times\Delta^b\to\Delta^{*a}\times\Delta^b
      \subset\Delta^a\times\Delta^b
$
given by the formula      
\[
      (z_1,\dots,z_a;s_{a+1},\dots,s_b)\mapsto
      (e^{2\pi \ii z_1},\dots,e^{2\pi \ii z_a},s_{a+1},\dots,s_b)
\]      
Let $\cH$ be an admissible variation of graded-polarized mixed Hodge
structure over $\Delta^{*a}\times\Delta^b$ with unipotent monodromy
$T_j=e^{N_j}$ about $s_j=0$.  Then, (cf. \eqref{eq:admissibility-1}),
admissibility implies that the period map of $\cH$ can be lifted to a
holomorphic, horizontal map of the form 
\begin{equation}
     F(z_1,\dots,z_a;s_{a+1},\dots,s_b) = e^{\sum_j\, z_j N_j}.\psi(s)
      \label{eq:admissibility-3}
\end{equation}
where $\psi(s)$ is a holomorphic map $\Delta^{a+b}\to\check{D}$ with
$\psi(0)=F_{\infty}$.

\par To continue, define:
\begin{equation}
       \mathcal C = \{\,\sum_j\, \lambda_j N_j \mid
                         \lambda_1,\dots,\lambda_a>0\,\} \label{eq:cone-1}  
\end{equation}
Then, by admissibility, there exists an increasing filtration $M(\mathcal C,W)$
such that if $N\in\mathcal C$ then $M(N,W)$ equals $M(\mathcal C,W)$.
The results of Kashiwara show that if $\cH$ is admissible then
$(F_{\infty},M)$ is a mixed Hodge structure relative to which each $N_j$ is
a $(-1,-1)$-morphism.  Moreover, if $\geg_{\bC}$ is the Lie algebra
attached to the period map \eqref{eq:admissibility-3} of $\cH$, then
$(F_{\infty},M)$ induces a graded-polarizable mixed Hodge structure on
$\geg_{\bC}$.

\par In particular, if
$\geg_{\bC} = \oplus_{p,q}\,\geg^{p,q}$ is the Deligne
bigrading induced by $(F_{\infty},M)$, then 
\begin{equation}
      \gq = \bigoplus_{p<0}\,\geg^{p,q}
       \label{eq:admissibility-4}
\end{equation}
is a vector space complement to the stabilizer
$\geg_{\bC}^{F_{\infty}}$ in $\geg_{\bC}$.  Therefore,
after shrinking $\Delta^{a+b}$ as needed, it follows that there exists a
unique $\gq$-valued holomorphic function $\Gamma(s)$ which
vanishes at $0$ such that
\begin{equation}
     \psi(s) = e^{\Gamma(s)}.F_{\infty}    \label{eq:admissibility-5}
\end{equation}  

\par Let $\Gamma_{-1} = \sum_q\,\Gamma^{-1,q}(s)$.  By equation $(6.14)$ and
Theorem $(6.16)$ of \cite{higgs}, the function $\Gamma_{-1}(s)$ satisfies the
following integrability condition
\begin{equation}
  \left[N_j + 2\pi \ii s_j\frac{\partial\Gamma_{-1}}{\partial s_j},
        N_k + 2\pi \ii s_k\frac{\partial\Gamma_{-1}}{\partial s_k}\right]=0
  \label{eq:admissibility-6}
\end{equation}
for all $j$ and $k$ (with $N_{\ell}=0$ for $\ell>a$).

\par Conversely, given an admissible nilpotent orbit
\[
       \theta(z_1,\dots,z_a) = e^{\sum_j\, z_j N_j}.F_{\infty}
\]
and a holomorphic function
$\Gamma_{-1}:\Delta^{a+b}\to\oplus_q\,\geg_{\bC}^{-1,q}$ which
vanishes at zero and satisfies the integrability condition
\eqref{eq:admissibility-6}, there exists a unique holomorphic function
$\Gamma:\Delta^{a+b}\to\gq$ which vanishes at 0 such that 
\begin{equation}
     F(z_1,\dots,z_a;s_{a+1},\dots,s_b)
     = e^{\sum_j\, z_j N_j}e^{\Gamma(s)}.F_{\infty}   \label{eq:admissibility-7}
\end{equation}
arises from the period map of a variation of graded-polarized mixed Hodge
structure defined for $\Im(z_1),\dots,\Im(z_a)\gg 0$ and
$s_{a+1},\dots,s_b\sim 0$ with $\Gamma_{-1} = \sum_q\,\Gamma^{-1,q}$.
We call \eqref{eq:admissibility-7} the \textbf{\emph{local normal form}} of the period
map.

\begin{rmk}
A published version of \eqref{eq:admissibility-6} and \eqref{eq:admissibility-7}
for variations of pure Hodge structures appears in~\cite{CF}.  The key point
is that the reconstruction of $\Gamma$ from $\Gamma_{-1}$ is really a
statement about the horizontal distribution, and hence applies equally well
to the mixed case.  The full mixed case appears in~\cite{higgs}.
\end{rmk}

\subsection{Hodge--Tate Variations}\label{subsect:hodge-tate}

\par In section~\eqref{subsect:unipot} it will be shown that if
$\cH\to\Delta^*$ is a unipotent variation of mixed Hodge structure
in the sense of R. Hain and S. Zucker then any flat section of $\cH$
has bounded mixed Hodge norm.  In this section, we prove the following
several variable result:

\begin{thm}\label{hodge-tate-estimate} Let
$\cH\to\Delta^{*a}\times\Delta^b$ be an admissible Hodge--Tate
variation with unipotent monodromy $T_j = e^{N_j}$ about $s_j=0$.
Let $v$ be a flat section of $\cH$.  Then, $v$ has bounded
Hodge norm $\|v \|$ with respect to the mixed Hodge metric of
$\cH$. Likewise, if $\alpha$ is a flat section of
$\cH\otimes\cH^*$ has bounded mixed Hodge norm
$\|\alpha\|$.
\end{thm}
\begin{proof}  By Prop. (2.14),~\cite{SZ} if $N$ acts trivially on $\gr^W$
then $M=M(N,W)$ exists if and only if $N(W_{\ell})\subseteq W_{\ell-2}$ for
all $\ell$, wherefrom $M=W$.

\par To continue, recall $\cH$ is Hodge--Tate means $\cH^{p,q}=0$
if $p\neq q$.  Therefore, by the Monodromy theorem discussed at the end of
\eqref{subsect:pure}, it follows that $N$ acts trivially on $\gr^W$.

\par In particular, it follows from the previous paragraph that each
$N_j = \log(T_j)$ acts trivially on $\gr^W$.  By admissibility, it
follows there exists a fixed increasing filtration $M=M(\mathcal C,W)$ such
that $M=M(N,W)$ for each $N\in\mathcal C$.  As each $N\in\mathcal C$ acts
trivially on $\gr^W$ it follows that $M=W$ and $N(W_{\ell})\subset W_{\ell-2}$
for each $\ell$.  Accordingly, since $\mathcal C$ consists of arbitrary
positive linear combination of $N_1,\dots,N_a$ it follows
that $N_a(W_{\ell})\subset W_{\ell-2}$ for each index $\ell$.

\par Since $M=W$, it follows that $F_{\infty}\in D$ and hence 
$\psi(s)$ also takes values in $D$.  Moreover, since $D$ classifies
Hodge--Tate structures it follows that for any $F\in D$,
\begin{equation}
  W_{-2}\geg_{\bC} = \bigoplus_{p<0}\,\geg_{(F,W)}^{p,p}
  = \Lambda^{-1,-1}_{(F,W)}
       \label{eq:hodge-tate-1}
\end{equation}
and hence $N_1,\dots,N_a\in\Lambda^{-1,-1}_{(F,W)}$.

\par Relative to the fixed reference fiber $H$ of $\cH$ used to
define the period map into $D$, a flat section of $\cH$
corresponds to an element of $H$ which is contained in $\bigcap_j\ker(N_j)$.
Thus, by equation \eqref{eq:hodge-tate-1},
\[
\begin{aligned}
    \|v\|_{(F(z;s),W)}
    &= \|v\|_{(e^{\sum_j\, z_jN_j}e^{\Gamma(s)}.F_{\infty},W)} \\
    &= \|e^{-\sum_j\, z_j N_j}v\|_{(e^{\Gamma(s)}.F_{\infty},W)} \\
    &= \|v\|_{(e^{\Gamma(s)}.F_{\infty},W)}
\end{aligned}
\] 
where the middle step is justified by the fact that $\sum_j\, z_j N_j$
belongs to $W_{-2}\geg_{\bC}$ and equation \eqref{eq:hodge-tate-1}.
As $e^{\Gamma(s)}.F$ takes values in a compact subset of $D$, it follows from
the last line of the equation that $\|v\|_{(F(z;s),W)}$ is bounded.
The proof for the case of $\cH\otimes\cH^*$ is identical,
except that $\alpha\in\cap_j\,\ker(\ad N_j)$ and 
$e^{-\sum_j\, z_j N_j}v$ is replaced by $e^{-\sum_j\, z_j \ad N_j}\alpha$.
\end{proof}  




\subsection{Gradings and Splittings of Mixed Hodge Structures}
\label{subsect:gradings-splittings}

\par Let $V$ be finite dimensional vector space over a field of characteristic
zero.  Then, a grading of $V$ is a semisimple endomorphism $Y$ of $V$ with
integral eigenvalues.  In particular, a grading of $V$ determines an
increasing filtration
\begin{equation}
  W_k(Y) = \bigoplus_{\ell\in \lambda(Y), \ell\leq k}\, E_{\ell}(Y)
  \label{eq:eigen-decomp}
\end{equation}
where $\lambda(Y)$ is the set of eigenvalues of $Y$ and $E_{\ell}(Y)$ is the
$\ell$-eigenspace of $Y$.  If $W$ is an increasing filtration of $V$
we say $Y$ grades $W$ if $W(Y) = W$.  In particular, a mixed Hodge structure
$(F,W)$ determines a grading $Y_{(F,W)}$ which acts a multiplication by
$p+q$ on each non-zero summand $I^{p,q}$ of the Deligne bigrading of $(F,W)$.
\medskip


\par As discussed in Remark \eqref{rmk:split}, the mixed Hodge structure
$(F,W)$ is split over $\bR$ if $\overline{I^{p,q}} = I^{q,p}$.
Equivalently, $(F,W)$ is split over $\bR$ if
$\overline{Y_{(F,W)}} = Y_{(F,W)}$.  In general, we say a grading $Y$ is
defined over $\bR$ if $\overline{Y} = Y$.
%
      
\par By Proposition $(2.20)$ of~\cite{degeneration}, given a mixed Hodge 
structure $(F,W)$ there exists a unique, real element
\begin{equation}
     \delta\in\Lambda^{-1,-1}_{(F,W)}     \label{eq:delta-1}
\end{equation}
such that $(\hat F,W) = (e^{-i\delta}.F,W)$ is an $\bR$-split mixed
Hodge structure.  Moreover, $\delta$ commutes with all morphisms of $(F,W)$ 
\footnote{This means $(-k,-k)$-morphisms for any integer $k$.}.  We henceforth
call $(\hat F,W)$ the \emph{Deligne $\delta$-splitting} of $(F,W)$.

\par Suppose now that if $\theta(z)=e^{zN}.F$ is a nilpotent orbit of pure
Hodge structure of weight $k$ polarized by $Q$.  Let $W=W(N)[-k]$ and
$(\hat F,W) = (e^{-i\delta}.F,W)$ be the Deligne $\delta$-splitting of
$(F,W)$.  Then, by equation $(3.11)$ in~\cite{degeneration}, $\delta$
is an infinitesimal isometry of $Q$.  Likewise if $\theta(z) = e^{zN}.F$
is an admissible nilpotent orbit of mixed Hodge structure the Deligne
$\delta$-splitting of the limit mixed Hodge structure $(F,M)$ is given
by an element $\delta\in\geg_{\bR}$.  The proof of this
last statement boils down to showing the compatibility of Deligne's
construction with passage to $\gr^W$.

\par If $Y$ is a grading of $W$, $y>0$ and $\alpha\in\bR$ we
define
\[
     y^{\alpha Y} = \exp(\alpha\log(y)Y)
\]
wherefrom $y^{\alpha Y}$ acts on $\gr^W_k$ as multiplication by $y^{\alpha k}$.
Accordingly, if $\gamma$ belongs to the Lie algebra $\geg_{\bC}$
attached to a classifying space $D$ with weight filtration $W$ then
\[
     y^{\alpha Y}.\gamma = \text{\rm Ad}(y^{\alpha Y})\gamma
                       = y^{\alpha Y}\gamma y^{-\alpha Y}
\]    
induces the same action on $\gr^W$ as $\gamma$.  Therefore
$y^{\alpha Y}.\gamma\in\geg_{\bC}$.  If $Y$ is defined over
$\bR$ then the adjoint action of $y^{\alpha Y}$ preserves
$\geg_{\bR}$.
    
\subsection{Unipotent Variations of Mixed Hodge Structure}
\label{subsect:unipot}

\par Let $\cH$ be a variation of graded-polarizable mixed Hodge
structure over a smooth, complex algebraic variety $S$.  Then,
$\cH$ is said to be unipotent~\cite{HZ} if the global monodromy
representation of $\cH$ is unipotent.  Equivalently, the
variations of Hodge structure induced by $\cH$ on $\gr^W$ are
constant ((1.4),~\cite{HZ}).  The global structure of admissible unipotent
variations of mixed Hodge structure on $S$ is governed by mixed Hodge
theoretic representations of the fundamental group of $S$
(Thm. (1.6),~\cite{HZ}).

\par For the remainder of this section we assume $\cH\to\Delta^*$
is admissible and unipotent in the sense of Hain and Zucker.  We prove that
the mixed Hodge norm of a flat section of $\cH$ is
bounded.  

\par To begin, we note that in this case, we again have $M=M(N,W)$ equals
$W$ (see Prop. (2.14),~\cite{SZ} or (1.5),~\cite{HZ}).   Thus, as in
\eqref{eq:admissibility-5} and \eqref{eq:admissibility-1}, we can write
the lift of the period map of $\cH$ to the upper half-plane in
the form
\begin{equation}
  F(z) = e^{zN}e^{\Gamma(s)}.F_{\infty}      \label{eq:unipotent-1}
\end{equation}  
where $\Gamma(s)$ is a $\gq$-valued function which vanishes at
$s=0$.  

\begin{rmk}\label{rmk:uni-1} In the unipotent case, the function $\Gamma(s)$
takes values in the subalgebra $W_{-1}\gq$ of $\gq$
consisting of elements which act trivially on $\gr^W$.
\end{rmk}  

\par To continue, let $(\hat F_{\infty},M) = (e^{-i\delta}.F_{\infty},M)$
be Deligne's $\delta$-splitting \eqref{eq:delta-1} of $(F_{\infty},M)$,
keeping in mind that $M=W$.  Let
\begin{equation}
      Y = Y_{(\hat F_{\infty},M)}     \label{eq:unipotent-2a}
\end{equation}
and note that $\overline{Y}=Y$ since $(\hat F_{\infty},M)$ is split over
$\bR$.  Note that since $[Y,N] = -2N$ we have
\begin{equation}
      y^{-Y/2}e^{iN}y^{Y/2} = e^{iyN}     \label{eq:unipotent-2}
\end{equation}
Define
\begin{equation}
\aligned
       e^{\Gamma(s,y)} &= \Ad(e^{iN})\Ad(y^{Y/2})e^{\Gamma(s)} \\
       e^{i\delta(y)}  &= \Ad(y^{Y/2})e^{i\delta}
\endaligned\label{eq:unipotent-3}       
\end{equation}

\begin{lemma}\label{lem:unipotent-4} If $\cH\to\Delta^*$ is unipotent
in the sense of Hain and Zucker then\footnote{For clarity, since we are acting
on filtrations here, we are using the linear action and not the adjoint
action.}    
\begin{equation}
    F(z) = e^{xN}y^{-Y/2}e^{\Gamma(s,y)}e^{i\delta(y)}e^{iN}.\hat F_{\infty} 
    \label{unipotent-4}
\end{equation}
\end{lemma}
\begin{proof} Starting from \eqref{eq:unipotent-1}, we have
\begin{equation}
\aligned
    F(z) &= e^{xN}e^{iyN}e^{\Gamma(s)}.F_{\infty} \\
         &= e^{xN}e^{iyN}e^{\Gamma(s)}e^{i\delta}.\hat F_{\infty} \\
         &= e^{xN}y^{-Y/2}e^{iN}y^{Y/2}e^{\Gamma(s)}e^{i\delta}.\hat F_{\infty}
\endaligned\label{eq:unipotent-4a}
\end{equation}
To further refine \eqref{eq:unipotent-4a}, we note that $[N,\delta]=0$ as
$N$ is a $(-1,-1)$-morphism of $(\hat F_{\infty},M)$ and
$Y(\hat F_{\infty}^p) \subseteq \hat F_{\infty}^p$ since $Y = Y_{(\hat F_{\infty},M)}$.
Therefore,
\begin{equation}
\aligned  
     e^{i\delta}.\hat F_{\infty}
     &= y^{-Y/2}y^{Y/2}e^{-iyN}e^{iyN}e^{i\delta}.\hat F_{\infty} \\
     &= y^{-Y/2}e^{-iN}y^{Y/2}e^{iyN}e^{i\delta}.\hat F_{\infty}  \\
     &= y^{-Y/2}e^{-iN}y^{Y/2}e^{i\delta}e^{iyN}.\hat F_{\infty}  \\
     &= y^{-Y/2}e^{-iN}e^{i\delta(y)}y^{Y/2}e^{iyN}.\hat F_{\infty} \\
     &= y^{-Y/2}e^{-iN}e^{i\delta(y)}e^{iN}y^{Y/2}.\hat F_{\infty} \\
     &= y^{-Y/2}e^{-iN}e^{i\delta(y)}e^{iN}.\hat F_{\infty}
\endaligned\label{eq:unipotent-4b}
\end{equation}
Inserting \eqref{eq:unipotent-4b} into \eqref{eq:unipotent-4a} and
simplifying gives \eqref{unipotent-4}.
\end{proof}

\par Fix a norm $|*|$ on $\geg_\mathbb{\bC}$.  Observe that
since $\Gamma(0)=0$ it follows that $|\Gamma(s,y)|$ can be
bounded by a constant multiple of $|s|(-\log|s|)^b$.  Likewise, since
\[
     \delta\in\Lambda^{-1,-1}_{(F_{\infty},M)}
      =\Lambda^{-1,-1}_{(\hat F_{\infty},M)} 
\]
it follows that $\delta$ decomposes as
$\delta=\delta_{-2} + \delta_{-3} + \cdots$ relative to $\ad Y$ since
$Y = Y_{(\hat F_{\infty},M)}$.  Accordingly, by equation \eqref{eq:unipotent-3}
it follows that $|\delta(y)|$ can be bounded by a multiple of $1/y$.

\par To continue, let $F\in D$ and
$\gq_F = \oplus_{p<0}\, \geg^{p,q}_{(F,W)}$.  Then, since
$\gq_F$ is a vector space complement to $\geg_{\bC}^F$
in $\geg_{\bC}$ it follows from the inverse function theorem
that there exists a neighborhood $\mathcal N$ of 0 in $\geg_{\bC}$
and unique holomorphic functions $v:\mathcal N\to\gq_F$,
$\phi^{\dag}:\mathcal N\to\geg_{\bC}^F$ such that
\begin{equation}
      u\in\mathcal N \implies e^u = e^{v(u)}e^{\phi^{\dag}(u)}
      \label{eq:unipotent-4}
\end{equation}
In particular, by uniqueness $v(0) = \phi^{\dag}(0) = 0$.

\par On the other hand, by \cite{higgs} there exists a neighborhood
$\cQ$ of $0$ in $\gq_F$ and distinguished real analytic
functions $\tilde\gamma:\cQ\to\geg_{\bR}$,
$\tilde\lambda:\cQ\to\Lambda^{-1,-1}_{(F,W)}$ and
$\tilde \phi:\cQ\to\geg_{\bC}^F$ such that
\begin{equation}
  v\in\cQ\implies
  e^v = e^{\tilde\gamma(v)}e^{\tilde\lambda(v)}e^{\tilde\phi(v)}
    \label{eq:unipotent-5}
\end{equation}    
Combining \eqref{eq:unipotent-4} and \eqref{eq:unipotent-5} it follows
that after shrinking $\mathcal N$, we have a real-analytic decomposition
\begin{equation}
     e^u = e^{\gamma(u)}e^{\lambda(u)}e^{\phi(u)}  \label{eq:unipotent-6}
\end{equation}
upon setting $\gamma(u) = \tilde\gamma(v(u))$,
$\lambda(u) = \tilde\lambda(v(u))$ and
$\phi(u) = \tilde\phi(v(u))\phi^{\dag}(u)$.

\par Denote the dependence
of the functions appearing in \eqref{eq:unipotent-6} on $F$ by
$\gamma_F$, $\lambda_F$ and $\phi_F$.  Then, since the decomposition
\[
    V = \bigoplus_{p,q} I^{p,q}_{(F,W)}
\]
is $C^{\infty}$ with respect to $F\in D$, it follows that $\gamma_F$,
$\lambda_F$ and $\phi_F$ also have a $C^{\infty}$ dependence on $F$.
Accordingly, a soft analysis argument shows
that given $F_o\in D$ there exists a compact set $K\subset D$ containing $F_o$
and constants $\rho$ and $C$ such that
\begin{equation}
    F\in K,\quad |u|<\rho \implies
    |\gamma_F(u)|,\quad |\lambda_F(u)|,\quad |\phi_F(u)|< C|u|
    \label{eq:unipotent-7}
\end{equation}
For the remainder of this section, we will drop the subscript $F$ from
$\gamma$, $\lambda$ and $\phi$.

\par Let $F_o = e^{iN}.\hat F_{\infty}\in D$ with corresponding compact set
$K$ and constants $\rho$ and $C$ as in \eqref{eq:unipotent-7}.  Let
$F(y) = e^{i\delta(y)}.F_o$.  Then, by our previous estimates of
$\delta(y)$ and $\Gamma(s,y)$ it follows that there exists a constant
$a>0$ such that $|s|=e^{-2\pi y}<e^{-2\pi a}$ implies $F(y)\in K$ and
$|\Gamma(s,y)|<\rho$.  Therefore, by \eqref{eq:unipotent-7}
\begin{equation}
      e^{\Gamma(s,y)} = e^{\gamma(s,y)}e^{\lambda(s,y)}e^{\phi(s,y)}
      \label{eq:unipotent-8}
\end{equation}
relative to $F(y)$.

\begin{rmk}\label{rmk:uni-2} Since $\Gamma(s)$ takes values in
$W_{-1}\geg_{\bC}$, so does $\Gamma(s,y)$.  Therefore,
$\gamma(s,y)$, $\lambda(s,y)$ and $\phi(s,y)$ take values in the subalgebra
$W_{-1}\geg_{\bC}\subseteq\geg_{\bC}$ consisting
of elements which act trivially on $\gr^W$. 
\end{rmk}

\begin{thm}\label{unpotent-estimate} Let $\cH\to\Delta^*$ be a
unipotent variation of mixed Hodge structure and $v$ be a flat section
of $\cH$.  Then, the mixed hodge norm $\|v\|$ is bounded.
\end{thm}
\begin{proof} To reduce notation, we write the mixed Hodge norm with
respect to $(F,W)$ as $\|*\|_F$ since $W$ is fixed throughout
the proof.  Returning to equation \eqref{unipotent-4}, we have
\begin{equation}
\aligned  
  \|v\|_{F(z)}
  &= \|v\|_{e^{xN}y^{-Y/2}e^{\Gamma(s,y)}e^{i\delta(y)}e^{iN}.\hat F_{\infty}} \\
  &= \|e^{-xN}v\|_{y^{-Y/2}e^{\Gamma(s,y)}e^{i\delta(y)}e^{iN}.\hat F_{\infty}} \\
  &= \|v\|_{y^{-Y/2}e^{\Gamma(s,y)}e^{i\delta(y)}e^{iN}.\hat F_{\infty}} \\
  &= \|v\|_{y^{-Y/2}e^{\Gamma(s,y)}.F(y)}
  \endaligned\label{eq:unipotent-9}
\end{equation}
because $e^{xN}\in G_{\bR}$, $v\in\ker(N)$ and
$F(y)=e^{i\delta(y)}e^{iN}.\hat F_{\infty}$.  Inserting \eqref{eq:unipotent-8}
into \eqref{eq:unipotent-9} and noting that $e^{\phi(s,y)}$ preserves
$F(y)$ it follows that
\begin{equation}
\aligned  
  \|v\|_{F(z)}
   &= \|v\|_{y^{-Y/2}e^{\gamma(s,y)}e^{\lambda(s,y)}.F(y)} \\
   &= \|v\|_{y^{-Y/2}e^{\gamma(s,y)}y^{Y/2}y^{-Y/2}e^{\lambda(s,y)}.F(y)}  \\
   &= \|y^{+Y/2}e^{-\gamma(s,y)}y^{-Y/2}v\|_{y^{-Y/2}e^{\lambda(s,y)}.F(y)} \\
   &= \|\exp(-\Ad(y^{+Y/2})\gamma(s,y))v\|_{y^{-Y/2}e^{\lambda(s,y)}.F(y)}
\endaligned\label{eq:unipotent-10}
\end{equation}
since $\Ad(y^{-Y/2})$ preserves $\geg_{\bR}$ and $\gamma(s,y)$
takes values in $\geg_{\bR}$.  As noted after Lemma
\eqref{lem:unipotent-4},
$|\Gamma(s,y)|$ can be bounded by a constant multiple of $|s|(-\log|s|)^b$.
By \eqref{eq:unipotent-7}, at the price of adjusting the constant multiplier,
the same is true of $|\gamma(s,y)|$ and $|\lambda(s,y)|$.  Likewise, as
$Y$ is semisimple with a finite number of eigenvalues,  
$|\Ad(y^{+Y/2})\gamma(s,y)|$ can be bounded by $c . |s|(-\log|s|)^m$, $c\in \bC$.
Since $\Ad(y^{+Y/2})\gamma(s,y)$ takes values in a nilpotent Lie algebra
$W_{-1}\geg_{\bC}$, it follows that
\[
   \exp(-\Ad(y^{+Y/2})\gamma(s,y)) = \mathbb{1} + \epsilon(s,y)
\]
where $|\epsilon(s,y)|$ can be bounded by a constant multiple
of $|s|(-\log|s|)^m$ for $|s|$ sufficiently small.

\par To continue, we note that since $\lambda(s,y)$ takes values in
$\Lambda^{-1,-1}_{(F(y),W)}$, $Y=\bar Y$ and $\Ad(y^{-Y/2})$ acts on
$\geg_{\bC}$ it follows that
\[
    \Ad(y^{-Y/2})\lambda(s,y)\in\Lambda^{-1,-1}_{(y^{-Y/2}.F(y),W)}
\]
By construction,
\[
\aligned
     y^{-Y/2}.F(y)
     &= y^{-Y/2}e^{i\delta(y)}e^{iN}.\hat F_{\infty} \\
     &= y^{-Y/2}y^{Y/2}e^{i\delta}y^{-Y/2}e^{iN}.\hat F_{\infty} \\
     &= e^{i\delta}y^{-Y/2}e^{iN}y^{Y/2}y^{-Y/2}.\hat F_{\infty} \\
     &= e^{i\delta}e^{iyN}.\hat F_{\infty} 
\endaligned
\]
since $Y$ preserves $\hat F_{\infty}$.  Accordingly, since $\delta$ commutes
with $N$ we have
\[
        y^{-Y/2}.F(y) = e^{iyN}.F_{\infty}
\]
Putting the last three equations together, we have
\[       
        y^{-Y/2}e^{\lambda(s,y)}.F(y)
        = \exp(\Ad(y^{-Y/2})\lambda(s,y))e^{iyN}.F_{\infty}
\]
where $\Ad(y^{-Y/2})\lambda(s,y)\in\Lambda^{-1,-1}_{(e^{iyN}.F_{\infty},W)}$.

\par Returning to \eqref{eq:unipotent-10}, we have
\[
\aligned
     \||v\|_{F(z)}
  &= \|(1+\epsilon(s,y))v\|_{y^{-Y/2}e^{\lambda(s,y)}.F(y)} \\
  &= \|(1+\epsilon(s,y))v\|_{\exp(\Ad(y^{-Y/2})\lambda(s,y))e^{iyN}.F_{\infty}} \\
  &= \|\exp(-\Ad(y^{Y/2})\lambda(s,y))
     (1+\epsilon(s,y))v\|_{e^{iyN}.F_{\infty}}
\endaligned
\]     
since $\Ad(y^{-Y/2})\lambda(s,y)\in\Lambda^{-1,-1}_{(e^{iyN}.F_{\infty},W)}$.

\par As above, $|\Ad(y^{Y/2})\lambda(s,y)|$ is bounded by a constant
multiple of $|s|(-\log|s|)^{m'}$.  Therefore,
\[
     \exp(-\Ad(y^{Y/2})\lambda(s,y)) = \mathbb{1} + \mu(s,y)
\]
where $|\mu(s,y)|$ can be bounded by a multiple of $|s|(-\log|s|)^{m'}$
for $|s|$ sufficiently small.  Thus,
\[
     \|v\|_{F(z)}
     = \|(1+\mu(s,y))(1+\epsilon(s,y))v\|_{e^{iyN}.F_{\infty}}
\]  
Finally, since $N\in\Lambda^{-1,-1}_{(F,W)}$ and $v\in\ker(N)$ we have
\[ 
\aligned
     \|v\|_{F(z)}
     &= \|e^{-iyN}(1+\mu(s,y))(1+\epsilon(s,y))v\|_{F_{\infty}} \\
     &=  \|e^{-iyN}(1+\mu(s,y))(1+\epsilon(s,y))e^{iyN}v\|_{F_{\infty}}
\endaligned     
\]
Therefore, since $|y| = \frac{-1}{2\pi}\log|s|$ it follows that
$$
      e^{-iyN}(1+\mu(s,y))(1+\epsilon(s,y))e^{iyN} \to \mathbb{1}
$$
as $y\to\infty$.  Thus, $\|v\|_{F(z)}$ is bounded.      
\end{proof}

\begin{rmk} If $\mathcal A$ and $\mathcal B$ are unipotent variations of
mixed Hodge structure then so is $\mathcal A\otimes\mathcal B$.  In
particular, we can apply the previous theorem to flat sections of
$\cH\otimes\cH^*$.
\end{rmk}
        

\subsection{Hodge theory of $\sll 2$-pairs}\label{sec:sl2-pairs}

\par By equation $(3.11)$ in~\cite{degeneration}, if $\theta(z)=e^{zN}.F$ is
a nilpotent orbit of pure Hodge structure of weight $k$ polarized by $Q$,
$W=W(N)[-k]$ and $Y = Y_{(F,W)}$ then $H=Y-k\,\mathbb{1}$ is belongs to the
complex Lie algebra of infinitesimal isometries of $Q$.  Likewise, if
$(\hat F,W)$ is the Deligne $\delta$-splitting of $(F,W)$ then
$\hat H = Y_{(\hat F,W)} - k\,\mathbb{1}$ belongs to the Lie algebra of
real infinitesimal isometries of $Q$.

\par As discussed in \S 2 of~\cite{degeneration}, given a nilpotent
element $N\in \gll  V $, there is a bijective correspondence between gradings
$H$ of $W(N)$ such that $[H,N]=-2N$ and representations
$\rho:\sll 2\to \gll  V $ such that
\begin{equation}
      \rho\begin{pmatrix} 0 & 0 \\
                          1 & 0 \end{pmatrix} = N,\qquad
      \rho\begin{pmatrix} 1 &  0 \\
                          0 & -1 \end{pmatrix} = H              
       \label{eq:sl2-pair}
\end{equation}
We call such a pair $(N,H)$ an $\sll 2$-pair, and $(N,H,N^+)$ the associated
$\sll 2$-triple, where
\[ 
      N^+ = \rho\begin{pmatrix} 0 & 1 \\
                                0 & 0 \end{pmatrix}
\]
If $N$ and $H$ are infinitesimal isometries of $Q$ then so is $N^+$.
Thus, by the previous paragraph, a nilpotent orbit $\theta$ of pure,
polarized Hodge structure of weight $k$ determines a representation of
$\sll 2(\bC)$ into the complex Lie algebra of infinitesimal
isometries of the polarization via the $\sll 2$-pair
\begin{equation}
      (N,Y_{(F,W)}-k\,\mathbb{1})     \label{eq:sl2-pair-1}
\end{equation}
Likewise, the $\sll 2$-pair $(N,Y_{(\hat F,W)}-k\,\mathbb{1})$ defines a
representation of $\sll 2(\bR)$ into the Lie algebra of real,
infinitesimal isometries of the polarization.

\subsection{Two theorems of P. Deligne}\label{subsect:two-theorems}

\par Let $W$ be an increasing filtration of a finite dimensional vector space
$V$ over a field of characteristic zero. Let $\endo^W(V)$ denote the subspace
of $\endo(V)$ consisting of elements which preserve $W$.  

\par Let $\gr^W = \oplus_k\, \gr^W_k$ and $\mathbf Y$ be the grading of $\gr^W$
which acts on $\gr^W_k$ as multiplication by $k$.  For clarity,
given an element $A\in\endo^W(V) $ we let $\gr^W(A)$ denote the induced
action of $A$ on $\gr^W$.  Then, an element $\alpha\in\endo(\gr^W)$ commutes
with $\mathbb Y$ if and only if there exists an element $A\in\endo^W(V) $
such that $\gr^W(A) = \alpha$.

\par More precisely, given a grading $Y'$ of $W$ and element $A\in\endo^W(V) $
we have a decomposition
\begin{equation}
       A = \sum_{k\geq 0}\, A_{-k},\qquad [Y',A_{-k}] = -kA_{-k}
       \label{eq:eigen-decomp-2}
\end{equation}
of $A$ into eigencomponents with respect to $\ad{Y'}$.   Moreover,
$\gr^W(A_0)=\gr^W(A)$.  Therefore, given $\alpha\in\endo(\gr^W)$ which commutes
with $\mathbf Y$ and a grading $Y'$ of $W$ there exists a
unique element $\alpha_0\in\endo^W(V) $ which commutes with $Y'$ such
that $\gr^W(\alpha_0)=\alpha$.  We call $\alpha_0$ the lift of $\alpha$
with respect to $Y'$.

\par Suppose now that $(e^{zN}.F,W)$ is an admissible nilpotent orbit and let
$M=M(N,W)$.  Then, $Y_M = Y_{(F,M)}$ is a grading of $M$ which preserves $W$
and satisfies $[Y_M,N] = -2N$.  In~\cite{dtock}, P. Deligne constructs a
grading $Y=Y(N,Y_M)$ of $W$ and an associated $\sll 2$-pair $(N_0,Y_M-Y)$ which
generalizes the construction \eqref{eq:sl2-pair-1} as follows: Let $N$ be a
nilpotent element of $\endo^W(V) $ such that $M=M(N,W)$ exists.  Let
$Y_M$ be a grading of $M$ which preserves $W$ and satisfies $[Y_M,N] = -2N$.
Then, it follows from the definition of the relative weight filtration that
\[
       (\gr^W(N),\gr^W(Y_M)-\mathbf Y)
\]
is an $\sll 2$-pair which commutes with $\mathbf Y$.  Let $\tilde N^+$ be
the third element of the associated $\sll 2$-triple.  By construction,
$\tilde N^+$ commutes with $\mathbb Y$.  Given a choice of grading $Y'$ of
$W$ let $N_0$ and $N_0^+$ the corresponding lifts of $\gr^W(N)$ and $\tilde N^+$.
Note the lift of $\gr^W(Y_M)-\mathbf Y$ is just $H=Y_M-Y'$.

\begin{thm}(P.~Deligne,~\cite{dtock})\label{thm:deligne-1}  Let $Y_M$ be a grading
of $M$ which preserves $W$, i.e. $Y_M(W_k)\subseteq W_k$ for all $k$, such that
$[Y_M,N]=-2N$.  Then, there exists a unique grading $Y = Y(N,Y_M)$ of $W$
such that $[Y,Y_M]=0$ and 
\[
      [N-N_0,N^+_0]=0
\]
\end{thm}

\par Another way of stating this result is that there exists a unique choice
of grading $Y$ of $W$ which commutes with $Y_M$ such that
$(N_0,Y_M-Y)$ is an $\sll 2$-pair with the following property:  If
\[
     N = \sum_{k\geq 0}\, N_{-k},\qquad [Y,N_{-k}] = -k N_{-k}
\]
is the decomposition of $N$ with respect to $\ad Y$ then for each positive
integer $k$, $N_{-k}$ is either zero or a vector of highest weight $k-2$
for the associated adjoint representation of $\sll 2$.  In particular,
$N_{-1}$ is either zero or a vector of highest weight $-1$ with respect
to the $\sll 2$-triple $(N_0,Y_M-Y,N_0^+)$.  Therefore, $N_{-1}=0$.  Likewise,
$N_{-2}$ commutes with $N_0$, $Y_M-Y$  and $N_0^+$ since it is a vector of
highest weight zero for the adjoint representation.

\begin{lemma}\label{lem:deligne} If $(e^{zN}.F,W)$ is an admissible
nilpotent orbit, $M=M(N,W)$ and $Y=Y(F,Y_{(F,M)})$ is the grading of
Theorem~\eqref{thm:deligne-1} then $Y$ preserves $F$.
\end{lemma}    
\begin{proof} See Theorem 4.15,~\cite{dmj}.
\end{proof}

\begin{corr}\label{cor:deligne} Let $(e^{zN}.F,W)$ be an admissible nilpotent
orbit with limit mixed Hodge structure $(F,M)$ split over $\bR$.  Let
$Y=Y(N,Y_{(F,M)})$ and $N=N_0 + N_{-2} + \cdots$ be the corresponding
decomposition of $N$ with respect to $\ad Y$.  Then, $Y=\overline{Y}$ and
$$
    Y_{(e^{zN_0}.F,W)} = Y
$$
for $\Im(z)>0$.
\end{corr}
\begin{proof} By definition $\overline{N} = N$, whereas
$\overline{Y_{(F,M)}} = Y_{(F,M)}$ since $(F,M)$ is split over $\bR$.
Therefore, by virtue of the linear algebraic nature of Deligne's construction,
$Y = \overline{Y}$.  By the previous Lemma, $Y$ preserves $e^{zN_0}.F$ and
hence $Y = Y_{(e^{zN_0}.F,W)}$ since $\overline{Y} = Y$.
\end{proof}

\par One important consequence of W.\ Schmid's $\slgr 2$-orbit theorem~\cite{schmid}
is the construction of another splitting operation
$(F,W)\mapsto (e^{-\xi}.F,W)$,
which we call the \emph{$\sll 2$-splitting}, on the category of mixed Hodge
structures.  If $(F,W)\mapsto (e^{-i\delta}.F,W)$ is Deligne's
$\delta$-splitting then $\xi$ (resp. $\delta$) can be expressed as universal
Lie polynomials in the Hodge components of $\delta$ (resp. $\xi$) relative
to $(F,W)$.

\begin{thm}(P.~Deligne,~\cite{dtock})\label{thm:deligne-2} Let
$(e^{zN}.F,W)$ be an admissible nilpotent orbit with limit mixed Hodge
structure $(F,M)$ split over $\bR$.  Let $Y = Y(N,Y_{(F,M)}$ and
$N=N_0 + N_{-2} + \cdots$ be the decomposition of $N$ into eigencomponents
with respect to $\ad Y$.  Then, $(e^{zN_0}.F,W)$ is the $\sll 2$-splitting of
$(e^{zN}.F,W)$ and $e^{\xi} = e^{zN}e^{-zN_0}$.
\end{thm}
\begin{proof} See~\cite{BP2}.  For the simpler statement that
\begin{equation}
  e^{iyN}e^{-iyN_0} \in \exp(\Lambda^{-1,-1}_{(e^{iyN_0}.F,W)})
  \label{eq:sl2-splitting}
\end{equation}
see the last section of \cite{singsvarmixed}.
\end{proof}

\begin{rmq} For proofs an extensive discussion of these results and their
history, see~\cite{BP2}, \cite{BPR} and references therein.
\end{rmq}

\subsection{Normal Functions and Biextensions}\label{subsect:nf-biext}

Recall (cf.~\cite{sl2anddeg}) that a  variation is type
$(\text{\rm I})$ if there exists an integer $k$ such that its  Hodge numbers
$h^{p,q}$   are zero unless $p+q=k$, $k-1$  (i.e. $\gr^W$ has
exactly two non-zero weight graded-quotients which are adjacent).  We say that
a variation  is type $(\text{\rm II})$ if there is an integer $k$ such that
$h^{p,q}=0$ unless $(p,q)=(k,k)$, $(k-1,k-1)$ or $p+q=2k-1$ and $h^{k,k}$,
$h^{k-1,k-1}$ are non-zero.

\par To continue, given a classifying space $D$ for period maps of type
$(\text{\rm I})$ or $(\text{\rm II})$, with ambient vector space $V$ (contrasting  previous usage of $H$),
 we let
$H$ be the subgroup of $G$ consisting of elements which induce real
automorphisms on $W_k/W_{k-2}$ for each index $k$.  In the case where $D$ is
classifying space of type $(\text{\rm I})$, $H=G_{\bR}$.  When $D$ is of
type $(\text{\rm II})$, $H$ will also contain the complex subgroup
$\exp(W_{-2}(\geg_{\bC}))$.  Moreover, by the form of the Hodge
diamond of a type $(\text{\rm II})$ mixed Hodge structure, it follows that
\begin{equation}
  \Lambda^{-1,-1}_{(F,W)} = W_{-2}(\geg_{\bC})
  \label{eq:lambda-type-2}
\end{equation}
for any element $F\in D$.  For this reason (see Theorem $(2.19)$,
~\cite{sl2anddeg}), it follows that $H$ acts by isometries on $D$.
Set $\mathfrak h = \lie H $.

\begin{thm}(see \cite[Theorem~4.2]{sl2anddeg})\label{thm:sl2-orbit}
Let $e^{zN}.F$ be an admissible nilpotent orbit of type $(\text{\rm I})$ or
$(\text{\rm II})$, with relative weight filtration $M= M(N,W)$ and
$\delta$-splitting
$
            (F,M) = (e^{i\delta}.\hat F,M)               
$.
Let $(N_0,H,N_0^+)$ denotes the $\sll 2$ triple attached to the
nilpotent orbit $e^{zN}.\hat F$ by Theorem \eqref{thm:deligne-2}, and
$N = N_0 + N_{-2}$ denote the corresponding decomposition of $N$ with
respect to $\text{\rm ad}\, Y$ where $H=Y_{(\hat F_{\infty},M)}-Y$.
\footnote{Of course $N_{-2}=0$ for variations of type $(\text{\rm I})$.}
Then, there exists an element 
\[
    \zeta\in\mathfrak h\cap\ker(N)\cap\Lambda^{-1,-1}_{(\hat F,M)}
\]
and distinguished real analytic function $g:(a,\infty)\to H$ such that 
\begin{itemize}
\item[(a)] $e^{iyN}.F = g(y)e^{iyN}.\hat F$;
\item[(b)] $g(y)$ and $g^{-1}(y)$ have convergent series expansions about
$\infty$ of the form
\[
\aligned
       g(y) &= e^{\zeta}(1 + g_1 y^{-1} + g_2 y^{-2} + \cdots)      \\
  g^{-1}(y) &= (1 + f_1 y^{-1} + f_2 y^{-2} + \cdots)e^{-\zeta}         
  \endaligned
  \]
\end{itemize}
with $g_k$, $f_k\in \ker((\text{\rm ad}\, N_0)^{k+1})
\cap\ker(\text{\rm ad}\, N_{-2})$.
\end{thm}

\begin{corr}(see, Corollary 4.3,~\cite{sl2anddeg})\label{corr:sl2-corr}
Let $\cH\to\Delta^*$ be an admissible variation of  type
$(\text{\rm I})$ or $(\text{\rm II})$, with period map $F(z):U\to D$ and
nilpotent orbit $e^{zN}.F$.  Then, adopting the notation of Theorem
\eqref{thm:sl2-orbit}, there exists a distinguished, real--analytic function
$\gamma(z)$ with values in $\mathfrak h$ such that, for $\Im(z)$ sufficiently
large,
\begin{itemize}
\item[(i)]  $F(z) = e^{xN}g(y)e^{iyN_{-2}}y^{-H/2}e^{\gamma(z)}.F_o$;
\item[(ii)] $|\gamma(z)| = O(\Im(z)^{\beta}e^{-2\pi\Im(z)})$ as $y\to\infty$ and
$x$ restricted to a finite subinterval of $\bR$, for some constant
$\beta\in\bR$.
\end{itemize}
where $F_o = e^{iN_0}.\hat F$.
\end{corr}

\begin{lemma} If $\cH$ is a variation of type $(\text{\rm II})$ then
$\alpha\in\geg_{\bC}\cap\ker(\text{\rm ad}\,N)$ if and only if
$\alpha\in\geg_{\bC}\cap\ker(\text{\rm ad}\,N_0)
  \cap\ker(\text{\rm ad}\,N_{-2})$.
\end{lemma}
\begin{proof} Since $N=N_0+N_{-2}$, clearly
$
  \ker(\text{\rm ad}\,N_0)\cap\ker(\text{\rm ad}\,N_{-2})
  \subseteq\ker(\text{\rm ad}\,N)
$.
Conversely, suppose $\alpha\in\geg_{\bC}\cap\ker(\text{\rm ad}\,N)$.
The non-zero weight graded-quotients of $gl(V)^W$ are
\[
     \gr^W_{\ell}(V\otimes V^*)
     \cong \bigoplus_{j+k=\ell}\, \gr^W_j(V)\otimes \gr^W_k(V^*),\qquad
      \ell\leq 0
\]
from which it follows that the only non-zero weight graded quotients of
$gl(V)^W$ occur in weights $0$, $-1$ and $-2$.  Using $\text{\rm ad}(Y)$
we can write $\alpha = \alpha_0 + \alpha_{-1}+\alpha_{-2}$.  Then,
\[
    0 = [N,\alpha]
      = [N_0 + N_{-2},\alpha_0 + \alpha_{-1} + \alpha_{-2}]
\]
and hence $[N_0,\alpha_0] = 0$, $[N_0,\alpha_{-1}]=0$, $[N_{-2},\alpha_{-1}]=0$,
$[N_{-2},\alpha_{-2}]=0$ and
\[
     [N_0,\alpha_{-2}] + [N_{-2},\alpha_0] = 0
\]
By the Monodromy theorem discussed at the end of \eqref{subsect:pure}, it
follows that $N$ acts trivially on $\gr^W_0$ and $\gr^W_{-2}$.   Therefore,
$N_0(V)\subseteq W_{-1}$ and hence $\alpha_{-2}(N_0(V))=0$.  Likewise,
$\alpha_{-2}(V)\subseteq W_{-2}$ and $\gr^W_{-2}=W_{-2}/\{0\}$.  As such,
$N_0(\alpha_{-2}(V))=0$.  This shows, $[N_0,\alpha_{-2}]=0$ and hence
$[N_{-2},\alpha_0] = 0$ as well by the previous equation.
\end{proof}     

\begin{corr}[cf. Theorem 4.7,~\cite{sl2anddeg}] Let $\cH\to\Delta^*$ be
an admissible variation of type $(\text{\rm I})$ or $(\text{\rm II})$ with
unipotent monodromy $T=e^N$.  Let $\alpha\in\geg(V)$ be a flat, global
section, which acts by infinitesimal isometries of the graded-polarizations.
Then, $\alpha$ has bounded mixed Hodge norm. \label{cor:IandII}
\end{corr}
\begin{proof} In the notation of Corollary ~\eqref{corr:sl2-corr}, the
statement boils down to computing the asymptotic behavior of
\[
     \|\alpha\|_{F(z)}
     = \|\alpha\|_{e^{xN}g(y)e^{iyN_{-2}}y^{-H/2}e^{\gamma(z)}.F_o}
\]
for $\alpha\in\ker(\text{\rm ad N})$ in some vertical strip of width 1
in the upper half-plane.  By part $(b)$ of Theorem \eqref{thm:sl2-orbit},
it follows that $g(y)$ and $e^{iyN_{-2}}$ commute.  Accordingly, by
\eqref{eq:lambda-type-2} and \eqref{eq:isom-eq} and the fact that
$g(y)$ takes values in $G_{\bR}$, it follows that,\footnote{
Again for emphasis $G_{\bC}$ acts linearly on filtrations and by the
adjoint action on $G_{\bC}$ and $\geg_{\bC}$.}  
\[
\aligned
      \|\alpha\|_{F(z)}
      &= \|e^{-xN}.\alpha\|_{g(y)e^{iyN_{-2}}y^{-H/2}e^{\gamma(z)}.F_o} \\
      &= \|\alpha\|_{g(y)e^{iyN_{-2}}y^{-H/2}e^{\gamma(z)}.F_o} \\
      &= \|e^{-iyN_{-2}}g^{-1}(y).\alpha\|_{y^{-H/2}e^{\gamma(z)}.F_o} \\
      &= \|g^{-1}(y)e^{-iyN_{-2}}.\alpha\|_{y^{-H/2}e^{\gamma(z)}.F_o}
\endaligned
\]
By the previous Lemma,
$\alpha\in\ker(\text{\rm ad}(N_0))\cap\ker(\text{\rm ad}(N_{-2}))$, and
so the preceding equation simplifies to
\[
     \|\alpha\|_{F(z)}
     = \|g^{-1}(y).\alpha\|_{y^{-H/2}e^{\gamma(z)}.F_o}
\]
Returning to part (b) of Theorem~\eqref{thm:sl2-orbit}, it follows that
upon decomposing $f_k$ into isotypical components with respect to
$(N_0,H,N_0^+)$ that $f_k$ occurs in components of highest weight $\leq k$
since $f_k\in\ker((\text{\rm ad}\,N_0)^{k+1})$. Therefore, since
$\zeta\in\ker(\ad N)$ and $f_k$ is the
coefficient of $y^{-k}$ in the expansion of $f(y)=g^{-1}(y)$ it follows that
\begin{equation}
     {\widetilde g}^{-1}(\infty) = 
     \lim_{y\to\infty}\, \text{\rm Ad}(y^{H/2})g^{-1}(y)    \label{eq:limit-f}
\end{equation}
exists as an element of the Lie group $H$\footnote{There is a typo at the
end of the proof of Theorem $4.7$ in\cite{sl2anddeg}, $\Ad(Y^{H/2})f_k y^{-k}$
is a polynomial without constant term in $y^{-1/2}$.}.  Thus,
\[
     \|\alpha\|_{F(z)}
     = \|\tilde g^{-1}(y) y^{H/2}.\alpha\|_{e^{\gamma(z)}.F_o}
\]
where ${\widetilde g}^{-1}(y) = \Ad(y^{H/2})g^{-1}(y)$.  Finally, since
$\alpha\in\ker(\text{\rm ad}\,N_0)$ it follows that $y^{H/2}\alpha$ converges
as $y\to\infty$.  As $\gamma(z)\to 0$ as $y\to\infty$ and $x$ constrained to
a finite interval, the proof is now complete.
\end{proof}     


\subsection{$\text{Ext}^1(\bR(0),\text{weight -2})$}
\label{subsect:higher-nf}

\par Let $\mathcal A$ and $\mathcal B$ be variations of pure Hodge structure
of respective weights $a$ and $b$.  Assume that $a=b+2$.  Then,
\[
      \ext^1_{\text{\rm AVMHS}}(\mathcal A,\mathcal B) \cong
      \ext^1_{\text{\rm AVMHS}}(R,\mathcal A^*\otimes B)
\]      
where $R=\mathbb Z$, $\mathbb Q$ or $\bR$ and AVMHS is the category
of admissible variations of graded-polarizable mixed Hodge structure.
Accordingly, for the remainder of this section, we will consider a
variation of Hodge structure $\mathcal H\to\Delta^*$ of weight $-2$
and an admissible variation
$\cH\in\ext^1_{\text{\rm AVMHS}}(\bR(0),\mathcal H)$, with
unipotent monodromy $T=e^N$.

\begin{thm}\label{thm:weight-minus-2-case} If $v$ is a flat section of
$\cH$ then $\|v\|$ is bounded.
\end{thm}
\begin{proof} Let $(F,M)$ denote the $\delta$-splitting of the limit mixed
Hodge structure of $\cH$.  Let $Y = Y(N,Y_{(F,M)})$ be the grading of
$W$ constructed in Theorem \eqref{thm:deligne-1}.  Then, by virtue of the
short length of the weight filtration $W$ of $\cH$,
\[
       N = N_0 + N_{-2}
\]
with respect to $\text{\rm ad}(Y)$.
\par If $N=N_{-2}$ the variation is unipotent in the sense of R.\ Hain and S.\ Zucker,
and the result follows from section \eqref{subsect:unipot}.  If $N=N_0$ the
result follows from section \eqref{subsect:quasi-pure} below.  It remains
to consider the case where $N=N_0 + N_{-2}$ with both $N_0$ and $N_{-2}$
non-zero.  In this case, we will show that $v$ is a section of
$W_{-2}(\cH)=\mathcal H$, and hence the result follows from W.\ Schmid's
$\slgr 2$-orbit theorem.
\end{proof}

\par To complete the proof, we recall that $[N_0,N_{-2}]=0$, $\bar Y=Y$ and
$Y$ preserves $F$ by Lemma \eqref{lem:deligne}.  For the remainder of this
section we assume that both $N_0$ and $N_{-2}$ are non-zero.  From this,
we will derive a contradiction unless $v\in W_{-2}$.

\par By the monodromy theorem, $N$ acts trivially on $\gr^W_0$ and hence
$N_0$ acts trivially on $E_0(Y)\cong \gr^W_0$.  By Corollary
\eqref{cor:deligne}, $Y= Y_{(e^{iN_0}.F,W)}$, and hence if $e_0$ is a generator
of $I^{0,0}_{(e^{iN_0}.F,W)}$ then $N_0(e_0)=0$ and 
\[
     e_0 = e^{-iN_0}(e_0)\in F^0
\]
Since $[Y,Y_{(\hat F,M)}]=0$ it follows that $(N_0,Y_{(F,M)}-Y)$ restricts
to a trivial $\sll 2$-pair on $E_0(Y)$.  Therefore, $e_0\in M_0$.  As such,
\[
      e_0\in F^0\cap\overline{F^0}\cap M_0 = I^{0,0}_{(F,M)}
\]
Accordingly $N_{-2}(e_0)\in I^{-1,-1}_{(F,M)}$.  Moreover, since $[N_0,N_{-2}]=0$
and $N_0(e_0)=0$ it follows that
\[
    N_0N_{-2}(e_0) = N_0N_{-2}(e_0) - N_{-2}N_0(e_0) = [N_0,N_{-2}](e_0)=0
\]
Thus, $N_{-2}(e_0) \in\ker(N_0)\cap I^{-1,-1}_{(F,M)}\cap W_{-2}$.   Moreover, if
$N_{-2}(e_0)=0$ then $N=N_0$ due to the short length.  By assumption,
$N_{-2}\neq 0$, and hence $N_{-2}(e_0)\neq 0$. 

\par Suppose now that $v\in\ker(N)$ and $v=v_0 + v_{-2}$ with
$v_j\in E_j(Y)$.  If $v_0=0$ we are done.  Otherwise, after rescaling, we can
assume that $v_0 = e_0$.  To continue, observe that $N_{-2}(v_{-2})=0$ by the
short length of $W$.  Therefore, since $N_0(e_0)=0$,
\[
    N(v) = N_{-2}(e_0) + N_0(v_{-2}) = 0
\]
and hence
\begin{equation}
  N_{-2}(e_0) \in\ker(N_0)\cap\text{\rm Im}(N_0)\cap I^{-1,-1}_{(F,M)}
              \cap W_{-2}     \label{eq:obstruction} 
\end{equation}
As we must also have $N_{-2}(e_0)\neq 0$, the following Lemma completes the
proof:

\begin{lemma} For $(F,M)$ as above,
$\ker(N_0)\cap\text{\rm Im}(N_0)\cap I^{-1,-1}_{(F,M)}\cap W_{-2}=0$. 
\end{lemma}
\begin{proof} This is a statement about the $\slgr 2$-orbits of pure Hodge
structure induced by $(e^{zN}.F,W)$ on $W_{-2}$.  By~\cite{schmid}, these
are classified as follows:  Let,
\begin{itemize}
\item[(a)] $\bC^2 = \text{\rm span}(e,f)$ with $e=\bar e$ type $(1,1)$ and
$f=\bar f$ type $(0,0)$ with respect to the limit mixed Hodge structure, and
\[
    N = \begin{pmatrix} 0 & 0 \\ 1 & 0 \end{pmatrix}
\] 
with respect to the basis $\{e,f\}$.  The resulting nilpotent orbit is
pure of weight $1$.    
\item[(b)] $E(p,q)=\text{\rm span}(e,f)$ with $p>q$, $N$ acting trivially,
$e=\bar e$, $f=\bar f$ and $e+if$ of type $(p,q)$ with respect to the
limit mixed Hodge structure;  
\item[(c)] $\bR(p)$ is rank 1 of pure of type $(-p,-p)$ and $N$ acting
trivially.
\end{itemize}
Then, every $\slgr 2$-orbit of pure Hodge structures is a direct sum of factors
which are tensor products of the form
$\text{\rm Sym}^m(\bC^2)\otimes\bR(p)$ and
$\text{\rm Sym}^n(\bC^2)\otimes E(p,q)$ where $m$ and $n\geq 0$ and
$\text{\rm Sym}^0(\bC^2)=\bR(0)$.

\par To continue, we observe that in the language of the orbit types
$(a)$--$(c)$ the Lemma asserts that
\begin{equation}
      \ker(N)\cap\text{\rm Im}(N)\cap I^{-1,-1} = 0 \label{eq:obstruction-2}
\end{equation}
(relative to the limit mixed Hodge structure) as $N_0$ becomes just $N$
for the induced orbit on $W_{-2}$.      

\par Next, we note that the factor
$\text{\rm Sym}^n(\bC^2)\otimes E(p,q)$ never
contributes any Tate classes to the limit mixed Hodge structure, so we
need only consider factors of the form
$\text{\rm Sym}^m(\bC^2)\otimes\bR(p)$.  Moreover, since
$\text{\rm Sym}^m(\bC^2)$ underlies a nilpotent orbit of weight
$m$, we must have $p=m+1$ in order to obtain an nilpotent orbit of pure
Hodge structure of weight $-2$.  

\par To finish the proof of the lemma, observe that on the factor
$\text{\rm Sym}^m(\bC^2)$,
\[
    \ker(N)\cap\text{\rm Im}(N) = \bC f^m
\]
where $m$ must be $>0$ (in order to have a non-trivial $N$ action).
Moreover, $f^m$ belongs to $I^{0,0}$ of the limit
mixed Hodge structure of $\text{\rm Sym}^m(\bC^2)$.  Accordingly,
$\ker(N)\cap\text{\rm Im}(N)$ is contained in $I^{-m-1,-m-1}$ of the
limit mixed Hodge structure of
$\text{\rm Sym}^m(\bC^2)\otimes\bR(m+1)$.  As $m>0$,
equation \eqref{eq:obstruction-2} holds.
\end{proof}  

\par We now consider the variation $\geg(\cH)$ where
$\cH\in\ext^1_{\text{AVMHS}}(\bR(0),\mathcal H)$ with $\mathcal H$
pure of weight $-2$.  Since $\cH$ only has weights $0$ and $-2$
whereas $\cH^*$ has weights $0$ and $2$ it follows that
$\cH\otimes\cH^*$ has weights $-2$, $0$ and $2$.  Therefore,
$\geg(\cH)$ only has weights $0$ and $-2$ since
$\geg(\cH)$ is the subvariation consisting of elements which
preserve the weight filtration and induce infinitesimal isometries of the
graded-polarizations.  Therefore, Theorem \eqref{thm:weight-minus-2-case}
applies to $\geg(\cH)$ upon viewing it as an extension
of $\bR(0)$ by a variation of pure Hodge structure of weight $-2$.

\subsection{Biextensions arising from higher height pairings}
\label{subsect:higher-cycles}  Let $X$ be a smooth, complex projective variety
of dimension $d$.  Following the notation of~\cite{hheights}, let
$Z\in \cZ^p(X,1)_{00}$ and $W\in \cZ^q(X,1)_{00}$ be higher cycles representing
elements of $\mathsf{CH}^p(X,1)$ and $\mathsf{CH}^q(X,1)$ respectively.  Then:

\begin{thm} (\cite[Theorem ~A]{hheights}) Assume that \label{thm:A}
\begin{itemize}
\item[(i)] $p+q=d+2$;
\item[(ii)] $\delta Z = \delta W =0$
\item[(iii)] the intersection of $Z$ and $W$ satisfies some extra
  technical conditions.
\end{itemize}
Then, there is a canonical mixed Hodge structure $B_{Z,W}$ attached to Z and
W from which one can extract a Hodge theoretical height pairing
$\langle Z,W\rangle_{\text{\rm Hodge}}$. Moreover, if Z and W both have real
regulator zero then
\[
     \langle Z,W\rangle_{\text{\rm Hodge}}
     = \langle Z,W\rangle_{\text{\rm Arch}}
\]
where $\langle Z,W\rangle_{\text{\rm Arch}}$ is the Archimedean part of an
intersection pairing on arithmetic Chow groups.
\end{thm}

The mixed Hodge structure $B_{Z,W}$ has weight graded-quotients
$\gr^W_0\cong\mathbb Z(0)$, $\gr^W_{-2}$ and $\gr^W_{-4}\cong\mathbb Z(2)$.  Let
$X\to S$ be a family of smooth complex projective varieties and
$Z$, $W$ be a flat family of higher cycles over $S$ such that
$\langle Z_s,W_s\rangle$ is defined over a Zariski dense open subset of $S$.
In this way, the construction of Theorem~\ref{thm:A}  produces an admissible variation
of mixed Hodge structure $\cH$ over a Zariski dense open set of $S$
with weight graded quotients $\gr^W_0(\cH)\cong\mathbb Z(0)$,
$\gr^W_{-2}(\cH)$ and $\gr^W_{-4}(\cH)\cong\mathbb Z(2)$.

\begin{lemma} Let $(F,W)$ be a mixed Hodge structure with underlying vector
space $V$ and weight graded quotients $\gr^W_0\cong\mathbb Z(0)$,
$\gr^W_{-2}$ and $\gr^W_{-4}\cong\mathbb Z(2)$.  Let $\geg_{\bC}(U)$
denote the Lie algebra of elements of $gl(U)$ which preserve $W(U)$
and induce infinitesimal isometries of $\gr^{W(U)}$ where
$U=W_{-2}(V)$, $V$ or $V/W_{-4}$.  Then, since elements of
$\geg_{\bC}(V)$ preserve $W$, we have an induced map
\[
     q:\geg_{\bC}(V) \to \geg_{\bC}(V/W_{-4})
\]
and a restriction map
\[
     r:\geg_{\bC}(V) \to \geg_{\bC}(W_{-2})
\]
By abuse of notation, let $q(F)$ and $r(F)$ denote the mixed Hodge structure
induced by $(F,W)$ on $\geg_{\bC}(V/W_{-4})$ and
$\geg_{\bC}(W_{-2})$.  Let $\beta\in\geg_{\bC}$ be
horizontal with respect to $F$.  Then,
\begin{equation}
    \|\beta \|_F \leq \|q(\beta) \|_{q(F)}
        + \|r(\beta) \|_{r(F)}
     \label{eq:quot-restrict}
\end{equation}
\end{lemma}
\begin{proof} The key point is that $W_{-4}\geg(V)$ is pure
of type $(-2,-2)$ and $W_{-3}\geg(V) = W_{-4}\geg(V)$.
Therefore, if $\beta = \sum_{p,q}\,\beta^{p,q}$ denote the decomposition
of $\beta$ into Hodge components with respect to $(F,W)$ then
$\beta^{p,q}=0$ unless $p\geq -1$.  As such $\beta^{p,q}=0$ unless
$p+q=0$ or $p+q=-2$.  Thus, \eqref{eq:quot-restrict} captures the
mixed Hodge norm of $\sum_{p+q=-2}\,\beta^{p,q}$ accurately and
double counts the mixed Hodge norm of $\sum_{p+q=0}\,\beta^{p,q}$.
\end{proof}

\par Suppose now that $\alpha$ is a horizontal section of
$\geg(\cH)$ then pointwise application of the previous
Lemma shows that 
\begin{equation}
    \|\alpha\|_{\cH}
    \leq \|q(\alpha) \|_{\geg(\cH/W_{-4}\cH)}
       +\|r(\alpha)\|_{\geg(W_{-2}\cH)}
    \label{eq:quot-restrict-2}
\end{equation}

\begin{corr}\label{corr:height-pairing}  Let $\cH\to\Delta^*$ be
an admissible variation of graded-polarized mixed Hodge structure
over the punctured disk with unipotent monodromy.  Assume that
$\cH$ has weight graded quotients $\gr^W_0\cong\mathbb Z(0)$,
$\gr^W_{-2}$ and $\gr^W_{-4}\cong\mathbb Z(0)$.  Let $\alpha$ be a
flat, horizontal section of $\geg(\cH)$.  Then,
$\alpha$ has bounded mixed Hodge norm.
\end{corr}
\begin{proof} By \eqref{eq:quot-restrict-2}, $\|\alpha\|_{\cH}$
is bounded by $\|q(\alpha)\|$ and $\|r(\alpha)\|$.  Moreover,
$q(\alpha)$ and $r(\alpha)$ are flat since $\alpha$ is flat and $W$ is flat.
Therefore, the result follows from Theorem
\eqref{thm:weight-minus-2-case} and the last paragraph of
section \eqref{subsect:higher-nf}.
\end{proof}
        
\subsection{A case where  norm estimates fail}
\label{subsect:higher-normal-functions}

\par In this section, we show via admissible nilpotent orbits that in
the case of a higher normal function with weight graded quotients
$\gr^W_0=\mathbb Z$ and $\gr^W_{-k}$ for $k>2$, the norm estimates required to
obtain rigidity need not hold.

\begin{lemma} \label{lem:UnBounded} Let $(e^{zN}.F,W)$ be an admissible nilpotent orbit with limit
mixed Hodge structure $(F,M)$ split over $\bR$.  Let $Y=Y(N,Y_{(F,M)})$ and
$N=N_0 + \cdots + N_{-k}$ relative to $\ad Y$ with $N_{-k}\neq 0$.  Then,
$
       \|N\|_{(e^{zN}.F,W)} = \|N\|_{(e^{iyN}.F,W)}
$
and there a non-zero constant $K$ such that 
\[
        \lim_{y\to\infty}\, y^{(2-k)/2}\|N\|_{(e^{iyN}.F,W)} =K
\]      
In particular, $\|N\|_{(e^{zN}.F,W)}$ is bounded for $k=2$
and unbounded for $k>2$.     
\end{lemma}

\begin{proof} Note that $N$ and $H = Y_{(F,M)}-Y$ are elements of
$\geg_{\bR}$.  Since $\overline{Y}=Y$ it follows that
$\overline{N_0}=N_0$.  As $\gr^W(N_0)=\gr^W(N)$ it follows that
$N_0\in\geg_{\bR}$.  Thus, omitting $W$ from the mixed
Hodge norm as in \eqref{eq:unipotent-9}, we have
\[
\begin{aligned}
      \|N\|_{e^{zN}.F}
       &= \|N \|_{e^{xN}e^{iyN}.F}  
        = \|e^{-xN}.N\|_{e^{iyN}.F} \\
       &= \|N\|_{e^{iyN}.F}
        = \|N\|_{e^{iyN}e^{-iyN_0}e^{iyN_0}.F} \\
       &= \|e^{iyN_0}e^{-iyN}.N\|_{e^{iyN_0}.F}
\end{aligned}            
\]
where the last step is justified by equation \eqref{eq:sl2-splitting}.
To continue, we note that since $[H,N_0] = -2N_0$
we have
\begin{equation}
    e^{iyN_0} = y^{-H/2}e^{iN_0}y^{H/2} = y^{-H/2}.e^{iN_0}    \label{eq:ds-1}
\end{equation}
wherefrom    
$$
       e^{iyN_0}.F = y^{-H/2}e^{iN_0}y^{H/2}.F = y^{-H/2}e^{iN_0}.F
$$
since $H$ preserve $F$.  Moreover, as a consequence of the $\slgr 2$-orbit
theorem in the pure case, $F_o = e^{iN_0}.F\in D$.  Therefore,
$$
\begin{aligned}
       \|N\|_{e^{zN}.F}
       &= \|e^{iyN_0}e^{-iyN}.N\|_{e^{iyN_0}.F} \\
       &= \|e^{iyN_0}.N\|_{e^{iyN_0}.F}  \\
       &= \|e^{iyN_0}.N\|_{y^{-H/2}.F_o} \\
       &= \|y^{H/2}e^{iyN_0}.N \|_{F_o} \\
       &= \|e^{iN_0}y^{H/2}.(N_0 + \cdots + N_{-k})\|_{F_o}
\end{aligned}
$$
where the last step is justified by \eqref{eq:ds-1}.
Accordingly, as $[H,N_{-j}]=(j-2)N_{-j}$ for $j=0,\dots,k$ it follows
that $\|N\|_{e^{zN}.F}$ is asymptotic to a constant multiple of
$y^{(k-2)/2}$ for large $y$.
\end{proof}

\subsection{The case $N=N_0$}\label{subsect:quasi-pure}

\par Let $\cH\to\Delta^*$ be an admissible nilpotent orbit with
unipotent monodromy $T=e^N$.  Let $(F_{\infty},M)$ be the limit mixed
Hodge structure of $\cH$ with $\delta$-splitting
\begin{equation}
      (\hat F_{\infty},M) = (e^{-i\delta}.F_{\infty},M)     \label{eq:qp-1}
\end{equation}
Let $Y_M = Y_{(\hat F_{\infty},M)}$ and $Y = Y(N,Y_M)$.  Let
\begin{equation}
    N = N_0 + N_{-2} + \cdots                           \label{eq:qp-2}
\end{equation}
denote the decomposition of $N$ into eigencomponents for $\text{\rm ad}\, Y$.
Let
\begin{equation}
      (N_0,H,N_0^+),\qquad H= Y_M-Y                      \label{eq:qp-3}
\end{equation}
be the associated representation of $\sll 2(\bR)$ of Theorem
\eqref{thm:deligne-1}.
\medskip

\par In this section we prove the following result, by essentially modifying
the unipotent case accordingly:

\begin{thm}\label{thm:n-is-n0} If $N=N_0$ and $v$ is a flat, global
section of $\cH$ then $v$ has bounded mixed Hodge norm.
\end{thm}

\par As the first step towards the proof of Theorem~\eqref{thm:n-is-n0},
we note that since $N=N_0$, $(N_0,H)$ is an $\sll 2$-pair and
$[N,\delta]=0$, it follows that $\delta$ is a sum of lowest weight vectors
for $(N_0,H,N_0^+)$.  Therefore, 
\begin{equation}
       \delta = \delta_0 + \delta_{-1} + \cdots,\qquad
       [H,\delta_{-j}] = -j\delta_{-j}                    \label{eq:qp-4}
\end{equation}
relative to the eigenvalues of $\text{\rm ad}\, H$.   Let
\begin{equation}
     \delta(y) = \text{\rm Ad}(y^{H/2})\delta = y^{H/2}.\delta
               = \sum_{k\geq 0}\, \delta_{-k}y^{-k/2}  
     \label{eq:qp-5}
\end{equation}

\begin{lemma}\label{lem:qp-delta} In the notation of
\eqref{eq:qp-1}--\eqref{eq:qp-5}, if $N=N_0$ then
\begin{equation}
      \delta(\infty) := \lim_{y\to\infty}\, \delta(y) = \delta_0
                        \label{eq:qp-6}
\end{equation}
and $e^{i\delta(\infty)}e^{iN}.\hat F_{\infty}\in D$.
\end{lemma}
\begin{proof} Equation \eqref{eq:qp-6} follows directly from equation
\eqref{eq:qp-5}.  To prove that the point  $e^{i\delta(\infty)}e^{iN}.\hat F_{\infty} $ belongs to $D$,
observe that it is sufficient to consider only the pure case, since the
property of being a MHS is only about the induced filtrations on $\gr^W$.
Accordingly, for the remainder of this proof only, we assume
$e^{zN}.F_{\infty}$ is a nilpotent orbit of pure Hodge structure. By
W.~Schmid's $\slgr 2$-orbit theorem, we have   
$$
       y^{H/2}e^{i\delta}e^{iyN}.\hat F_{\infty}
       = y^{H/2}g(y)e^{iyN}.\hat F_{\infty}
$$
which we can rewrite as
\begin{equation}
  e^{i\delta(y)}e^{iN}.\hat F_{\infty} = y^{H/2}g(y)y^{-H/2}e^{iN}.\hat F_{\infty}
  \label{eq:qp-7}
\end{equation}
using $e^{iyN}.\hat F_{\infty} = y^{-H/2}e^{iN}.\hat F_{\infty}$.   Mutatis
mutandis, the argument of equation \eqref{eq:limit-f} shows that 
\[
        {\widetilde g}(\infty) = \lim_{y\to\infty}\, y^{H/2}g(y)y^{-H/2}
\] 
exists, and is an element of $G_{\bR}$ (since we are in the pure case).
By W.~Schmid's $\slgr 2$-orbit theorem, $e^{iN}.\hat F_{\infty}\in D$.     Taking the
limit of \eqref{eq:qp-7} as $y\to\infty$, it follows that
\[
        e^{i\delta(\infty)}e^{iN}.\hat F_{\infty}
        = {\widetilde g}(\infty)e^{iN}.\hat F_{\infty} \in D
\]
as required.
\end{proof}

\par To continue, let $F(z) = e^{zN}e^{\Gamma(s)}.F_{\infty}$ be the local normal
form of the period map of $\cH$.  Then, in analogy with equation
\eqref{eq:unipotent-4a} we have 
\begin{equation}
\aligned
    F(z) &= e^{xN}e^{iyN}e^{\Gamma(s)}.F_{\infty} \\
         &= e^{xN}e^{iyN}e^{\Gamma(s)}e^{i\delta}.\hat F_{\infty} \\
         &= e^{xN}y^{-H/2}e^{iN}y^{H/2}e^{\Gamma(s)}e^{i\delta}.\hat F_{\infty}
\endaligned\label{eq:qp-8}
\end{equation}
In analogy with the derivation of \eqref{eq:unipotent-4b}, since
$[H,N]=-2N$ and $H$ preserves $\hat F_{\infty}$, we have
\begin{equation}
\aligned  
     e^{i\delta}.\hat F_{\infty}
     &= y^{-H/2}y^{H/2}e^{-iyN}e^{iyN}e^{i\delta}.\hat F_{\infty} \\
     &= y^{-H/2}e^{-iN}y^{H/2}e^{iyN}e^{i\delta}.\hat F_{\infty}  \\
     &= y^{-H/2}e^{-iN}y^{H/2}e^{i\delta}e^{iyN}.\hat F_{\infty}  \\
     &= y^{-H/2}e^{-iN}e^{i\delta(y)}y^{H/2}e^{iyN}.\hat F_{\infty} \\
     &= y^{-H/2}e^{-iN}e^{i\delta(y)}e^{iN}y^{H/2}.\hat F_{\infty} \\
     &= y^{-H/2}e^{-iN}e^{i\delta(y)}e^{iN}.\hat F_{\infty}
\endaligned\label{eq:qp-9}
\end{equation}
Inserting \eqref{eq:qp-9} into \eqref{eq:qp-8} yields
\begin{equation}
\aligned
     F(z) &= e^{xN}y^{-H/2}e^{iN}y^{H/2}e^{\Gamma(s)}
             y^{-H/2}e^{-iN}e^{i\delta(y)}e^{iN}.\hat F_{\infty} \\
          &=e^{xN}y^{-H/2}e^{\Gamma(s,y)}e^{i\delta(y)}e^{iN}.\hat F_{\infty}
\endaligned
\label{eq:qp-10}
\end{equation}
where
\begin{equation}
  e^{\Gamma(s,y)} = e^{iN}y^{H/2}e^{\Gamma(s)}y^{-H/2}e^{-iN}
                =\exp(e^{iN}y^{H/2}.\Gamma(s))
      \label{eq:qp-11}
\end{equation}      
In particular, since $\Gamma(0)=0$ and $|s|=e^{-2\pi y}$, there exist positive
constants $C$, $k$ and $a$ such that
\begin{equation}
      |s|<a \implies |\Gamma(s,y)|<C|s|(-\log|s|)^k    \label{eq:qp-12}
\end{equation}
with respect to a choice of fixed norm $|\ast|$ on $\geg_{\bC}$.
      
\begin{proof}[Proof of Theorem \eqref{thm:n-is-n0}] Since $v\in\ker(N)$ it follows
from \eqref{eq:qp-10} that
\begin{equation}
\aligned
     \|v \|_{F(z)}
     &= \|v \|_{e^{xN}y^{-H/2}e^{\Gamma(s,y)}e^{i\delta(y)}e^{iN}.\hat F_{\infty}} \\
     &= \|e^{-xN}.v \|_{y^{-H/2}e^{\Gamma(s,y)}e^{i\delta(y)}e^{iN}.\hat F_{\infty}} \\
     &=  \|v \|_{y^{-H/2}e^{\Gamma(s,y)}e^{i\delta(y)}e^{iN}.\hat F_{\infty}}
\endaligned
\label{eq:qp-13}     
\end{equation}
since $N\in\geg_{\bR}$.   To continue, observe that
$v\in\ker(N)$ implies that
$$
     v = \sum_{k\geq 0}\, v_{-k},\qquad H.v_{-k} = -kv_{-k}
$$
where the number of non-zero terms is finite since $\cH$ has finite
rank.  Therefore,
\begin{equation}
     v(y) := y^{H/2}.v = \sum_{k\geq 0}\, v_{-k}y^{-k/2}
\end{equation}
is a vector-valued polynomial in $y^{-1/2}$, and hence
\begin{equation}
\aligned
    \|v \|_{F(z)}
     &=  \|y^{H/2}.v \|_{e^{\Gamma(s,y)}e^{i\delta(y)}e^{iN}.\hat F_{\infty}}
     &=  \|v(y) \|_{e^{\Gamma(s,y)}e^{i\delta(y)}e^{iN}.\hat F_{\infty}}
\endaligned
\end{equation}
To finish the proof, recall that $D$ is an open subset of $\check D$
in the complex analytic topology, and $G_{\bC}$ acts transitively
on $\check D$ by biholomorphisms.  In particular, since
$e^{i\delta(\infty)}e^{iN_0}.\hat F\in D$ by Lemma \eqref{lem:qp-delta}, it follows
from equation \eqref{eq:qp-5} and the estimate \eqref{eq:qp-12} that
there exists a constant $b>0$ such that 
$$
    K  = \{\, e^{\Gamma(s,y)}e^{i\delta(y)}e^{iN}.\hat F \mid
              y\geq b,\hphantom{a} |s| \leq e^{-2\pi y}\,\}
$$
is a compact subset of $D$.  Therefore, since $\lim_{y\to\infty}\, v(y) = v(0)$
it follows from equation \eqref{eq:qp-13} that $\|v\|_{F(z)}$ is bounded
as $y\to\infty$ and $x$ is constrained to a finite interval.
\end{proof}
     
\begin{rmq} (1) If $N=N_0$ for $\cH$ then $N=N_0$ for
$\cH\otimes\cH^*$.
\\
(2)  The results in this section cover the case of variations of pure
Hodge structure and variations of type $(\text{\rm I})$.
\end{rmq}

  \section{Deformations of mixed period maps}
  \label{sec:DefoMPS}

 \subsection{General set-up}
 \label{ssec:SetUp} 
 The set-up is  similar to the one  in the pure case. More precisely, we only consider deformations   of  a period map $F:S \to \Gamma\backslash D$ such that

 \begin{itemize}
\item $S,D$ and $\Gamma$ remain fixed.
\item the deformation remains locally liftable and horizontal.
\end{itemize}

 However, there is an additional requirement "at infinity": we want the variation to be admissible\footnote{The concept of admissibility is recalled in Appendix~\ref{sec:admis}.  Note that our convention of admissibility
  includes   as a requirement  that the
  monodromy operators around the  boundary are quasi-unipotent. This  is automatically the case   
   if the variation  has  an underlying $\bZ$-structure such as the ones coming from geometry. }, a requirement that is automatic for pure variations 
 and which holds for mixed variations of geometric origin  (cf.\ \cite[Def. 14.49]{mht}).  
   \par
Mixed period maps of admissible variations will be called \textbf{\emph{admissible period maps}}. 
Admissibility is preserved under small deformations:
\begin{lemma} If $F$ is an admissible period map,   sufficiently small deformations of $F$  that
stay horizontal, also  stay admissible.
\end{lemma}
\begin{proof}
 We can test admissibility on curves and so we may replace
$S$ with a curve.   We employ the test given in \cite{higgs}. 

For a neighborhood of $p\in \bar S\setminus \del S$ we take a small disc $\Delta$ centered at  $p$ with coordinate $s$
and monodromy $T$ around the origin. We may assume that
$T$ is unipotent. Set $N= -\log T$.
If $s\in \Delta\setminus \set{0}$, we may put  $s=\e^{2\pi\ii z}$. Then the untwisted period map
$\e^{-zN}. F(z)$ extends over the origin  as a holomorphic map $\Delta \to \check D$ 
where its value at $s=0$ is traditionally denoted $F(\infty)\in \check D$ (since it corresponds to a limit for $s\to \infty$). The canonical extension to $\Delta$ of the 
local system (with weight filtration and rational structure) over $\Delta^*$ puts a weight filtration and
rational structure on the ``central''  fiber $H $ over $0$.
Admissibility implies that
there is a relative weight filtration $M$ on the central fiber $H $ and $(H ,M,  F(\infty))$ is a mixed Hodge structure,
the ``limit mixed Hodge structure''.
Hence  we have a Deligne decomposition and we can speak of horizontal endomorphisms with respect to 
the limit mixed Hodge structure. We shall call these ``limit-horizontal'' and denote these as $\gq_{F(\infty)}^\hor$.
In this case the local normal form   \eqref{eq:admissibility-7} reads
\[
F(s) = \exp\left(\frac{\log (s)}{2\pi\ii}   N \right) \exp (\Gamma(s) ). F(\infty), \,\Gamma(0)=0, 
\]
and where
\[
 \Gamma(s)   = 1+  \Gamma_{-1} (s)+ \Gamma_{-2}(s)  +\dots  ,\quad \Gamma_{-k}(s)\in    U^{-k}_ {F(\infty)} 
\]
is uniquely determinable from $\Gamma_{-1}\in  \gq_{F(\infty)}^\hor$. 
Let 
\[
F(s,t) = \exp\left(\frac{\log (s)}{2\pi\ii}   N \right) \exp (\Gamma(s,t) ). F(\infty)  ,\quad  \Gamma(s, t)\in \gq_{F(\infty)}  .
\]
 be a deformation of $F(s) $ as   a period map.  This is nothing but  a $2$-parameter period map
 $\Delta^*\times \Delta \to \Gamma\backslash D$  with trivial monodromy in  the second factor. If now  
 \[
 \exp (\Gamma(s,t)) = 1+ \tilde \Gamma_{-1}(s,t) + \tilde \Gamma_{-2}(s,t)+\dots ,   \tilde \Gamma_{-k}(s,t) \in U^{-k}_{F(\infty)},
 \]
  the initial value constraint reads $\tilde \Gamma _{-1}(0,0)=\Gamma_1(0)=0$
 and the ``Higgs bundle constraint'' holds since $F(s,t)$ is assumed to be horizontal. Indeed, the Higgs bundle constraint is equivalent to
 the image at any point of the tangent space under the period map being an abelian subspace of $ \geg_\bC$ which is the case, cf.\ Lemma~\ref{lem:AbSA}.
  But then, by loc. cit., 
   $F(s,t)$ is an admissible  nilpotent orbit   with the same relative weight filtration $M$  and  
      limit mixed Hodge structure  $F(\infty)$  as before. 
      \end{proof}

In view of the above, we call deformations of admissible period maps
 that stay  locally liftable and horizontal  (and hence admissible) simply
 \textbf{\emph{admissible deformations}}.

\begin{rmk}
\warn{Recall the commutative diagram~\eqref{eqn:endoHodgeFlag} which provides a surjection
\[ 
 \cF^{-1}  \qendo {\cH}    \mapright{\,\pi^\hor}      F^* T^\hor (\Gamma \backslash D) .
\]
Choosing a lift for  this map at some point $s\in S$ determines a  unique global  lift.  This is
 a consequence of the
  rigidity theorem for  variations of admissible mixed  Hodge structures (cf.\ \cite[Theorem 4.20]{SZ} for $S$ a curve and the remarks in \cite[\S 9]{bryzuck} for the general case). But at a given point $s$, there is a natural identification of $ T^\hor _{F(s)} D$  with  the subspace $ U^{-1} \geg _{F(s)}    $ 
  of $\geg _{F(s)}$ and so
  we have a unique global lift. This lift can  be used to    identify infinitesimal deformations of an admissible variation with a subspace
of    the space of  sections of  $\cU ^{-1}  \qendo {\cH}  \subset \cF^{-1}  \qendo {\cH} $.}
\end{rmk}

\subsection{Main results} \label{ssec:Main}

\begin{thm}[Main Theorem I]   Let $S$ be quasi-projective and $F:S\to \Gamma\backslash D$ a horizontal holomorphic map to a mixed domain $D$ parametrizing
mixed Hodge structures on    $(H,W,Q)_\bR$ and assume that the     variation of mixed Hodge structure $\cH$  corresponding to $F$ is admissible.
\begin{enumerate}

\item Let $\eta $ be a global holomorphic section of  $\geg(\cH)$
corresponding to an admissible  infinitesimal deformation of $F$ with bounded Hodge norm. 
If the section  $\eta $ is plurisubharmonic  along $S$, then  $\eta$ is a   flat section of    $\geg(\cH)$ which is moreover horizontal,  i.e.,  a section of  
$\cU^{-1}\geg(\cH)= \bigoplus_{k\le 1}  \geg ^{-1,k}  (\cH) $.
 
 Equivalently, at any point $s\in S$,     $ \eta(s)  $ is a  horizontal endomorphism of  $\geg(\cH_s) $ which commutes with the action of the fundamental group $ \pi_1(S,s)$.  
 
  \item Conversely,  let   $\eta(s)$ be a  flat  horizontal section  of  $\cU^{-1}\geg(\cH)$  such that $  \eta(s)$
commutes with every element in  $F_* T_{S,s}\subset U^{-1} \geg_{F(s)} $.
 If  the (constant) Hodge norm $\| \eta(s) \|$   is small enough, then $\eta$ defines a  deformation of $F$ keeping  source and target fixed and which remains a period map.
\end{enumerate}
\label{thm:Main}
 \end{thm}
 \begin{proof} (1) This is a direct application of Proposition~\ref{prop:OnPSH}.  The condition $[u,v]=0$ follows as in the pure  case, since we are considering deformations which stay horizontal (see e.g. \cite[Prop. 5.5.1]{3authorsBIS}). \\
 (2)
 We use  an argument due to  Faltings (for weight $1$) \cite{arafal}.  
 Let $\eta $ a parallel horizontal section of
$\qendo {\cH}$ and define the filtration $F_\eta(s)$ by setting
\[
F_\eta(s) =\e^{  \eta(s)} F(s), \quad s\in S.
\]
On the weight graded parts $Q(f(s), f(s))=0$, $f(s)\in F_\eta(s)$.  The map $s\mapsto F_\eta(s)$ is holomorphic but might land
in the compact dual $\check D$ (cf. formula~\eqref{eqn:CompactDual}).

We claim 
\begin{itemize}
\item the second Riemann condition holds 
  if  $\|\eta(s)\|$  is small enough so that  this filtration gives a point inside the period domain $D$. 
 \item The   commuting property guarantees horizontality. 
\end{itemize}     

To prove these claims, first note that  since $\eta$ is parallel, its Hodge norm is constant and
hence also the auxiliary operators 
\begin{align*}
 w_{k,\ell} &=  (\eta^*)^\ell  \comp \eta^k   + (\eta^*)^k\comp \eta^\ell  , \quad  k\not=\ell \\
 w_{k,k} &=   (\eta^*)^k\comp \eta^k  ,
\end{align*} 
have constant Hodge norm. These operators, being self-adjoint, have real eigenvalues
(which might  be negative). Let  the smallest of  these be  $m_{k,\ell}$. Suppose that the nilpotent operator $\eta$ has index of nilpotency $M$ and
set  
\[
\mu= \sum_{1\le k\le \ell\le M } \frac{ m_{k,\ell}}{{k!} \ell!}.
\]
 Then for all $f(s)\in F_s$ we have
\[
\|\e^ { \eta(s)}  f (s) \|^2_{F(s) } =  h_{F(s)} \left([\id +\sum _{1\le k\le \ell\le M } \frac  {1}  {k!\ell!}  w_{k,\ell}]  f(s), f(s)\right)
\ge (1- |\mu|) \|f(s)\|^2
\] 
  and so if $|\mu|<1$ (which is the case if $\eta$ is close to  zero) we have
\[
\|
\e^ { \eta(s)}  f (s) \|^2_{F(s) }\ge (1-|\mu|)  \|f(s)\|^2_{F(s)}>0.
\]
Hence, as we claimed,  $Q$ polarizes the induced Hodge structures on weight graded parts so that the deformed period map $F_\eta=\e^{\eta } F$  gives a holomorphic  
  map $S\to \Gamma\backslash D$.
If, moreover, for  all $s\in S$ and  all tangents $\xi\in T_sS$ one has  $[\eta(s),F_*\xi]=0$, the deformation $F_\eta $ satisfies Griffiths' transversality condition since  this commutativity implies
\begin{align*}
 \nabla_\xi F^p_\eta(s) &= \left( F_*\xi  . \e^{\eta(s)}\right) . F^p(s)  \\
    &=  \left( \e^{\eta(s)} . F_*\xi  \right) .  F^p(s)\\
    &=  \e^{\eta(s)} \nabla_\xi F^p(s)
    \subset \e^{\eta(s)} F(s)^{p-1} = F^{p-1}_\eta (s) .
\end{align*} 
Here we use (cf.  \eqref{eqn:TngIdAdj}) that $\nabla_\xi$ acts as   $\del_\xi +\ad{F_*\xi}$ on $\geg_{F(s)}$, and that $\del_\xi \eta=0$, because $\eta$ is locally constant.
  \end{proof}

\begin{rmk}[Smoothness]  \label{rmk:smoothdefs} 
The second part of the theorem is equivalent to  the relevant deformation space  being smooth  at $F$.  In particular,  $F$ is rigid
if and only if this component  is a non-reduced point.
\end{rmk}

\begin{exmples}[Non-rigid examples]\textbf{\emph{(1) Hodge--Tate variations.}} \label{exmpl:NonRigid}
As we have remarked in Section~\ref{ssec:PSH} (2), one can  easily construct variations that can be  deformed in suitable $(-1,-1)$-directions.
\\
 \textbf{\emph{(2) Nilpotent  orbit associated to K\"ahler classes.}}  We come back to  Example (4)   in Section~\ref{ssec:PSH}. 
The variation we started with is
the $\bR$-split variation defined by the total cohomology of a family of K\"ahler manifolds. The nilpotent orbit  construction gives a deformation of the associated period map as in   the second part of Theorem~\ref{thm:Main}. The role of $\eta(s)$ is played by 
$\sum_{j=1}^k u_j N_j(s)$ where the $N_j$ are coming from independent ample classes which gives a multi-parameter deformation. Suppose that the dimension of the K\"ahler cone  in $H^{1,1}(X)\cap H^2(X,\bR)$ equals $\kappa$. Then 
$k\le \kappa$. If this inequality is strict, the variation is not rigid in at least one $(-1,-1)$-direction.
\\
\textbf{\emph{(3) Biextensions coming from  higher Chow cycles on surfaces.} }
We studied these in   Section~\ref{ssec:PSH} (6). 
Observe that as in the previous example, a flat  infinitesimal  deformation $v$ in a  $(-1,-1)$-direction gives rise to
a nilpotent orbit of deformations and so these deformations are  never rigid in  such directions.
An example of such a  flat $v$ can be constructed as follows. Let $X_s$, $s\in S$  be family of surfaces embedded in  a product $\bP^a\times\bP^b$
of projective spaces,
$A,A'$ hyperplane sections coming from $\bP^a$ giving rise to biextension variation $\cH_s$ over $S$ and $C,D$ 
hyperplane sections  coming from $\bP^b$.
Then $(C,D)$ defines an independent flat  infinitesimal variation $v$  of  biextension type and hence $\exp(tv) \cH_s$
is a deformation of $\cH_s$.

\end{exmples}

%

  In order to formulate the  second main result, we recall that
    Proposition~\ref{prop:OnPSH} states that    for a plurisubharmonic horizontal endomorphism $\eta$  and for all  all tangents $u$ to the period map, one has 
 $\pi_\gq[\pi_+\bar u,v]=0 $,  $v=\eta(s)$. Moreover, this property is equivalent to $v$ being parallel.
 
 In analogy with the pure case (cf. Definition~\ref{dfn:RegTangPure}) we introduce  the following concept:  
 
 \begin{dfn}   Fix a  subspace $\ga \subset  \geg_\bC^\hor$. 
 The period map  $F$ is called  \textbf{\emph{regularly tangent}} at  $s\in S$, respectively \textbf{\emph{regularly tangent  in the $\ga$-directions}},  if   
   the only vector $v\in  \geg^\hor_{F(s)}$, respectively $v\in  \ga$,  with $\pi_\gq[\pi_+\bar u,v]=0$ for  all $u\in F_* T_sS $ is the zero vector.
   \end{dfn}

\begin{rmk} \label{rmk:OnRegTangent}   Because  of type reasons, a period map can only be regularly tangent if there are non-zero $(-1,1)+ (-1,0)$-directions. 
Moreover, if     $F$ is  regularly tangent  in  the  $(-1,1) $-directions  as well as in $(-1,0$-directions, then  $F$ is  regularly tangent  in  
  all directions. 
\end{rmk}

\begin{thm}[Main Theorem II]  \label{thm:MainCrit2}  Fix a  subspace $\ga \subset  \geg_\bC^\hor$.  Suppose 
  we are in  one of the following situations:
 \begin{enumerate}
\item the pure case  with  $\ga =\geg^{-1,1}$;
\item we have two adjacent weights and  $\ga =\geg^{-1,0}$;
\item  in the setting of   unipotent variations , i.e. $u^{-1,1}=0$ provided  either $ \Lambda^{-1,-1}=0$ and $ \ga =\geg^{-1,0}$,  or
$u \in  \Lambda^{-1,-1}$ and $\ga = \oplus_{k\le 0} \geg^{-1, k}$. 

\item $u =u^{-1,1}+u^{-1,-1}$,   and $\ga =\geg^{-1,-1}$.
\item two non-adjacent weights, say $0, k$, $|k|\ge 2$ with $h^{0,0}=1$, $h^{p,-p}=0$ for  $p\not=0$,
and    $ \ga = \geg^{-1, -k+1}$.  
Moreover, we assume that $\|v\|$ is bounded near infinity. \footnote{This is always the case for    $|k|= 2$ by Section~\ref{sec:norm-estimates}.}  

\item  A variation of type
\[
\xymatrix{  & I^{0,0} \ar[d]_{u^{-1,-1}} \ar[dr]^{u^{0,-2}}  \\
       I^{-2,0}  \ar[dr]_{u^{0,-2}} & \ar[l]  I^{-1,-1} \ar[l]_{u^{-1,1}}  \ar[l]_{u^{-1,1}}    \ar[d]^{u^{-1,-1}}   &\ar[l]_{u^{-1,1}}  I^{0,-2}    \\
 & I^{-2,-2}
}
 \]   
 and   $\ga =\geg^{-1,-1}$.
\end{enumerate}
Then  deformations of $F$ in the $\ga $-directions are in one-to-one correspondence with   those endomorphisms of $(H,Q  )$
that belong to  $\ga $ and  which intertwine the action of the monodromy.

 In particular,   the following  properties are equivalent:
\begin{itemize}
\item  $F$  has no horizontal  deformations    in   $\ga$-directions;
\item  $(H,Q)$ has  no  endomorphisms  in   $\ga$-directions intertwining the action of the monodromy.
\end{itemize}

These properties  hold  in particular, if  $F$ is  regularly tangent in the $\ga$-directions at $o\in S$ (and  hence  along $S$).

\end{thm}

\begin{proof}  In each of the above cases,  by Proposition~\ref{prop:harmendo}, a holomorphic horizontal endomorphism is 
plurisubharmonic and so its Hodge norm is plurisubharmonic. By  the results of Section~\ref{sec:norm-estimates},  this function is
    bounded.    Now apply  Theorem~\ref{thm:Main}. 
\end{proof}

%
%
%
%
%
%
\subsection{Conditions implying rigidity}

 Suppose we have a variation with two weights $0$ and $1$ and of 
  Hodge width   are $a$, respectively $b$.  Suppose that
$a>b$ and the weight $0$ variation is a direct sum of two variations, one having
maximal Higgs field and a piece $Z'$ of pure type $(0,0)$ with trivial Higgs field. We claim that this implies that  the mixed variation is  then  is
regularly tangent in the $(-1,0)$-directions.
To illustrate the set-up, we take $a=2$ and $b=1$:
\[
\xymatrix{
 H^{ 1,-1}   & \ar[l]^{(u^{-1,1})^*}  H^{0,0}  & H^{-1,1} \ar[l]^{(u^{-1,1})^*}   \\
& H^{1,0} \ar[u]_{v   }& H^{0,1}\ar[l]^{(u^{-1,1})^*} \ar[u]_{v }   .
}
\] 
Indeed,   by assumption, the upper row splits into two strands at most, that is $H^{0,0}= Z\oplus Z'$ such that
the  upper right component of $(u^{-1,1})^*$ maps isomorphically to 
  $Z$   which in turn is mapped isomorphically to $H^{-1,1}$ by the relevant component 
of $(u^{-1,1})^*$.

To test  regularity, suppose  that $[(u^{-1,1})^*  ,v]=0$.
The commutative diagram   implies that the image of $v: H^{1,0}\to H^{0,0}$
lands in $Z$ and so, if $(u^{-1,1})^*\comp v=0$ on $H^{1,0}$, we must have $v|_{H^{1,0}}=0$. 
Since then $(u^{-1,1})^*\comp v=0$ on $H^{0,1}$, a similar argument shows that $v=0$ on $H^{0,1}$ as well.
 
Now remark that if $h^{ 1,-1}=1$,  the Higgs field in the $u^{-1,1}$-direction is maximal precisely if
it is non-zero.  By Lemma~\ref{lem:OnInfPermap} 
this is the case if and only if  the tangent map to the weight zero period map in that direction    is  non-zero.
 We have shown: 
\begin{prop} Suppose we have a mixed period map $F$ for a variation of adjacent weights $0$ and $1$.
For the  pure weight $0$ variation we assume
\begin{itemize}
\item the only  non-zero Hodge numbers are $h^{-1,1}=h^{1,-1}= 1$  and $h^{0,0}\ge 1$.
\item its period map  is non-constant.
\end{itemize}   
 
Then $F$ is  regularly tangent in the $(-1,0)$-directions and hence admits no deformation in these directions. \label{prop:AdjWeights2}
\end{prop}

We finish this section by giving a criterion for rigidity using the monodromy action. It uses   the following general result.
\begin{lemma} Let $\pi$ be a group, $k$ a field and $V,V_1,V_2$   finite dimensional  $k$-vector spaces,
\[
 0 \to V_1 \mapright{i} V \mapright{p}  V_2\to 0
 \]
   an exact sequence of   $\pi$-modules and $\varphi\in \endo^\pi V$, i.e., an endomorphism of $V$ intertwining the $\pi$-action.
 Suppose
 \begin{itemize}
\item  $\varphi$  induces the zero map on $V_1$ and $V_2$.
\item $V_1$ is an irreducible $\pi$-module.
\item  $\dim V_1 >\dim V_2$.
\end{itemize}
Then $\varphi=0$. \label{lem:OnNeighbors} 
\end{lemma}

\begin{proof}  We claim that the  assumptions imply that the map $\varphi$  induces a $\pi$-equivariant morphism $\bar\varphi:  V_2\to V_1$
and if it is  the zero-map then $\varphi=0$. Let us  prove this claim. First we  define $\bar\varphi$.  Lift $\bar x\in V_2$ to an element $x\in  V$. Then $\varphi(x) \in i(V_1)$ since $\varphi$ is $\pi$-equivariant and induces the zero map on $V_2$. So $\varphi(x)= i(y)$. Then set
$\bar\varphi(\bar x)= y$. This is independent of the lift since $\varphi$ induces the zero map on $V_1$.
By  construction $\bar\varphi=0$ if and only $\varphi=0$. 

Since $V_1$ is irreducible as a $\pi$-module, by Schur's lemma, either $\bar\varphi=0$ or  
$\bar\varphi(V_2)=V_1$. In the latter case we would have $\dim V_2\ge \dim V_1$, contrary to the third assumption and hence $\varphi=0$.
\end{proof}

\begin{corr} 
\label{cor:TwoStepCrit} Consider a  period map $F:S\to \Gamma\backslash D$ associated   to a two-step weight 
filtration $0\subset W_{ 1}\subset W_2=H$. If the weight graded quotients have distinct dimensions and
  the one  of largest
dimension is an   irreducible $\Gamma$-module, then  $F$  is rigid  in the $(-1,0)$-directions.
So, if in addition the induced period maps for the weight-graded  pure 
variations  of Hodge structure  on $S$ are rigid,   then  $F$ is rigid as a period map.  
 \end{corr}
 \begin{proof} By duality we may assume that $\dim W_1> \dim \gr^W_2$.  We apply Lemma~\ref{lem:OnNeighbors} with $v\in \geg^{-1,0}$ playing the role of $\varphi$.
 So  $v=0$ and hence, by  Theorem~\ref{thm:MainCrit2},  $F$ is   rigid.
 \end{proof}

  \section{Examples of rigid  mixed period maps} \label{sec:Exampl}
  
   \subsection{Complements of smooth divisors}
  \label{ssec:Compl}
  
  Let $X$ be a smooth compact variety of dimension $d+1$ and $Y\subset X$ a smooth divisor.
  We let $i:Y\into X$ be the inclusion and $j:U=X\setminus Y \into X$ the inclusion of the complement.
  Then we have an exact sequence (in rational cohomology)
  \begin{align*}
    0 \to \coker(H^{k-2}(Y(-1)) \mapright{i_*}  H^k(X))& \mapright{j^*}    H^k(U) \mapright{\,\, r\quad} \\
    &  \hspace{-3em}\ker (H^{k-1}(Y)(-1)  
     \mapright{i_*}  H^{k+1}(X)) \to 0,
  \end{align*} 
   an extension of a weight $k+1$ Hodge structure by a weight $k$ Hodge structure.
   Since the category of pure polarized Hodge structures is abelian, there are
splittings $H^r(X)= \im( i_*)\oplus P^r(X)$,  and $H^r(Y)(-1)= \ker (i_*) \oplus V^{r+2}(Y)$ so the sequence reduces to
   \[
   0 \to P^k(X) \mapright{j^* }  H^k(U) \mapright{r}  V^{k+1}(Y) \to 0.
   \]
   If $Y$ is an ample divisor, this sequence is only interesting in the middle dimensions $d,d+1$ and simplifies to
   \begin{equation} \label{eqn:MHSOpen}
   0 \to H^{d+1}_\prim(X) \mapright{j^*} H^{d+1}(U) \mapright{r} H^d_\var(Y) (-1) \to 0.
   \end{equation}
   Suppose that we have a family of such  pairs $(X_s,Y_s)$, $s\in S$, with $S$ quasi-projective and smooth. 
   We give some applications of Eqn.~\eqref{eqn:MHSOpen}.
   
   First we invoke  Corollary~\ref{cor:TwoStepCrit}   and deduce:
   
   \begin{prop}  \label{prop:geom2neighbors} The  period map for 
$H^{d+1}(U) $ is rigid  in the $(-1,0)$-directions    if    the following two conditions hold simultaneously:

\begin{itemize}
\item the monodromy representation
on $H^{d+1}_\prim(X)$ is    irreducible;
\item  $\dim H^{d+1}_\prim(X)$ $> \dim H^d_\var(Y) $.
\end{itemize}

 If, in addition,       the period maps  associated to \allowbreak$ H^{d+1}_\prim(X)$ 
   and $H^d_\var(Y)  $ are  
 rigid, then  the period map is  rigid  in all horizontal directions.
     \end{prop}

   \begin{exmple} \label{exmpls:OpenRig}  
 The obvious  example is a family $\set{C_s\setminus \Sigma_s}$ of
  \textbf{\emph{quasi-projective smooth curves}}. If    the monodromy acts irreducibly on $H^1(C_s)_\bC$ and if also $\#\Sigma < 2g$, the
  mixed period map   is rigid in the $(-1,0)$-directions.  
  
  More generally, we can consider the Hodge structure  on  $H^1(X)$ for  $X$ of any dimension (and for $H^0(Y)$).  
For instance  take any rigid family  of abelian varieties  
(see    Section~\ref{ssec:ExRigPure} Example (6))   and leave out a smooth, possibly reducible  divisor.
If  the monodromy action is   irreducible and $Y$ does not have too many components,  the
  mixed period map will  again be  rigid.  
 \end{exmple}
 
We next we use Eqn.~\eqref{eqn:MHSOpen}  in conjunction with   Proposition~\ref{prop:AdjWeights2}.
So we  start from   a \textbf{\emph{K3-type}} Hodge structure  that is, we recall, a   weight two  Hodge structure  with $h^{2,0}=h^{0,2}=1$ and $h^{p,q}=0$ for $p<0$ or $q<0$.    As a consequence of Proposition~\ref{prop:AdjWeights2} we have:

\begin{prop} Suppose  that $H^2_\prim (X_s)$   is a non-constant variation of K3-type Hodge structure. 
Then the mixed period map for $H^2(X_s\setminus  Y_s)$ is rigid in  the directions of type $(-1,0)$.
 The above holds in particular for $X_s$ a K3 surface.
\label{prop:OnK3Vars}
 \end{prop}
%
 \begin{rmk}  
   To obtain examples with rigidity in all horizontal directions, one can consult the examples in Section~\ref{ssec:ExRigPure},
in particular   Proposition~\ref{prop:RigidK3}.  
 \end{rmk}
 
One can handle  many more geometric examples based on the remark that
surfaces   with $h^{2,0}=1$ have K3 type Hodge structure on $H^2$ and $H^2_\prim$. Let us especially consider the case of regular surfaces, that is
surfaces with  $b_1=0$, that are moreover minimal and of general type. By \cite[Thm. VII.2.1]{4authors} one then has $K^2=1,\dots, 8$ and one finds
$h^{1,1}= 20 - K^2$ so that the period domain for the primitive cohomology has dimension 
\[
d(H^2_\prim)=  19- K^2.
\]
Since $p_g=1$, there is a unique canonical curve $K$ of arithmetic genus $K^2+1$.

For the purpose of this article we say that  $X$ is a \textbf{\emph{Catanese--Kynev--Todorov  surface}} or
\textbf{\emph{CKT-surface}} if $X$ is a   simply connected Galois $\bZ/2\bZ\times \bZ/2\bZ$ cover  of the plane
with  Hodge numbers $h^{2,0}(X)=1$, $h^{1,1}(X)=19$ (and so   $K_X^2=1$). These were first constructed by V. Kynev \cite{kynev} and investigated in detail by  F. Catanese \cite{catanese} and A. Todorov
\cite{todorov2}.
 Let us recall (loc. cit.)  some of their properties.   The quotient by one of the involutions  is a  double cover of $\bP^2$ which  is 
 branched in the union of two cubics meeting transversely.
 This is a  K3 surface $Z$ with  $9$ ordinary double points. The family of such $Z$ depends on $10$ effective parameters and 
 the period domain is a linear section $D_2\cap L$  of codimension $9$ of the period domain $D_2$ for K3 surfaces with  a degree $2$ polarization. 
 In other words, the K3 family has a period map which is generically one-to-one onto a suitable quotient of $D_2\cap L$ .
 Over  a general line lies a smooth genus $2$ curve $C$  in $Y$. Branching in  $C$  and the $9$
 ordinary double points  produces the desired surface of general type.  Since there is an   ample divisor and  $9$ smooth rational
 curves of self-intersection $-2$ on the K3 surface, this shows that
  the Picard number of the general member   is  at least  $1+9=10$. Equality follows from   the surjectivity of the period map for $Z$.
In constructing the second double cover, the choice of the line gives $2$ extra parameters which do not vary with  the Hodge structure and so
for those surfaces the period map has fibers of dimension $2$. The resulting surfaces of general type depend on $10+2=12$ moduli.
  
A. Todorov \cite{todorov} has generalized the above construction to give surfaces of general type with $b_1=0$, $h^{2,0}=1$ and $K^2=2,\dots, 8$. We call these \textbf{\emph{Todorov surfaces}}.
These are birational to double covers of a  classical Kummer surface, branched in a quadratic section passing through $8- K^2$
double points plus  the remaining $8+K^2$ double points. These last double points resolve to $-2$ curves on the K3 surface
and the  resulting family has $19- ( 8+K^2)= 11 - K^2$ moduli.
The  choice of the 
quadric section adds $  K^2 +1$ parameters which do not vary with  the Hodge structure and so, as before,  we get in total $12$
parameters and the period map has fibers of dimension $  K^2 +1$. To calculate the generic Picard number,
note that the $8- K^2$ double points through which the curve passes
give just as  many $(-2)$ curves and there are $3$ more independent divisors on   the Kummer surface we started with.
 The results have been summarized in Table~\ref{tab:OnTodSurfs}. In  the table $d(H^2_\prim)$ stands for
 the dimension of the period domain for the weight two  K3-variation with period map $F_2$, "moduli" stands for
the number of moduli of the CKT-  and Todorov-surfaces\footnote{The full moduli space for surfaces with these invariants is expected to be (much)  larger.  See for example \cite{catanese,catdebar}
for $K^2=1,2$.}, $\rho$ is the generic Picard number of the K3 surface,  
 $\dim W_2$ is the dimension of the essential part of the
variation and $\dim (W_3/W_2)= 2 g(K_s)=\dim H^1(K_s)$.
\begin{table}[htp]
\caption{Invariants for   open  CKT   and Todorov modular families }
\begin{center}
\begin{tabular}{|c|c|c|c|c|c|}
\hline
$K^2 $ & $d(H^2_\prim)$ &   moduli & $\rho$  &  fiber dim. of $F_2$ & $(\dim W_2,\dim W_3/W_2)$ \\
\hline
           $1$& 18  & 12 & 10     &2 & (11, 4) \\
\hline $2$ & 17 & 12 & 9         & 3 & (11, 6)\\
\hline $3$ & 16 & 12 & 8          & 4 & (11, 8)\\
\hline $4$ & 15 & 12 & 7         & 5& (11, 10)\\
\hline $5$ & 14 & 12 & 6      &6& (11, 12)\\
\hline $6$ & 13 & 12 & 5    & 7& (11, 14)\\
\hline $7$ & 12 & 12 &4      & 8 & ( 11, 16)\\
\hline $8$ & 11 & 12 & 3      & 9& (11, 18)\\
\hline
 \end{tabular}
 \end{center}
\label{tab:OnTodSurfs}
\end{table}%
\medskip

  The main result about these surfaces is as follows:
  
  \begin{prop} Let $\set{X_s}_{s\in S}$ be a family of CKT-surfaces  or  of Todorov surfaces and let $K_s\subset X_s$ be the canonical curve.   The family $\set{X_s\setminus K_s}$
is rigid if all of the following conditions hold: \label{prop:CKT-surfs}

\begin{enumerate}

\item the family $\set{K_s}_{s\in S}$ of curves is rigid.

\item  The essential part of the K3 variation  is  non-constant and   rigid. 

\item The mixed period map   is  an immersion.  

\end{enumerate}
 These conditions are all satisfied for a modular family, that is, a family  with $12$ effective parameters.
 This is also the case  for any subfamily of a modular family for which the essential part of the K3 variation has rank $11$. This 
 is the case for a sufficiently generic subfamily of a modular family.
  
  \end{prop}
  \begin{proof}
  First of all,  since by (2)  the K3 variation is non-constant,  Proposition~\ref{prop:OnK3Vars} implies that the variation is rigid in $(-1,0)$-directions. 
   It is rigid in $(-1,1)$ directions
  if this is the case for the pure   variations coming from the curves as well as for  the K3 variation. Assumption (1)  covers  the curve 
  case (since we are  interested in the variations coming from the geometry of the open surfaces) and (2) covers the
  K3 variation.
  Condition (3) then implies that the family of open surfaces  is   itself   rigid whenever the mixed variation is rigid.
  
  Condition (1) holds as soon as the period map for the curves is an immersion. This is   a consequence of 
   Arakelov's theorem, recalled in Section~\ref{ssec:history}. For a modular family this is the case.
   Indeed, for a modular family the period for the  fibers of  $F_2$  is injective.
  By Proposition~\ref{prop:RigidK3}, the second condition is   satisfied if the rank of the essential variation is not divisible by $4$.
  From the table we see that this is the case for  a modular family.
      
 The third condition is a bit more involved since  the pure K3 variation does not determine the family  because of the failure of  infinitesimal  Torelli.
  Indeed, this is exactly the reason they were constructed! 
The failure of infinitesimal Torelli  is caused by the non-trivial kernel of the  tangent map to the K3 period map. Since $T_X\simeq \Omega^1_X\otimes K_X^{-1}$, the tangent to the period map is the map 
\[
u^{(2)}_S:  T_S\to H^1(T_X) \to \hom\left(H^0(K_X) \to H^1(\Omega^1_X)\simeq  H^1(\Omega^1_X)\right),
\]
where the resulting morphism on the right,
\[
\mu_X:   H^1(T_X)=H^1(\Omega^1_X(-K)) \to  H^1(\Omega^1_X) ,
\]
   is induced by multiplication by a non-zero section of $K_X$ vanishing along the canonical curve $K$. From the exact sequence
   \[
   0= H^0(\Omega^1_X)  \to H^0(\Omega^1_X|K)\to H^1(\Omega^1_X(-K)) \to H^1(\Omega^1_X),
   \]
   one sees that the 
    kernel of $\mu_X$ is isomorphic to $H^0(\Omega^1_X|K)$.
To interpret this space,  recall that, as observed by A. Todorov \cite[proof of Prop. 4.1]{todorov2} and F. Catanese \cite[p. 150]{ttag}  the involution $\tau$ on $X$ that produces the K3-quotient 
induces a splitting of  the exact sequence
\[
0\to \cO_K (-K)  \to \Omega^1_X|_{K} \to \Omega^1_{K} \to 0 .
\]
Indeed, local coordinates $(x,y)$ centered at a point $P$ of $K$ can be chosen in such a way that $x=0$ gives the canonical curve $K$
and $\tau^*x=-x, \tau^*y=y$. Then the eigenspace decomposition   of $\Omega^1_{X,P}$ is just  $\bC (dx)_P\oplus \bC (dy)_P$ and this gives a
global splitting along $K$ with the first factor giving $\cO_{K}({K})$  and the second $\Omega^1_{K}$.
 For a modular family  $T_sS \simeq  H^1(T_X) $ and then the split sequence shows that the kernel of the Higgs field $u^{(2)}_S$ is isomorphic to $H^{0}(\Omega^1_{K})$. Its dimension, $K^2+1$,  is 
 the genus of the canonical curve $ K $,  as  indicated in Table~\ref{tab:OnTodSurfs}.
This kernel is captured by the cup product
\[
\mu_K : H^1(T_K)    \to \hom(H^{1,0}({K}), H^{0,1}({K})),
\]
which is injective  (infinitesimal Torelli) since  by \cite[Lemma 5.2]{todorov}, the canonical curve is non-hyperelliptic for the Todorov surfaces.
The Higgs field   $u^{(1)}_S$ for the pure weight $1$ variation is the composition of the map $T_sS \to H^1(T_K)$ and $\mu_K$
and it is generically injective for a modular family.   Combining the two calculations, we have shown that
the kernel of the partial mixed Higgs field $u^{(2)}_S+ u^{(1)}_S$ is trivial and so the mixed period map is an immersion. Hence 
  $X_s\setminus K_s$ can be locally reconstructed from the  period map. \footnote{For Kynev--Todorov surfaces  one can also use  
  M. Letizia's argument \cite{kunev}   showing  that the mixed Hodge structure generically determines the pair consisting of the surface and its canonical curve.}
  For a subfamily this is also the case.   
  \end{proof}

\subsection{Projective varieties singular along a smooth divisor}
\label{ssec:Sings} 

 Let $X$ be a   compact variety of dimension $d+1$ whose singular locus  $Y $  is a smooth divisor.
  We let $\sigma: \tilde X \to X$ be the desingularization of $X$ and set $\tilde Y=\sigma^{-1} Y$,
  $i: Y\into X$, 
  $\tilde {\imath}:\tilde Y\into \tilde X$ be the inclusions.
  Then by \cite[\S5.3.2]{mht}  we have an exact sequence of rational cohomology groups
  \begin{align*}
    0 \to \coker\left(H^{k-1}(\tilde X)\oplus H^{k-1}(Y)  \mapright{\tilde{\imath}^*-\sigma^*}  H^{k-1}(\tilde Y)\right)& \mapright{\quad }    H^k(X) \mapright{\quad} \\
    &  \hspace{-10em}\ker\left(H^{k}(\tilde X)\oplus H^{k}(Y)  \mapright{\tilde{\imath}^*-\sigma^*}  H^{k}(\tilde Y)\right)  \to 0. 
       \end{align*} 
In this case     $\sigma: \tilde Y\to Y$ is an unramified double cover, $\coker \sigma^*$ is the
anti-invariant part of the cohomology and $\sigma^*$ is an embedding.
Assuming that $\tilde Y$ is a hyperplane section (or, more generally, very ample), in the middle dimension, the kernel of of
$ \tilde \imath^* $ is  then the  variable cohomology.
Hence  the sequence reduces to
\[
0 \to H^{d}_\prim(\tilde Y)^{-} \to H^{d+1}(X) \to H^{d+1}_\var(\tilde X) \to 0.
\]
As a consequence of   Corollary~\ref{cor:TwoStepCrit}, we have
   
   \begin{prop}  \label{prop:geom3neighbors} Suppose that  the monodromy representation on $H^{d}_\prim(\tilde Y_s)^-$ is irreducible and  that
\[
\rank (H^{d}_\prim(\tilde Y_s)^-)> \rank (H^{d+1}_\var(\tilde X_s)).
\]
 Then the mixed period map for $H^d(X_s)$ is rigid in the $(-1,0)$-directions. 
    \end{prop}

\begin{exmple} \label{ex:CurvWithDPS}
  Projective curves with $\delta$ ordinary double points. Here $d=0$ and we get
\[
0 \to \oplus^\delta \bZ  \to H^1(X) \to H^1(\tilde X) \to 0.
\]
The mixed period map is rigid in $(-1,0)$-directions, if the monodromy 
of the family of curves acts    irreducibly
on the set of double points  of which there are many ($\delta> 2g$). This result is dual
to the case of an open curve treated in Example~\ref{exmpls:OpenRig} (a).

By Proposition~\ref{prop:maxHiggsIsRigid},    rigidity for the  pure  variation follows  if  the Higgs field   is maximal  
and   for weight one this  is the case    for ``most'' period maps.
 \end{exmple}
 
%
%
%

\subsection{Normal functions and higher normal functions}
\label{ssec:NormFncs}

  Recall that these are associated to a variation of the form $ H_{2p+1}(X_s)(-p)$ where $\set{X_s}_{s\in S}$
  is a family of smooth complex projective varieties equipped with a family $Z_s$ of $p$-dimensional algebraic cycles
  homologous to zero proving a variation of mixed Hodge structure
  \[
 0 \to H_{2p+1}(X_s)(-p) \to \cH^p(Z/S)  \to   \bZ(0)   \to 0.
\]
 As a consequence of  Corollary~\ref{cor:TwoStepCrit}   we have:
   
   \begin{prop}  \label{prop:norm}       If   the monodromy acts irreducibly on $H_{2p+1}(X_s) $,  the normal function  $\cH^p(Z/S)$  is   rigid in $(-1,0)$-directions. If, moreover, the   period map   associated to $H_{2p+1}(X_s) $   
   is    rigid, the normal function is rigid in all directions.
 \end{prop} 
 
    As an example, for $p$ even we have a  normal function associated to cycles in a Lefschetz pencil of complete intersections. Also normal functions for  certain K3-variations, abelian varieties and   Calabi--Yau's give examples of normal functions, rigid in all directions.
    See the examples in Section~\ref{prop:RigidK3}.

 A similar result holds for higher normal functions
 \[
 0 \to H^{p-1}(X_s)(q) \to  \cH^{p,q} \to   \bQ(0) \to 0
 \]
with $p-2q-1<0$. Here we have rigidity for $\cH^{p,q} $ in $(-1, k)$-directions with $k= p-2q-1$ provided 
for these directions boundedness for the Hodge norm at infinity holds.
%
%

\subsection{Unipotent variations}
\label{ssec:Unipot}
 We consider adjacent weights and rigidity in $(-1,0)$-directions only:
\[
\xymatrix{
  H^{ p,q}   \ar@<1ex>[r]^{v^{-1,0}}&  \ar@<1ex>[l]^{(u^{-1,0})^*}   H^{p+1,q}     .
}
\] 
Such a $v$ is regularly tangent if  for some $u$  the relation $u^*\comp v=0$ implies $v=0$ which
is the case if $u$ is surjective (then $v^*\comp u=0 $ is equivalent to $v^*=0$.) More generally
this is the case if for given  $x\in  H^{p+1,q}$ we can find $u=u_x$ with $x$ in its image since then $v^*(x)=
v^*\comp u_x (x') =0$ by assumption.

In  \cite[Thm. 3.6]{pdmixed} we considered the differential geometric aspects of unipotent variations of mixed Hodge structures 
associated  the based fundamental group of $X$ when the base point $x\in X$ varies. The set-up is detailed in Section~\ref{ssec:PSH},
example (6).  If we vary the base point in a submanifold $S\subset X$, by \cite[Lemma 6.8]{pdmixed}, the Higgs field comes
from a map
\[
u: \ker [ H^1(X)\otimes H^1(X)\to H^2(X)] \otimes T^{1,0}_sS \to H^1(X)
\]
given by 
\[
\alpha \otimes \beta \otimes \theta \mapsto (\theta\intprod \alpha ) \beta - (\theta\intprod\beta)\alpha.
\]
This map is of Hodge type $(-1,0)$ since it sends $I^{2,0}\subset H^{1,0}\otimes H^{1,0}$ to  $H^{1,0}$ and
$I^{1,1}\subset H^{1,0}\otimes   H^{0,1}$ to  $H^{0,1}$. Moreover, the restriction to $I^{2,0}$ determines
the entire morphism. Note also that  $u$   factors
through $ \ker [  \Lambda^2 H^1(X) \to  H^{2 }(X)]$.
Let $V=H^{1,0}(X)$, $K=\ker [  \Lambda^2 V \to  H^{2,0 }(X)]$,    $T=T_sS$ and consider the maps $e:T\to  V^*$ given by $e_\theta (\omega )= \theta\intprod \omega$ and 
\begin{equation}
\label{eqn:OnMHSonFundGrps}
u: K\otimes T \to V,\quad  \sum_{i,j,k}    (  \omega_i\wedge \omega_j )\otimes \theta_k
=  \sum_{i,j,k}   [e_{\theta_k} (\omega_i)\omega_j   - e_{\theta_k }(\omega_j)\omega_i].
\end{equation} 
If this map is surjective, for every $\omega\in V$ we can find $\theta_j\in T$ and $A^j\in K$ such that
$\sum_j u (A^j\otimes \theta_j)= \omega$ which suffices to show regular tangency. We formulate the
conclusion explicitly:

\begin{prop} \label{prop:OnMHSonFundGrps} Let $X$ be a smooth projective variety and consider the variation of mixed Hodge structure
on $\hom_\bZ( J_x/J_x^3,\bC)$ where  $x$ varies over a smooth subvariety $S\subset X$.
If the map $u$ from \eqref{eqn:OnMHSonFundGrps} is surjective,  the variation is rigid,
\end{prop}

To see what this means geometrically, suppose for instance that  there is a generic direction $\theta$ such that
$u_\theta(A)= u(A\otimes \theta)=0$ imposes $ \dim K-\dim V$ independent conditions on $K$. Then the map $u_\theta$
is surjective  which implies regular tangency. 
Since the condition $u_\theta (A)=0$ amounts to $\dim V$ equations on $A$, the latter condition 
  can only hold if
$\dim K \ge 2\dim V$ and if so,     for generic $\theta$ these equations are expected to be independent.
  Depending on the geometry of the cotangent bundle this then is the case or not.

 \appendix 
\section{Admissibility}
\label{sec:admis}

\par In~\cite{SZ}, J. Steenbrink and S. Zucker defined a category of
admissible variations of graded-polarizable mixed Hodge structure over a
punctured disk $\Delta^*$ with unipotent monodromy.  This definition
can be modified to handle the case of quasi-unipotent monodromy via
a covering trick (see \S 1.8 of~\cite{kashi}).  Given this local model, the
category of admissible variations of graded-polarized mixed Hodge structure
over a smooth complex algebraic curve $C$ is defined as follows:  The curve
$C$ has a smooth completion $\bar C$ which is unique up to isomorphism.  A
graded-polarizable variation
$\cH\to C$ is admissible if and only if for each $p\in\bar C-C$ the
restriction of $\cH$ to a deleted neighborhood of $p$ is admissible.
  
\par In higher dimensions, let $S$ be a smooth quasi-projective variety
over $\bC$ and $j:S\to\bar S$ be a smooth partial compactification
of $\bar S$ such that $\bar S-S$ is a normal crossing divisor.
In~\cite{kashi}, M. Kashiwara showed that one obtains a good category
of admissible variations of graded-polarizable mixed Hodge structure on
$S$ via a curve test.  In particular, the admissibility of $\cH$
does not depend on the choice of $j:S\to\bar S$.

\par Implicit in the previous paragraph is the assumption that the local
monodromy of $\cH$ is quasi-unipotent, which we shall assume throughout
this appendix.   This is automatic whenever $\cH$ carries an integral
structure $\cH_{\mathbb Z}$ (e.g. variations of geometric origin). To
continue, we recall that if $f:A\to B$ is a holomorphic map between complex
manifolds and $\cH$ is a variation of graded-polarizable mixed Hodge
structure on $B$ then, $f^*(\cH)$ is a variation of graded-polarizable
mixed Hodge structure on $A$ (see \S 1.7,~\cite{kashi}).  

\par We now recall the definition of an admissible variation of mixed Hodge
structure over the punctured disk with unipotent monodromy following
Steenbrink and Zucker:  Let $\Delta=\{\, s\in\bC \mid |s|<1\,\}$ and
$\Delta^*=\Delta-\{0\}$.  Let $\cH\to\Delta^*$ be a variation of
graded-polarizable mixed Hodge structure.  Let $U$ denote the upper
half-plane $\{\, z=x+iy\in\bC \mid y>0\,\}$, and $U\to\Delta^*$ be
the covering map $s=e^{2\pi \ii z}$.

\par After selecting a choice of graded-polarization (in order to define
the classifying space $D$), the period map of
$\cH$ fits into a commutative diagram
\begin{equation}
\begin{CD}
            U  @> F >>            D \\
            @V s  VV    @VVV  \implies  F(z+1)= T.F(z)\\
            \Delta^{* } @> \varphi >> \langle T \rangle \backslash D
\end{CD}                                           \label{eq:mixed-diagram}
\end{equation}
where $T=e^N$.  Accordingly, the map
\[
      \tilde\psi:U\to\check{D},\qquad \tilde\psi(z) = e^{-zN}.F(z)
\] 
satisfies $\tilde\psi(z+1)=\tilde\psi(z)$ and hence descends to a map
$\psi:\Delta^*\to\check{D}$.

\par By Schmid's nilpotent orbit theorem (Thm (4.12),~\cite{schmid}), if
$\cH$ is pure then
\begin{equation}
  \lim_{s\to 0}\, \psi(s) = F_{\infty}\in\check{D}
     \label{eq:admissibility-1}
\end{equation}  
exists.  Moreover, $N(F^p_{\infty})\subseteq F^{p-1}_{\infty}$ and there exists a
constant $a$ such that $\Im(z)>a\implies e^{zN}.F_{\infty}\in D$.  Finally,
given a $G_{\bR}$-invariant metric on $D$, there exist constants $K$ and
$b$ such that 
\[ 
      \Im(z)>a \implies d_D(F(z),e^{zN}.F_{\infty})<K\Im(z)^b e^{-2\pi\Im(z)}
\]
      
\begin{rmq} Schmid's result also covers the case of pure variations of Hodge
structure with quasi-unipotent monodromy by passage to a finite cover.  If
$t$ is another choice of holomorphic coordinate on $\Delta$ which vanishes
at $0\in\Delta$ then tracing through the above construction shows that
the resulting limit filtration is related to \eqref{eq:admissibility-1}
by the action of $e^{\lambda N}$ where $\lambda$ depends on $(ds/dt)_0$. 
\end{rmq}  

\par In contrast, the mixed period domain $D'$ with Hodge numbers
$h^{1,1}=h^{0,0}=1$ is isomorphic to $\bC$ and has trivial infinitesimal
period relation.  Accordingly, the period map $\varphi:\bC^*\to D'$
given by $\varphi(s) = e^{1/s}$ arises from a Hodge--Tate variation with
trivial monodromy which does not have limit Hodge filtration.

\par Let $V$ be a finite dimensional vector space and $W$ be an increasing
filtration of $V$.  Then, $W[j]$ is the filtration $W[j]_k = W_{j+k}$.
Given a nilpotent endomorphism $N$ of a finite dimensional vector space $V$,
it follows from upon writing $N$ in Jordan canonical form that exists a unique,
increasing monodromy weight filtration $W(N)$ of $V$ such that
\begin{itemize}
\item[---] $N(W_k)\subseteq W_{k-2}$;
\item[---] $N^k:\gr^W_k\to \gr^W_{-k}$ is an isomorphism
\end{itemize}
for each $k$.

\par Suppose instead that $V$ is equipped with an increasing
filtration $W$ such that $N(W_k)\subseteq W_k$ for each  index $k$.
Then, there exists at most one increasing filtration $M=M(N,W)$ of
$V$ such that
\begin{itemize}
\item[---] $N(M_k)\subset M_{k-2}$;
\item[---] $M$ induces on $\gr^W_k$ the filtration $W(N:\gr^W_k\to \gr^W_k)[-k]$.
\end{itemize}
If $M$ exists it is called the relative weight filtration of $W$ with respect
to $N$.  In general, $M(N,W)$ does not exist.  For example, if $W$ has only
two non-trivial weight graded quotients which are adjacent (e.g. $\gr^W_0$
and $\gr^W_{-1}$) then $M(N,W)$ exists if and only if $W$ has an $N$-invariant
splitting.

\begin{dfnA} Let $\cH\to\Delta^*$ be a variation of graded-polarized
mixed Hodge structure with unipotent monodromy $T=e^N$ and weight filtration
$W$.  Let $\varphi:\Delta^*\to\langle T\rangle\backslash D$ be the period map
of $\cH$.  Then, $\cH$ is admissible if
\begin{itemize}
\item[(a)] The limit Hodge filtration \eqref{eq:admissibility-1} exists;
\item[(b)] The relative weight filtration $M=W(N,W)$ exists.
\end{itemize}
A variation of graded-polarized mixed Hodge structure $\cH\to\Delta^*$
with quasi-unipotent monodromy is admissible if the pullback $f^*(\cH)$
to a finite covering of $\Delta^*$ with unipotent monodromy is admissible.
\end{dfnA}

\begin{rmq} See \S 3 of~\cite{SZ} and \S 1.8--1.9 of~\cite{kashi} for the
definition of admissibility in terms of the canonical extension of
$\cH$ to a system of holomorphic vector bundles over $\Delta$.
\end{rmq}

\par An increasing filtration $W$ of a vector space $V$ is pure of weight $k$
if $\gr^w_{\ell} = 0$ for $\ell\neq k$ and $\gr^W_k\cong V$.  Reviewing the
definition of $M=M(N,W)$ it follows that if $W$ is pure of weight $k$ then
$M=W(N)[-k]$ (Prop. (2.11),~\cite{SZ}).

\begin{corrA} If $\cH\to\Delta^*$ is a variation of pure, polarized
Hodge structure then $\cH$ is admissible.
\end{corrA}
\begin{proof} The limit Hodge filtration exists by Schmid's nilpotent orbit
theorem, and the relative weight filtration exists by the previous paragraph.
\end{proof}


\par In the pure case, it follows from Schmid's $\slgr 2$-orbit theorem
(Thm. (5.13),~\cite{schmid}) that if $\varphi$ is a the period map
of a variation of polarizable Hodge structure $\cH\to\Delta^*$ of
weight $k$ with unipotent monodromy $T=e^N$ then 
\begin{equation}
    (F_{\infty},W(N)[-k])
\end{equation}  
is a mixed Hodge structure relative to which $N$ is a $(-1,-1)$-morphism,
where $F_{\infty}$ is the limit Hodge filtration \eqref{eq:admissibility-1}.
Moreover, it follows from the $\slgr 2$-orbit theorem (Thm. (6.6) and
Cor. (6.7),~\cite{schmid}) that the Hodge norm of a flat section of
$\cH$ is bounded.


\par One of the main results of~\cite{SZ} is that if $\cH\to\Delta^*$
is an admissible variation of graded-polarized mixed Hodge structure then
$(F_{\infty},M)$ is a mixed Hodge structure relative to which $N$ is a
$(-1,-1)$-morphism.  In particular $N(F^p_{\infty})\subseteq F^{p-1}_{\infty}$.
Moreover if
\begin{equation}
       \theta(z) = e^{zN}.F_{\infty}   ,    \label{eq:admissibility-2}
\end{equation}
then there exists a constant $a>0$ such that
$\Im(z)>a \implies \theta(z)\in D$.  Finally, by~\cite{dmj} it follows that
there exists constants $K$ and $b$ such that
$$
      \Im(z)>a \implies d_D(F(z),\theta(z))\leq K\Im(z)^b e^{-2\pi\Im(z)}.
$$

\begin{dfnA} Let $D$ be a classifying space of graded-polarized mixed Hodge
structure with underlying filtration $W$ and associated real Lie
algebra $\geg_{\bR}$.  Then, the pair $(N,F)$ consisting of an
element $N\in\geg_{\bR}$ and $F\in\check{D}$ defines an
admissible nilpotent orbit $\theta(z) = e^{zN}.F$ if
\begin{itemize}
\item[(a)] $N(F^p)\subseteq F^{p-1}$;
\item[(b)] The relative weight filtration $M=M(N,W)$ exists;
\item[(c)] There exists $a$ such that $\Im(z)>a\implies \theta(z)\in D$.
\end{itemize}
\end{dfnA}

\par The foundations of the theory of admissible nilpotent orbits of
graded-polarized mixed Hodge structure is given by Kashiwara in~\cite{kashi},
where they are called infinitesimal mixed Hodge modules.  In the pure case,
a strengthened form of Schmid's several variable nilpotent orbit theorem as
well as the several variable $\slgr 2$-orbit theorem appear
in~\cite{degeneration}.

\section{Properly Discontinuous Actions on Mixed Period Domains}
\label{sec:ProperDisc}

\par Let $\cH\to S$ be a variation of graded-polarized mixed Hodge
structure on a complex manifold $S$.  Let $\rho:\pi_1(S,b)\to G_{\bR}$
be the monodromy representation of $\cH$ on the reference fiber
$V=\cH_b$ and $D=G/G^F$ be the classifying space of graded-polarized
mixed Hodge structure defined by $\cH_b$.  Let $W$ denote the weight
filtration of $V$.
 
\begin{propA} If $\Gamma$ is discrete and closed in $G_{\bR}$ then
$\Gamma$ acts properly discontinuously on $D$, and hence the quotient
$\Gamma\backslash D$ is a complex analytic space.   
\end{propA}
\begin{proof} In the case where $\cH$ is a variation of pure Hodge
structure, this result is well known from the work of P.\ Griffiths,
and boils down to the fact that, in the pure case, the stabilizer
$G_{\bR}^{F}$ a point $F\in D$ is compact.

\par Turning to the mixed case, let $K$ and $K'$ be compact subsets of $D$.
The map from $D$ to the graded classifying spaces $D_j$ is continuous, 
and hence the respective images $K_j$ and $K'_j$ of $K$ and $K'$ in $D_j $
are compact for all $j$.  If $\Gamma$ does not act properly discontinuously,
there exist an infinite set of distinct elements ${g_n}\in \Gamma$ such that
$g_n(K) \cap K'$ is non-empty for all $n$. Then
$((\gr^W  \!\!g_n)  K_j)  \cap K'_j $ is non-empty for all $j$ and $n$.  Since,
by P.\ Griffiths' results, the action of $\gr^W\Gamma$  on each $D_j$ is properly
discontinuous, it follows that the set $ \set{\gr^W  \!\!g_n  }$  contains
only finitely many elements.  Thus, after partitioning $\set{g_n}$ into a
finite collection of subsets, we may assume that  there exists
$h\in \Gamma$ such that for all $n$ we have $\gr^W  \! g_n = \gr^W  \!  h $
for an infinite collection $\set{g_n}$. From this we shall derive a
contradiction.

\par To this end, we introduce the complex, unipotent Lie group
$$
    U_{\bC} = \{g\in \gl {V_{\bC}}) \mid (g-\id)W_k\subset W_{k-1}\}
$$
and let $U_{\bR} = U_{\bC}\cap \gl{ V_{\bR}}$.  Observe that
$u_n := g_n h^{-1}\in U_{\bR}$ for each index $n$, since $g_n$ and $h$
induce the same action on $\gr^W$.  
      

\par To continue let $\mathcal Y$ denote the set of all (complex) gradings
of $W$ (see section \ref{subsect:gradings-splittings}).  Then,  the group
$G_{\bC}$ acts continuously on $\mathcal Y$ via the adjoint action.
Moreover, by (2.2, \cite{degeneration}), the subgroup $U_{\bC}$ acts
simply transitively on $\mathcal Y$.  Furthermore, the map
\[
Y: D \to  \cY,\quad F\mapsto Y(F), \text{ the Deligne grading of } (F,W) 
\]
is continuous, and hence both $Y(K)$ and $Y(K')$ are compact subset of
$\mathcal Y$.   By construction,
\[
       Y(g.\tilde F) = g.Y(\tilde F)
\]
for any $\tilde F\in D$ and $g\in G_{\bR}$.
Applying this to  $g_n, h\in G_{\bR}$,  we find 
\[
Y(g_n(K)) = g_n \cdot Y(K) =( u_n  h)\cdot  Y(K)= u_n (h \cdot Y(K)),
\]
with $h \cdot Y(K)$ compact. So our question is: for how many
$u_n\in U_{\bR}$ can $u_n \cdot h \cdot Y(K)$  intersect $Y(K')$?

\par Fix $Y_o\in\mathcal Y$.  Since $U_{\bC}$ acts simply transitively
upon $\mathcal Y$, it follows that there are compact subsets $C'$ and $C''$
of $U_{\bC}$ such that
\[
       Y(K') = C'\cdot Y_o,\qquad h\cdot Y(K) = C''\cdot Y_o
\]
So, if $u_n \cdot h \cdot Y(K) $  intersects $Y(K')$  then there exist
elements $c'\in C'$ and $c''\in C''$ such that
\[
    u_nc'' \cdot Y_o = c'\cdot Y_o
\]
By simple transitivity, $u_nc'' = c'$ and hence $u_n$ belongs to the compact
set $C = C'(C'')^{-1}$.  Equivalently, $g_n = u_n h$ belongs to the compact
subset $C\cdot h\subset G_{\bC}$.

\par By hypothesis, the image of $\Gamma$ in $G_{\bR}$ (and hence
$G_{\bC}$) is discrete and closed.  As $C\cdot h$ is compact, it can
contain only finitely many elements $g_n$ from $\Gamma$, which contradicts
the supposition that there infinitely many elements $g_n\in\Gamma$
such that $\gr^W(g_n) = \gr^W(h)$.  Hence  $\Gamma$ acts properly
discontinuously on $D$.
\end{proof}

\bibliographystyle{plain}
\bibliography{bisect-arxiv.bib} 

\begin{thebibliography}{10}

\bibitem{ara}
S.~Ju. Arakelov.
\newblock Families of algebraic curves with fixed degeneracies. ({R}ussian).
\newblock {\em Izv. Akad. Nauk SSSR Ser. Mat}, 35:1269--1293, 1971.

\bibitem{4authors}
W.~Barth, K.~Hulek, C.~Peters, and A.~Van de~Ven.
\newblock {\em Compact {C}omplex {S}urfaces}.
\newblock Number~4 in Ergebn. der Math. 3. Folge. Springer-Verlag, Berlin etc.,
  second enlarged edition, 2004.

\bibitem{BP2}
P.~Brosnan and G~Pearlstein.
\newblock On the algebraicity of the zero locus of an admissible normal
  function.
\newblock {\em Compos. Math.}, 149:1913--1962, 2013.

\bibitem{BPR}
P.~Brosnan, G.~Pearlstein, and C.~Robles.
\newblock Nilpotent cones and their representation theory.
\newblock In {\em Hodge theory and {$L^2$}-analysis}, volume~39 of {\em Adv.
  Lect. Math. (ALM)}, pages 151--205. Int. Press, Somerville, MA, 2017.

\bibitem{bryzuck}
J.-L. Brylinski and S.~Zucker.
\newblock An overview of recent advances in {H}odge theory.
\newblock In {\em Several complex variables, {VI}}, volume~69 of {\em
  Encyclopaedia Math. Sci.} Springer, Berlin, 1990.

\bibitem{hheights}
J.I. Burgos~Gil, S.~Goswami, and Pearlstein. G.
\newblock {Height Pairing on Higher and Mixed Hodge Structures}.
\newblock \url{https://arxiv.org/abs/2007.06036} [math.AG], july 2020.

\bibitem{carlson}
J.~Carlson.
\newblock The geometry of the extension class of a mixed {H}odge structure.
\newblock In {\em Algebraic Geometry Bowdoin, 1985}, volume~46 of {\em Proc.
  Symp. Pure Math.}, pages 199--222, Providence, R.I., 1987. Amer. Math. Soc.

\bibitem{3authorsBIS}
J.~Carlson, S.~M\"uller-Stach, and C.~Peters.
\newblock {\em {Period Mappings and Period Domains, Second Edition}}, volume
  168.
\newblock Cambridge Univ. Press, Cambridge, 2017.

\bibitem{catanese}
F.~Catanese.
\newblock The moduli and the global period mapping of surfaces with
  {$K^{2}=p_{g}=1$}: a counterexample to the global {T}orelli problem.
\newblock {\em Compos. Math.}, 41:401--414, 1980.

\bibitem{ttag}
F.~Catanese.
\newblock Infinitesimal {T}orelli theorems and counterexamples to {T}orelli
  problems.
\newblock In P.~Griffiths, editor, {\em Topics in transcendental algebraic
  geometry ({P}rinceton, {N}.{J}., 1981/1982)}, volume 106 of {\em Ann. of
  Math. Stud.}, pages 143--156. Princeton Univ. Press, Princeton, N.J., 1984.

\bibitem{catdebar}
F.~Catanese and O.~Debarre.
\newblock Surfaces with {$K^2=2,\; p_g=1,\; q=0$}.
\newblock {\em J. Reine Angew. Math.}, 395:1--55, 1989.

\bibitem{CF}
E.~Cattani and J~Fernandez.
\newblock Asymptotic {H}odge theory and quantum products.
\newblock In {\em Advances in algebraic geometry motivated by physics
  ({L}owell, {MA}, 2000)}, volume 276 of {\em Contemp. Math.} Amer. Math. Soc.,
  Providence, RI, 2001.

\bibitem{luminy-hodge}
E.~Cattani and A.~Kaplan.
\newblock Degenerating variations of {H}odge structure.
\newblock In {\em Actes du Colloque de Th\'{e}orie de Hodge (Luminy, 1987)},
  number 179-180 in Ast\'{e}risque, pages 9, 67--96. Soc. Math. France, 1989.

\bibitem{degeneration}
E.~Cattani, A.~Kaplan, and W.~Schmid.
\newblock Degeneration of {H}odge structures.
\newblock {\em Annals of Math.}, 123:457--535, 1986.

\bibitem{tdh}
P.~Deligne.
\newblock Th{\'e}orie de {H}odge {II}.
\newblock {\em Publ. Math. I.H.E.S.}, 40:5--58, 1971.

\bibitem{dtock}
P.~Deligne.
\newblock About your note on holomorphy.
\newblock Unpublished letter to E. Cattani and A. Kaplan, 1993.

\bibitem{arafal}
G.~Faltings.
\newblock {A}rakelov's theorem for {A}belian varieties.
\newblock {\em Invent. Math.}, 73:337--347, 1983.

\bibitem{flenner}
H.~Flenner.
\newblock The infinitesimal {T}orelli problem for zero sets of sections of
  vector bundles.
\newblock {\em Math. Z.}, 193:307--32, 1986.

\bibitem{periods}
P.~Griffiths.
\newblock Periods of integrals on algebraic manifolds {I, II}.
\newblock {\em Amer. J. Math.}, 90:568--626; 805--865, 1968.

\bibitem{curv}
P.~Griffiths and W.~Schmid.
\newblock Locally homogenous complex manifolds.
\newblock {\em Acta Math.}, 123:145--166, 1969.

\bibitem{hainfunda}
R.~M. Hain.
\newblock The geometry of the mixed {H}odge structure on the fundamental group.
\newblock In {\em Algebraic geometry, {B}owdoin, 1985 ({B}runswick, {M}aine,
  1985)}, volume~46 of {\em Proc. Sympos. Pure Math.}, pages 247--282. Amer.
  Math. Soc., Providence, RI, 1987.

\bibitem{HZ}
R.~M. Hain and S~Zucker.
\newblock Unipotent variations of mixed {H}odge structure.
\newblock {\em Invent. Math.}, 88:83--124, 1987.

\bibitem{asdeg}
T.~Hayama and G.~Pearlstein.
\newblock Asymptotics of degenerations of mixed hodge structures.
\newblock {\em Adv. Math.}, 273:380--420, 2015.

\bibitem{imashi}
Y.~Imayoshi and H.~Shiga.
\newblock A finiteness theorem for holomorphic families of {R}iemann surfaces.
\newblock In {\em Holomorphic functions and moduli, {V}ol. {II} ({B}erkeley,
  {CA}, 1986)}, volume~11 of {\em Math. Sci. Res. Inst. Publ}. Springer, New
  York, 1988.

\bibitem{singsvarmixed}
A.~Kaplan and G.~Pearlstein.
\newblock Singularities of variations of mixed {H}odge structure.
\newblock {\em Asian J. Math.}, 7:307--336, 2003.

\bibitem{kashi}
M.~Kashiwara.
\newblock A study of variation of mixed {H}odge structure.
\newblock {\em Publ. Res. Inst. Math. Sci.}, 22:991--1024, 1986.

\bibitem{degmhs}
K.~Kato, C.~Nakayama, and S.~Usui.
\newblock $sl(2)$-orbit theorem for degeneration of mixed {H}odge structure.
\newblock {\em J. Algebraic Geom.}, 17:401--479, 2008.

\bibitem{kerr2017polarized}
Matt Kerr, Gregory Pearlstein, and Colleen Robles.
\newblock Polarized relations on horizontal sl(2)s, 2017.

\bibitem{kynev}
V.~I. Kynev.
\newblock An example of a simply connected surface of general type for which
  the local {T}orelli theorem does not hold. ({R}ussian).
\newblock {\em C. R. Acad. Bulgare Sci.}, 30:323--325, 1977.

\bibitem{psh}
P.~Lelong.
\newblock {\em Fonctions plurisousharmoniques et formes diff\'{e}rentielles
  positives}.
\newblock Gordon \& Breach, Paris-London-New York (Distributed by Dunod
  \'{e}diteur, Paris), 1968.

\bibitem{kunev}
M.~Letizia.
\newblock Intersections of a plane curve with a moving line and a generic
  global {T}orelli-type theorem for {K}unev surfaces.
\newblock {\em Amer. J. Math.}, 106:1135--1146, 1984.

\bibitem{higgs}
G.~Pearlstein.
\newblock Variations of mixed {H}odge structure, {H}iggs fields, and quantum
  cohomology.
\newblock {\em Manuscripta Math.}, 102:269--310, 2000.

\bibitem{dmj}
G.~Pearlstein.
\newblock Degenerations of mixed {H}odge structure.
\newblock {\em Duke Math. J.}, 110:1--67, 2001.

\bibitem{sl2anddeg}
G.~Pearlstein.
\newblock ${\rm sl}\sb 2$-orbits and degenerations of mixed {H}odge structure.
\newblock {\em Journal of Differential Geometry}, 74:1--67, 2006.

\bibitem{pdmixed}
G.~Pearlstein and C.~Peters.
\newblock Differential geometry of the mixed {H}odge metric.
\newblock {\em Communications in Analysis and Geometry}, 27(3):673--744, 2019.

\bibitem{hodge4}
C.~Peters.
\newblock Rigidity for variations of {H}odge structure and {A}rakelov-type
  finiteness theorems.
\newblock {\em Comp. Math.}, 75:113--126, 1990.

\bibitem{hodge10}
C.~Peters.
\newblock Rigidity, past and present.
\newblock In {\em Teichm\"{u}ller theory and moduli problem}, volume~10 of {\em
  Ramanujan Math. Soc. Lect. Notes Ser.}, pages 529--548, Mysore, 2010.
  Ramanujan Math. Soc.

\bibitem{Sa}
M.-H. Saito.
\newblock Classification of non-rigid families of abelian varieties.
\newblock {\em Tohoku Math. J.}, 45:159--189, 1993.

\bibitem{SaZ}
M.-H. Saito and S.~Zucker.
\newblock Classification of non-rigid families of {K}3-surfaces and a
  finiteness theorem of {A}rakelov type.
\newblock {\em Math. Ann.}, 289:1--31, 1991.

\bibitem{schmid}
W.~Schmid.
\newblock Variation of {H}odge structure: the singularities of the period
  mapping.
\newblock {\em Invent. Math.}, 22:211--319, 1973.

\bibitem{Sern}
E.~Sernesi.
\newblock {\em Deformations of {A}lgebraic {S}chemes}, volume 334 of {\em
  Grundl. der math. {W}issensch.}
\newblock Springer-Verlag, Berlin etc., 2006.

\bibitem{mht}
J.~Steenbrink and C.~Peters.
\newblock {\em {Mixed Hodge Theory}}, volume~52 of {\em Ergebn. der Math.}
\newblock Springer Verlag, Berlin, 2008.

\bibitem{SZ}
J.~Steenbrink and S.~Zucker.
\newblock Variation of mixed {H}odge structure {I}.
\newblock {\em Invent. Math.}, 80:489--542, 1985.

\bibitem{todorov2}
A.~N. Todorov.
\newblock Surfaces of general type with {$p_{g}=1$} and {$(K,\,K)=1$}. {I}.
\newblock {\em Ann. Sci. \'{E}cole Norm. Sup. (4)}, 13:1--21, 1980.

\bibitem{todorov}
A.~N. Todorov.
\newblock A construction of surfaces with $p_g=1, q=0$ and $2\le ({K}^2)\le 8$.
  counterexamples of the global {T}orelli theorem.
\newblock {\em Invent. Math.}, 63:287--304, 1981.

\bibitem{usui}
S.~Usui.
\newblock Variation of mixed {H}odge structure arising from family of
  logarithmic deformations {II}: classifying space.
\newblock {\em Duke Math. J.}, 51:851--875, 1983.

\bibitem{VZ}
E.~Viehweg and K.~Zuo.
\newblock Families over curves with strictly maximal {H}iggs field.
\newblock {\em Asian J. Math.}, 7:575--598, 2003.

\end{thebibliography}

\end{document}